\documentclass[11pt,a4paper]{article}
\usepackage[margin=1in]{geometry}
\usepackage[T1]{fontenc}
\usepackage{lmodern}
\usepackage{amsmath,amssymb,amsfonts,amsthm,mathrsfs,mathtools}
\usepackage{graphicx,multirow,booktabs}
\usepackage[title]{appendix}
\usepackage{xcolor,textcomp}
\usepackage{algorithm,algorithmicx,algpseudocode}

\usepackage{float,enumerate,subcaption}
\usepackage{cite}
\usepackage{microtype}
\usepackage[hidelinks]{hyperref}
\hypersetup{
  pdftitle={Hierarchical Matrices with Tucker Blocks: Construction and Forward Evaluation},
  pdfauthor={Yingzhou Li and Jingyu Liu}
}
\newtheorem{theorem}{Theorem}
\newtheorem{proposition}[theorem]{Proposition}%
\newtheorem{lemma}[theorem]{Lemma}%

\newtheorem{example}{Example}%
\newtheorem{remark}{Remark}%

\newtheorem{definition}{Definition}%

\newcommand{\transpose}{\top}
\newcommand{\herm}{*}

\newcommand{\real}{\mathbb{R}}
\newcommand{\complex}{\mathbb{C}}
\renewcommand{\natural}{\mathbb{N}}
\newcommand{\Matrix}[1]{\boldsymbol{#1}}
\newcommand{\Tensor}[1]{\boldsymbol{\mathcal{#1}}}
\newcommand{\Operator}[1]{\mathscr{#1}}
\newcommand{\Tucker}{\operatorname{Tucker}}

\newcommand{\bigO}{\mathscr{O}}
\newcommand{\imag}{\mathrm{i}}
\newcommand{\e}{\mathrm{e}}

\newcommand{\fro}{\mathrm{F}}
\newcommand{\diff}{\mathrm{d}}
\newcommand{\indexFunction}{\mathbf{1}}
\newcommand{\hierarchical}{\mathcal{H}}
\newcommand{\tree}{\mathsf{T}}
\newcommand{\children}{\operatorname{children}}
\newcommand{\diam}{\operatorname{diam}}
\newcommand{\dist}{\operatorname{dist}}
\newcommand{\inter}{\Operator{J}}
\newcommand{\HMTB}{\operatorname{HMTB}}

\newcommand{\quasi}{\mathrm{qu}}
\newcommand{\uni}{\mathrm{uni}}
\newcommand{\as}{\mathrm{as}}
\newcommand{\HConstruct}{\operatorname{Construct}}
\newcommand{\HMultV}{\operatorname{HMultVec}}
\newcommand*{\Rmnum}[1]{\uppercase\expandafter{\romannumeral #1\relax}}
\newcommand*{\construct}{\mathrm{c}}
\newcommand*{\memory}{\mathrm{m}}
\newcommand*{\apply}{\mathrm{a}}
\newcommand*{\h}{\mathrm{h}}
\newcommand*{\rand}{\mathrm{r}}

\title{Hierarchical Matrices with Tucker Blocks: Construction and Forward Evaluation}
\author{Yingzhou Li$^{1,2}$\thanks{E-mail: \href{mailto:yingzhouli@fudan.edu.cn}{\nolinkurl{yingzhouli@fudan.edu.cn}}.}
\and Jingyu Liu$^1$\thanks{Corresponding author. E-mail: \href{mailto:jyliu22@m.fudan.edu.cn}{\nolinkurl{jyliu22@m.fudan.edu.cn}}.}}
\date{}

\begin{document}
\maketitle
\begin{center}
\small
$^1$School of Mathematical Sciences, Fudan University\\
$^2$Shanghai Key Laboratory for Contemporary Applied Mathematics,\\
Fudan University\\
220 Handan Road, Shanghai 200433, China

\medskip
Accepted for publication in \emph{Linear Algebra and its Applications}.
\end{center}

\begin{abstract}
In this paper, we introduce hierarchical matrices with Tucker blocks~(HMTB) for data-sparse approximation of matrices arising from nonoscillatory and mildly oscillatory kernels.
Building on the hierarchical matrix framework, HMTB matrices replace conventional low-rank blocks with Tucker representations.
Using high-dimensional interpolation and tensor compression and contractions, we develop
algorithms for constructing HMTB matrices and performing matrix-vector multiplication.
For an \(N \times N\) matrix, these algorithms have linear and quasi-linear complexities in \(N\), respectively.
Numerical comparisons with an interpolation-based
\(\hierarchical\)-matrix implementation demonstrate storage savings and,
in several cases, reduced construction times, while maintaining
competitive forward-evaluation costs.
The framework is also applicable to quasi-uniform grids.
\end{abstract}


\section{Introduction} \label{sec:introduction}
This paper considers the integral equation~(IE)
\begin{equation} \label{eq:IE}
    a(\Matrix{x}) u(\Matrix{x}) +
    \int_{\Omega} k(\Matrix{x}, \Matrix{y}) u(\Matrix{y}) \diff \Matrix{y}
    = f(\Matrix{x}), \quad \Matrix{x} \in \Omega,
\end{equation}
where \(\Omega = [0,1]^d \subset \real^d\).
The kernel \(k(\Matrix{x},\Matrix{y})\) and the function \(a(\Matrix{x})\) are given.
We consider kernels arising from Green's functions associated with
partial differential equations~(PDEs), such as the Laplace and
Helmholtz equations, as well as Gaussian kernels.
Given \(u(\Matrix{x})\), computing \(f(\Matrix{x})\) through~\eqref{eq:IE} is called forward evaluation; given \(f(\Matrix{x})\), solving~\eqref{eq:IE} for \(u(\Matrix{x})\) is called inversion.
Both problems are important in practice.
Although we focus on rectangular domains, the proposed approach also applies to simply-structured non-tensor domains that can be partitioned into tensor subdomains, such as L-shaped domains.

Discretizing~\eqref{eq:IE} by standard methods yields a dense linear system
\begin{equation} \label{eq:Au=f}
    \Matrix{A} \Matrix{u} = \Matrix{f},
\end{equation}
where \(\Matrix{A} \in \complex^{N \times N}\) and \(N\) is the number of discretization points.
For example, applying the Nystr{\"o}m method on a uniform grid with \(n\) points in each direction yields a system of size \(N = n^{d}\):
\begin{equation} \label{eq:Nystrom_IE_uni}
    a(\Matrix{x}_{i}) u_{i} + \sum_{j = 1}^{N} K_{i, j} h^{d} u_{j}
    = f(\Matrix{x}_{i}), \quad i = 1, \dotsc, N,
\end{equation}
where \(u_{i} \approx u(\Matrix{x}_{i})\).
The entries of \(\Matrix{K}\) are given by
\begin{equation} \label{eq:K_uni}
    K_{i, j} \coloneq
    \begin{cases}
        k(\Matrix{x}_{i}, \Matrix{x}_{j}), & j \neq i, \\
        \displaystyle\frac{1}{h^{d}}
        \int_{\Omega_{i}} k(\Matrix{x}_{i}, \Matrix{y}) \diff \Matrix{y},
        & j = i,
    \end{cases}
\end{equation}
where \(\Omega_{j}\) is the box of side length \(h = 1/n\) centered at \(\Matrix{x}_{j}\).
Thus, \(K_{i,j}h^{d}\) approximates the integral of \(k(\Matrix{x}_{i},\cdot)\) over \(\Omega_{j}\), with the diagonal integral evaluated separately by numerical quadrature.
Although other discretization schemes can be used, this simple Nystr{\"o}m method suffices to demonstrate our approach.
It can also be viewed as a collocation method with piecewise-constant basis functions supported on \(\{\Omega_{i}\}\), where the off-diagonal integrals are approximated by point evaluations.

\subsection{Related Work} \label{subsec:related_work}

Dense direct solvers for~\eqref{eq:Au=f}, such as LU factorization, require \(\bigO(N^{3})\) time and \(\bigO(N^{2})\) storage and thus become prohibitively expensive for large \(N\).

Fast algorithms exploit the fact that off-diagonal blocks of \(\Matrix{A}\) in~\eqref{eq:Nystrom_IE_uni} are numerically low-rank when the kernel \(k\) is smooth away from the diagonal and not highly oscillatory.
This property enables efficient storage and matrix operations.
Early examples include the \textit{Barnes--Hut algorithm}~\cite{Barnes_Hut_1986}, also known as the tree code, and the \textit{fast multipole method}~(FMM)~\cite{adaptive_FMM, black_box_FMM, classical_FMM, kernel_independent_FMM}, which accelerate matrix-vector multiplication for kernel matrices.
Hackbusch and collaborators subsequently introduced \textit{hierarchical matrices}~\cite{Hmatrix_partI, Hmatrix_partII}, including \(\hierarchical\)- and \(\hierarchical^2\)-matrices.
They support many matrix operations, such as matrix-vector multiplication, matrix-matrix multiplication, and LU factorization, with quasi-linear or linear complexity~\cite{intro_Hmatrix_H2matrix, Hmatrix_partI, Hackbusch_2015_book, H2matrix, Hmatrix_partII, DHMAT}.
The \(\hierarchical\)- and \(\hierarchical^2\)-matrix frameworks can be viewed as algebraic counterparts of the Barnes--Hut algorithm and FMM, respectively.

\textit{Hierarchical block-separable}~(HBS) matrices, also known as \textit{hierarchically semiseparable}~(HSS) matrices~\cite{Martinsson_Rokhlin_2005}, form another class of data-sparse matrices for accelerating kernel matrix operations.
They are closely related to \(\hierarchical^2\)-matrices under weak admissibility.
Various HBS factorization algorithms achieve quasi-linear or linear complexity for matrix-vector multiplication and linear system solution~\cite{ULV_HSS, Corona_Martinsson_Zorin_2015, Gillman_Young_Martinsson_2012, Greengard_Gueyffier_Martinsson_Rokhlin_2009, fast_HSS}.
Another class constructs structured block-sparse factorizations of
\(\Matrix{A}\) that can be applied or inverted efficiently.
Examples include \textit{recursive skeletonization factorization}~(RS)~\cite{RS, RS_strong_admissibility} and \textit{hierarchical interpolative factorization}~(HIF)~\cite{HIFIE}.
These algorithms rely primarily on \textit{interpolative decomposition}~\cite{ID} and \textit{proxy surfaces}~\cite{Martinsson_2019_book}.
For smooth, nonsingular kernels, such as Gaussian kernels, low-rank approximation can be applied more aggressively, even to diagonal blocks~\cite{BBF_theory, BBF}.

Although these fast algorithms have low complexity in the matrix size, their costs may grow rapidly, often exponentially, with the dimension \(d\), leading to the \textit{curse of dimensionality}~(CoD).
Tensor low-rank decompositions are commonly used to mitigate this issue.
Kolda and Bader~\cite{Kolda_Bader_2009} provide a comprehensive introduction to the \textit{CANDECOMP/PARAFAC}~(CP) and \textit{Tucker} decompositions, which generalize the matrix singular value decomposition~(SVD) to tensors.
These decompositions can be computed using algorithms such as the \textit{higher-order SVD}~(HOSVD)~\cite{HOSVD}, \textit{higher-order orthogonal iteration}~(HOOI)~\cite{HOOI}, \textit{sequentially truncated HOSVD}~(STHOSVD)~\cite{STHOSVD}, and their randomized variants~\cite{RTSMS, RTucker}.
In addition, \textit{tensor train}, \textit{tensor ring}, and, more generally, \textit{tensor-network decompositions} are widely used in practice, particularly in computational physics and chemistry~\cite{TN1, TT}.

Combining hierarchical and tensor decompositions can reduce both asymptotic complexity and prefactors, potentially yielding linear or quasi-linear algorithms.
The \textit{hierarchical Kronecker tensor-product}~(HKT) approximation~\cite{Gavrilyuk_Wolfgang_Khoromskij_2005, Hackbusch_Khoromskij_2006a, Hackbusch_Khoromskij_2006b, Hackbusch_Khoromskij_2008, Hackbusch_Khoromskij_Tyrtyshnikov_2005} combines these techniques to provide quasi-linear representations of certain high-dimensional integral and elliptic operators.
In~\cite{Corona_Rahimian_Zorin_2017}, the quantized tensor train~(QTT) decomposition was used to approximate integral operators, and matrix-vector multiplication and inversion algorithms were developed for the resulting representation.
Ehrlacher, Grigori, Lombardi, and Song~\cite{Ehrlacher_Grigori_Lombardi_Song_2021} constructed a hierarchical partition to identify numerically low-rank subtensors, which were then compressed using standard tensor decomposition algorithms such as HOSVD~\cite{HOSVD}.
Their problem differs slightly from ours, as they primarily consider function approximation rather than operator approximation.
Other related work can be found in~\cite{Qian_Yucel2021}.
More recently, the \textit{tensor butterfly algorithm}~\cite{TensorBF} combines \textit{butterfly factorization}~\cite{IBF, BF, MBF} with Tucker-type tensor compression for high-dimensional oscillatory integral operators.
It extends the complementary low-rank property of conventional butterfly algorithms to a tensor setting, partitions the discretized operator into multilevel subtensors, and compresses the middle-level interactions using Tucker-type interpolative decompositions.
By exploiting this tensor structure across levels, the method achieves \(\bigO(n^{d})\) computational and storage complexities for a \(d\)-dimensional operator discretized as a \(2d\)-mode tensor with \(n^{2d}\) entries.

\subsection{Contributions} \label{subsec:contributions}

In this paper, we define \textit{hierarchical matrices with Tucker blocks}~(HMTB) by replacing the low-rank blocks of an \(\hierarchical\)-matrix with \textit{Tucker low-rank}~(TLR) representations.
On tensor grids, the Tucker format mitigates the curse of dimensionality~(CoD), reducing storage and computational costs compared with conventional low-rank representations.
We show that HMTB matrices require storage that is linear in the number \(N\) of discretization points.
We also develop construction and matrix-vector multiplication algorithms with linear and quasi-linear complexities in \(N\), respectively.
The construction uses multidimensional interpolation, with an error bound established under suitable assumptions on the kernel.
For representations generated by the same tensor-product interpolation, the matrix-vector multiplication algorithm uses tensor contractions and has a smaller leading prefactor than its \(\hierarchical\)-matrix counterpart.
The approach is also applicable to quasi-uniform grids.
Numerical experiments support the complexity analysis and demonstrate the efficiency of HMTB matrices.

Although developed here within the \(\hierarchical\)-matrix framework, the underlying idea can also be incorporated into the kernel-independent FMM~\cite{kernel_independent_FMM}, \(\hierarchical^2\)-matrices~\cite{H2matrix}, HBS matrices~\cite{Martinsson_2019_book}, HIF~\cite{HIFIE}, RS~\cite{RS}, and HODLR matrices~\cite{HODLR}, potentially yielding similar memory and runtime savings.
A similar idea was independently incorporated into butterfly factorization for oscillatory kernels in~\cite{TensorBF}.
Algorithms for solving linear systems with HMTB matrices, such as LU factorization, may also be developed analogously to those for \(\hierarchical\)-matrices, but are considerably more involved and left for future work.

\subsection{Organization} \label{subsec:organization}

The remainder of this paper is organized as follows.
Section~\ref{sec:preliminaries} introduces the notation and briefly reviews \(\hierarchical\)-matrices.
Section~\ref{sec:hierarchical_matrices_with_tucker_blocks} introduces HMTB matrices on uniform grids, establishes their linear storage complexity, and discusses their application to quasi-uniform grids.
Section~\ref{sec:HMTB_construction_matrix_vector_multiplication} presents the construction and forward evaluation algorithms for HMTB matrices, together with their complexity analysis.
Section~\ref{sec:numerical_results} presents numerical experiments in two and three dimensions to demonstrate the efficiency of HMTB matrices.
Finally, Section~\ref{sec:conclusion_future_work} concludes the paper and discusses future work.

\section{Preliminaries} \label{sec:preliminaries}

\subsection{Notation} \label{subsec:notations}

For a positive integer \(N\), we denote the index set \(\{1, \dotsc, N\}\) by \([N]\).
The notation \(|\cdot|\) denotes either the cardinality of a set or the measure of a domain.
Vectors are denoted by boldface lowercase letters, e.g., \(\Matrix{a}\), matrices by boldface uppercase letters, e.g., \(\Matrix{A}\), and tensors of order three or higher by boldface Euler script letters, e.g., \(\Tensor{A}\).
We use MATLAB notation throughout the paper.
For example, the \(i\)-th entry of a vector \(\Matrix{a}\) is denoted by \(\Matrix{a}_{i}\) or \(\Matrix{a}(i)\).
The submatrix of \(\Matrix{A}\) corresponding to the row and column index sets \(\tau\) and \(\sigma\) is denoted by \(\Matrix{A}_{\tau,\sigma}\) or \(\Matrix{A}(\tau,\sigma)\).
A colon indicates all indices along a dimension; for example, the \(j\)-th column of \(\Matrix{A}\) is denoted by \(\Matrix{A}_{:,j}\) or \(\Matrix{A}(:,j)\).
The conjugate transpose of \(\Matrix{A}\) is denoted by \(\Matrix{A}^{\herm}\).
A matrix \(\Matrix{U}\) is \textit{orthonormal} if its columns are orthonormal, i.e., \(\Matrix{U}^{\herm}\Matrix{U}=\Matrix{I}\).

We use multi-indices to describe tensor structures.
Consider the two-dimensional case as an example.
A tensor grid in two dimensions is a set of discretization points of the form
\(\{(x_{1;i_{1}},x_{2;i_{2}})\}_{i_{1},i_{2}=1}^{n}\),
that is, the Cartesian product of the discretization points in the two coordinate directions.
The subscripts \(1\) and \(2\) indicate that the discretization points may differ between the two directions.
We use the multi-index \(\Matrix{i}=(i_{1},i_{2})\) to index each grid point and write
\(\Matrix{x}_{\Matrix{i}}=(x_{1;i_{1}},x_{2;i_{2}})\).
Thus, the grid can also be denoted by
\(\{\Matrix{x}_{\Matrix{i}}\}_{\Matrix{i}\in[n]^{2}}\).
To arrange the multi-indices as row or column indices of a matrix, we use the one-to-one correspondence
\begin{equation*}
    \Matrix{i}=(i_{1},i_{2})
    \longleftrightarrow
    i=i_{1}+(i_{2}-1)n.
\end{equation*}
We further define
\(|\Matrix{i}|\coloneq i_{1}+i_{2}\) and
\(\Matrix{i}!\coloneq i_{1}!i_{2}!\).
These notations extend analogously to the \(d\)-dimensional case.

Our tensor notation and terminology are consistent with~\cite{Kolda_Bader_2009}.
The mode-\(\ell\) product of a tensor
\(\Tensor{A} \in \complex^{n_{1} \times \dotsb \times n_{d}}\)
and a matrix \(\Matrix{U} \in \complex^{m \times n_{\ell}}\) is denoted by
\(\Tensor{A} \times_{\ell} \Matrix{U}
\in
\complex^{n_{1} \times \dotsb \times n_{\ell-1} \times m \times n_{\ell+1} \times \dotsb \times n_{d}}\).
We denote the contraction of two tensors \(\Tensor{A}\) and \(\Tensor{B}\) by
\(\Tensor{A} \times_{\Matrix{a},\Matrix{b}} \Tensor{B}\),
where \(\Matrix{a}\) and \(\Matrix{b}\) specify the corresponding modes of \(\Tensor{A}\) and \(\Tensor{B}\) to be contracted.
For example, let \(\Tensor{A} \in \complex^{3 \times 4 \times 5}\) and
\(\Tensor{B} \in \complex^{4 \times 5 \times 7}\).
Then
\(\Matrix{C} \coloneq \Tensor{A} \times_{[2,3],[1,2]} \Tensor{B}
\in \complex^{3 \times 7}\),
where
\(\Matrix{C}(i,j)
=
\sum_{\alpha=1}^{4}
\sum_{\beta=1}^{5}
\Tensor{A}(i,\alpha,\beta)
\Tensor{B}(\alpha,\beta,j)\).

A \textit{Tucker decomposition} of a tensor \(\Tensor{A}\) with \textit{core tensor}
\(\Tensor{G} \in \complex^{p_{1} \times \dotsb \times p_{d}}\)
and \textit{factor matrices}
\(\{\Matrix{U}_{\ell} \in \complex^{n_{\ell} \times p_{\ell}}\}_{\ell=1}^{d}\)
is given by
\begin{equation} \label{eq:Tucker}
    \Tensor{A}
    =
    \Tucker(\Tensor{G},\{\Matrix{U}_{\ell}\}_{\ell=1}^{d})
    \coloneq
    \Tensor{G}
    \times_{1} \Matrix{U}_{1}
    \times_{2} \dotsb
    \times_{d} \Matrix{U}_{d}.
\end{equation}
The entries of \(\Matrix{p}=(p_1,\dotsc,p_d)\) are called the
\textit{Tucker bond dimensions} of this representation.
They specify the numbers of columns of the factor matrices and
the corresponding mode sizes of the core tensor, and need not
be minimal.
When \(p_1=\dotsb=p_d=p\), we call \(p\) the
\textit{uniform Tucker bond dimension}.

A Tucker decomposition can be compressed without explicitly forming the full tensor.
Given a decomposition~\eqref{eq:Tucker}, we first compute QR factorizations of the factor matrices and absorb the resulting triangular factors into the core tensor.
We then compute a compressed Tucker decomposition of the updated core tensor and combine its factor matrices with the orthonormal factors obtained from the QR factorizations.
For a prescribed relative tolerance \(\varepsilon\), the resulting decomposition
\(\widetilde{\Tensor{A}}\) satisfies
\(\|\Tensor{A}-\widetilde{\Tensor{A}}\|_{\fro}
= \bigO(\varepsilon)\|\Tensor{A}\|_{\fro}\)
and has factor matrices with orthonormal columns.
Assuming that \(p_{\ell}=p\) and \(n_{\ell}=n\) for all \(\ell\), and that a standard HOSVD- or STHOSVD-type algorithm is used to compress the updated core tensor, the total compression cost is
\(\bigO(dnp^{2}+dp^{d+1})\).
The procedure is summarized in Algorithm~\ref{alg:Tucker_compression}.

\begin{algorithm}
\caption{\(\operatorname{TuckerCompress}
(\Tensor{G},\{\Matrix{U}_{\ell}\}_{\ell=1}^{d},\varepsilon)\)}
\label{alg:Tucker_compression}
\begin{algorithmic}[1]
    \Require{A Tucker decomposition
    \(\Tensor{A}
    =
    \Tucker(\Tensor{G},\{\Matrix{U}_{\ell}\}_{\ell=1}^{d})\)
    and a relative tolerance \(\varepsilon\)}
    \Ensure{A compressed Tucker decomposition
    \(\widetilde{\Tensor{A}}
    =
    \Tucker(\widetilde{\Tensor{G}},
    \{\widetilde{\Matrix{U}}_{\ell}\}_{\ell=1}^{d})\)
    satisfying
    \(\|\Tensor{A}-\widetilde{\Tensor{A}}\|_{\fro}
    = \bigO(\varepsilon)\|\Tensor{A}\|_{\fro}\)}

    \State{\(\widehat{\Tensor{G}}\gets\Tensor{G}\)}
    \For{\(\ell=1,\dotsc,d\)}
        \State{Compute the QR factorization
        \(\Matrix{U}_{\ell}
        =
        \Matrix{Q}_{\ell}\Matrix{R}_{\ell}\)}
        \State{\(\widehat{\Tensor{G}}
        \gets
        \widehat{\Tensor{G}}
        \times_{\ell}\Matrix{R}_{\ell}\)}
    \EndFor
    \State{Compute a compressed Tucker decomposition
    \(\Tucker(\widetilde{\Tensor{G}},
    \{\Matrix{W}_{\ell}\}_{\ell=1}^{d})\)
    of \(\widehat{\Tensor{G}}\) with relative tolerance \(\varepsilon\)}
    \For{\(\ell=1,\dotsc,d\)}
        \State{\(\widetilde{\Matrix{U}}_{\ell}
        \gets
        \Matrix{Q}_{\ell}\Matrix{W}_{\ell}\)}
    \EndFor
    \State{\Return{
    \(\widetilde{\Tensor{G}},
    \{\widetilde{\Matrix{U}}_{\ell}\}_{\ell=1}^{d}\)}}
\end{algorithmic}
\end{algorithm}

Order-\(2d\) tensors are often used to represent interaction matrices between two tensor grids.
For example, let
\(\{\Matrix{x}_{\Matrix{i}}=(x_{1;i_{1}},x_{2;i_{2}})\}_{\Matrix{i}\in[n]^{2}}\)
and
\(\{\Matrix{y}_{\Matrix{j}}=(y_{1;j_{1}},y_{2;j_{2}})\}_{\Matrix{j}\in[n]^{2}}\)
be two tensor grids, and let \(k\) be a kernel function.
The corresponding interaction matrix
\(\Matrix{A}\in\complex^{n^{2}\times n^{2}}\)
is defined by
\(\Matrix{A}(\Matrix{i},\Matrix{j})
=
k(\Matrix{x}_{\Matrix{i}},\Matrix{y}_{\Matrix{j}})\).
It is convenient to view \(\Matrix{A}\) as an order-\(4\) tensor
\(\Tensor{A}\in\complex^{n\times n\times n\times n}\), with
\(\Tensor{A}(i_{1},i_{2},j_{1},j_{2})
=
\Matrix{A}(\Matrix{i},\Matrix{j})\).
For such tensors, we use the notation
\begin{equation*}
    \Tucker\bigl(
        \Tensor{G},
        \{\Matrix{U}_{\ell}\}_{\ell=1}^{d},
        \{\Matrix{V}_{\ell}\}_{\ell=1}^{d}
    \bigr)
    \coloneq
    \Tensor{G}
    \times_{1} \Matrix{U}_{1}
    \times_{2} \dotsb
    \times_{d} \Matrix{U}_{d}
    \times_{d+1} \Matrix{V}_{1}
    \times_{d+2} \dotsb
    \times_{2d} \Matrix{V}_{d}
\end{equation*}
for their Tucker decompositions.

\subsection{Hierarchical Matrices} \label{subsec:hierarchical_matrices}

The definition of \(\hierarchical\)-matrices is based on the concept of a cluster tree.

\begin{definition}[Cluster tree]
  Let \(I\) be the index set associated with the grid points arising from the discretization of IE~\eqref{eq:IE}.
  A tree \(\tree_{I}\) is called a \textit{cluster tree} corresponding to \(I\) if the following conditions hold:
  \begin{enumerate}
    \item Every node in \(\tree_{I}\) is a subset of \(I\).
    \item The root of \(\tree_{I}\) is \(I\).
    \item Every node \(\tau \in \tree_{I}\) is nonempty.
    \item For every \(\tau \in \tree_{I}\), let \(\children(\tau)\) denote the set of its children.
    Then either \(\children(\tau)=\emptyset\), in which case \(\tau\) is a leaf node, or \(\children(\tau)\) forms a partition of \(\tau\), i.e.,
    \(\tau=\sqcup_{\alpha \in\children(\tau)} \alpha\),
    where \(\sqcup\) denotes the union of pairwise disjoint sets.
  \end{enumerate}
\end{definition}

The cluster tree corresponds to a hierarchical partition of the set of points associated with the matrix.
Typically, each point \(\Matrix{x}_{i}\) is associated with a computational domain \(\Omega_{i}\), and these domains satisfy
\(\lvert\Omega_{i}\cap\Omega_{j}\rvert=0\) for \(i\neq j\).
In this sense, we say that the computational domains are non-overlapping.
The computational domains may be chosen as the Voronoi cells~\cite{Voronoi} associated with the grid points.
In particular, for a uniform grid, each \(\Omega_{i}\) can be chosen as a box centered at \(\Matrix{x}_{i}\).
The cluster tree can then be constructed by recursively partitioning the computational domain uniformly into \(2^{d}\) subdomains.
For an index set \(\tau\) associated with the points
\(\{\Matrix{x}_{i}\}_{i\in\tau}\), we define
\(\Omega_{\tau}\coloneq\bigcup_{i\in\tau}\Omega_{i}\).

In the construction of a cluster tree, a stopping criterion is typically imposed to prevent excessive partitioning by prescribing a maximum leaf size \(N_{0}\).
Specifically, if \(\lvert\tau\rvert\leq N_{0}\), then \(\tau\) is designated as a leaf node and is not partitioned further.

\begin{example}[A cluster tree on a 2D uniform grid]\label{example:cluster_tree_2D}
  Suppose that \(\Omega=[0,1]^2\) and that \(I\) is associated with the uniform grid points
  \begin{equation*}
    \Matrix{x}_{\Matrix{i}}
    =
    \biggl(
      (i_{1}-1)h+\frac{h}{2},
      (i_{2}-1)h+\frac{h}{2}
    \biggr),
    \quad
    \Matrix{i}=(i_{1},i_{2}),
    \quad
    1\leq i_{1},i_{2}\leq n,
  \end{equation*}
  where \(n\in\natural_{+}\) and \(h=1/n\).
  The computational domain associated with each grid point is defined by
  \(\Omega_{\Matrix{i}}
  =
  [(i_{1}-1)h,i_{1}h]\times[(i_{2}-1)h,i_{2}h]\).
  By recursively partitioning \(\Omega\) uniformly into four subdomains, we obtain a quadtree cluster tree \(\tree_{I}\), whose leaf nodes are shown in Figure~\ref{fig:leaf_node_cluster_tree_2D}.
\end{example}
\begin{figure}[htbp]
    \centering
    \includegraphics[width=.4\textwidth]{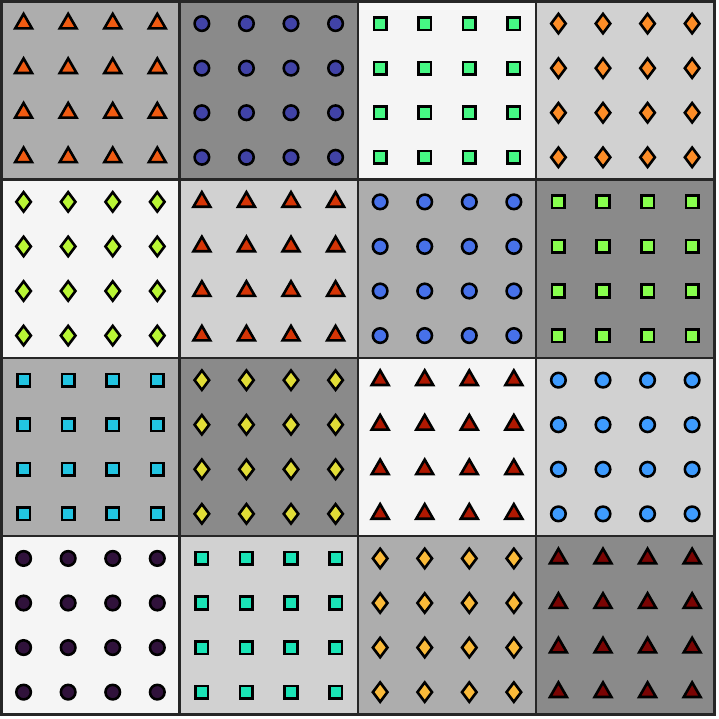}
    \caption{Leaf nodes in the cluster tree in Example~\ref{example:cluster_tree_2D}.
    Here \(n = 16\), \(N = n^{2} = 256\) and \(N_0 = 16\).
    Different shapes or colors represent different leaf nodes.}
    \label{fig:leaf_node_cluster_tree_2D}
\end{figure}

The structure of \(\hierarchical\)-matrices is determined by an admissibility condition, which specifies whether the interaction matrix between two domains can be approximated by a low-rank matrix.
Two types of admissibility conditions are commonly used: weak admissibility and strong admissibility.
Figure~\ref{fig:weak_strong_admissibility} illustrates these two admissibility conditions.

\begin{definition}[Weak admissibility] \label{definition:weak_admissibility}
    Two nodes \(\tau\) and \(\sigma\) in \(\tree_{I}\) are weakly admissible if \(\Omega_{\tau}\) and \(\Omega_{\sigma}\) are non-overlapping.
\end{definition}

\begin{definition}[Strong admissibility]
    \label{definition:strong_admissibility}
    Two nodes \(\tau\) and \(\sigma\) in \(\tree_{I}\) are strongly admissible if
    \begin{equation*}
        \dist(\Omega_{\tau},\Omega_{\sigma})
        \geq
        \eta
        \max\bigl\{
            \diam(\Omega_{\tau}),
            \diam(\Omega_{\sigma})
        \bigr\},
    \end{equation*}
    where
    \(\eta>0\) is a prescribed admissibility parameter,
    \begin{equation*}
        \diam(\Omega_{\tau})
        \coloneq
        \max_{\Matrix{x},\Matrix{y}\in\Omega_{\tau}}
        \lVert\Matrix{x}-\Matrix{y}\rVert
    \end{equation*}
    is the diameter of \(\Omega_{\tau}\), and
    \begin{equation*}
        \dist(\Omega_{\tau},\Omega_{\sigma})
        \coloneq
        \min_{\substack{
            \Matrix{x}\in\Omega_{\tau}\\
            \Matrix{y}\in\Omega_{\sigma}
        }}
        \lVert\Matrix{x}-\Matrix{y}\rVert
    \end{equation*}
    is the distance between \(\Omega_{\tau}\) and \(\Omega_{\sigma}\).
    In this case, \(\Omega_{\tau}\) and \(\Omega_{\sigma}\) are said to be well-separated.
\end{definition}

\begin{figure}[tbhp]
    \centering
    \begin{minipage}[c]{0.4\textwidth}
        \centering
        \includegraphics[width=.8\textwidth]{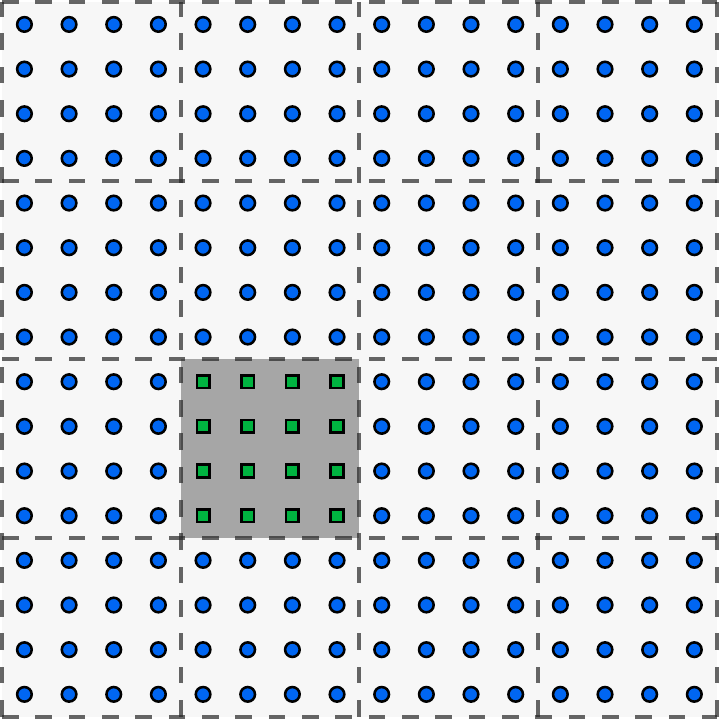}
    \end{minipage}
    \begin{minipage}[c]{0.4\textwidth}
        \centering
        \includegraphics[width=.8\textwidth]{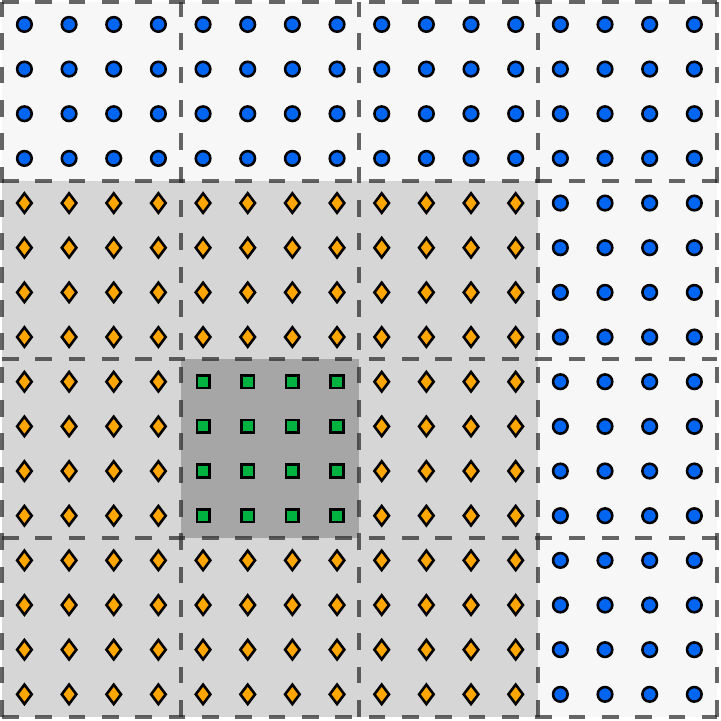}
    \end{minipage}
    \caption{Weak admissibility~(left) and strong admissibility with
    parameter \(\eta = 1 / \sqrt{2}\)~(right) in
    Example~\ref{example:cluster_tree_2D}.
    The grid points we consider are colored by green and marked by squares.
    The admissible grid points with respect to them are colored by blue and marked by circles, while the inadmissible ones are colored by yellow and marked by diamonds.
    The dashed domain is the computational domain \(\Omega_{\tau}\) for each
    index set \(\tau\).}
    \label{fig:weak_strong_admissibility}
\end{figure}

Given a cluster tree \(\tree_{I}\) and an admissibility condition, the block cluster tree defined below specifies which blocks admit low-rank representations, while its leaves form a partition of \(I\times I\).
A node of the block cluster tree represents a block
\(\tau\times\sigma\) and is denoted by \((\tau,\sigma)\),
where \(\tau,\sigma\in\tree_I\).
It is called admissible if the pair \((\tau,\sigma)\) satisfies
the chosen admissibility condition, and inadmissible otherwise.
A block cluster tree may contain both admissible and inadmissible nodes.
Moreover, it can be directly verified that a block cluster tree is itself a cluster tree.
An \(\hierarchical\)-matrix is then defined based on the block cluster tree by approximating the submatrices corresponding to admissible leaves by low-rank matrices.
The structures of the \(\hierarchical\)-matrices corresponding to weak and strong admissibility with
\(\eta=1/\sqrt{2}\) for Example~\ref{example:cluster_tree_2D} are shown in
Figure~\ref{fig:weak_strong_Hmatrix}.

\begin{definition}[Block cluster tree]
  \label{definition:block_cluster_tree}
  Let \(\tree_{I}\) be a cluster tree associated with the index set \(I\).
  The \textit{block cluster tree} \(\tree_{I\times I}\), associated with \(\tree_{I}\) and an admissibility condition, is a tree whose nodes are subsets of \(I\times I\).
  Starting from the root node \((\tau,\sigma)=(I,I)\), it is constructed recursively as follows:
  \begin{itemize}
    \item If the pair \((\tau,\sigma)\) is admissible, then
    \((\tau,\sigma)\) is an admissible leaf node.
    \item If the pair \((\tau,\sigma)\) is inadmissible and at least
    one of the clusters \(\tau\) and \(\sigma\) has no children,
    then \((\tau,\sigma)\) is an inadmissible leaf node.
    \item Otherwise, \((\tau,\sigma)\) is a non-leaf node whose children are
    \begin{equation*}
      \children\bigl((\tau,\sigma)\bigr)
      =
      \children(\tau)\times\children(\sigma).
    \end{equation*}
  \end{itemize}
\end{definition}

\begin{definition}[Hierarchical matrix,~\(\hierarchical\)-matrix,~\cite{Bebendorf_2008_book}]
    \label{definition:Hmatrix}
    Let \(I\) be the index set associated with the grid points arising from the
    discretization of IE~\eqref{eq:IE}, and let \(\tree_{I\times I}\) be a block
    cluster tree.
    An \(\hierarchical\)-matrix of rank \(r\), denoted by
    \(\Matrix{A}^{\hierarchical}\), is a matrix defined on
    \(\tree_{I\times I}\) such that, for every node
    \((\tau,\sigma)\in\tree_{I\times I}\), the following conditions hold:
    \begin{itemize}
        \item If \((\tau,\sigma)\) is an admissible leaf node, then
        \(\Matrix{A}^{\hierarchical}_{\tau,\sigma}\) is a matrix of rank at
        most \(r\), represented as
        \begin{equation*}
            \Matrix{A}^{\hierarchical}_{\tau,\sigma}
            =
            \Matrix{U}_{\tau,\sigma}
            \Matrix{G}_{\tau,\sigma}
            \Matrix{V}_{\tau,\sigma}^{\herm},
        \end{equation*}
        where
        \(\Matrix{U}_{\tau,\sigma}\in
        \complex^{\lvert\tau\rvert\times r}\),
        \(\Matrix{V}_{\tau,\sigma}\in
        \complex^{\lvert\sigma\rvert\times r}\)
        are orthonormal matrices, and
        \(\Matrix{G}_{\tau,\sigma}\in\complex^{r\times r}\).
        \item If \((\tau,\sigma)\) is an inadmissible leaf node, then
        \(\Matrix{A}^{\hierarchical}_{\tau,\sigma}
        \in\complex^{\lvert\tau\rvert\times\lvert\sigma\rvert}\)
        is stored as a dense matrix.
        \item If \((\tau,\sigma)\) is a non-leaf node, then
        \(\Matrix{A}^{\hierarchical}_{\tau,\sigma}\) is a block matrix whose
        blocks are
        \(\Matrix{A}^{\hierarchical}_{\alpha, \beta}\), where
        \((\alpha,\beta)
        \in\children\bigl((\tau,\sigma)\bigr)\).
    \end{itemize}
\end{definition}

\begin{figure}[tbhp]
    \centering
    \begin{minipage}[c]{0.4\textwidth}
        \centering
        \includegraphics[width=0.8\textwidth]{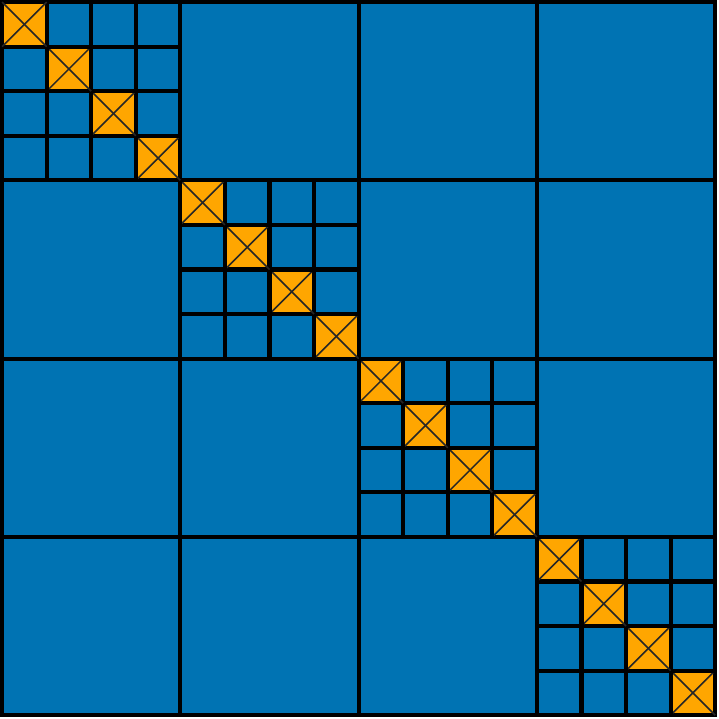}
    \end{minipage}
    \begin{minipage}[c]{0.4\textwidth}
        \centering
        \includegraphics[width=0.8\textwidth]{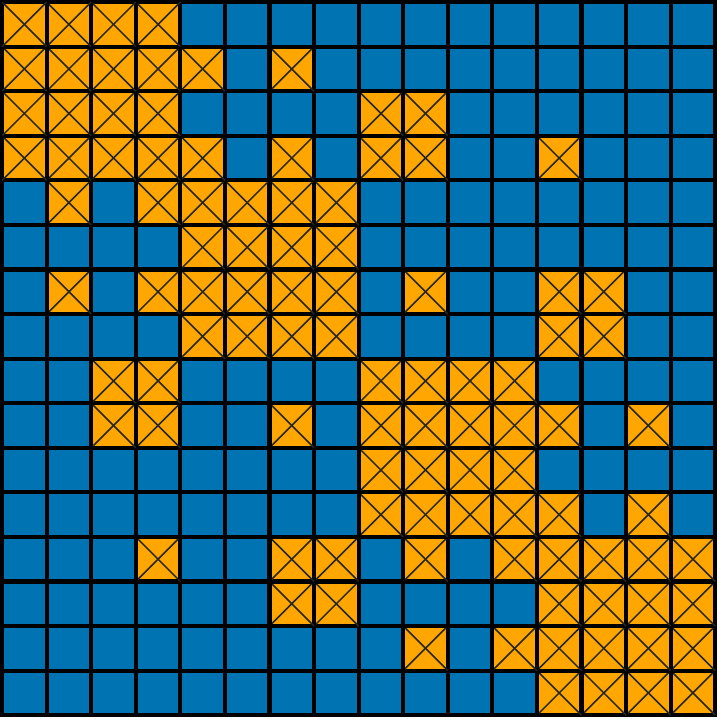}
    \end{minipage}
    \caption{\(\hierarchical\)-matrices corresponding to
    Example~\ref{example:cluster_tree_2D} under weak admissibility~(left)
    and strong admissibility with parameter \(\eta = 1 / \sqrt{2}\)~(right)
    where \(n = 16\), \(N = n^{2} = 256\) and \(N_0 = 4^2\).
    The yellow and crossed submatrices are dense while the blue ones are low-rank.
    The weak admissibility implies that there is only one dense submatrix (the
    diagonal block) in each row block.
    On the other hand, under strong admissibility, each row block has at most \(9\) dense submatrices.}
    \label{fig:weak_strong_Hmatrix}
\end{figure}

When the kernel function is singular only on the diagonal, smooth everywhere else, and not highly oscillatory, as is the case for many commonly used kernels, including the Laplace and low-frequency Helmholtz kernels, the main distinction between weak and strong admissibility lies in whether the numerical rank of the interaction matrix associated with a pair of admissible nodes can be bounded independently of the matrix size.
\(\hierarchical\)-matrices based on weak admissibility are sometimes referred to as \textit{hierarchical off-diagonal low-rank}~(HODLR) matrices~\cite{HODLR}, since all off-diagonal blocks are approximated by low-rank matrices.
However, the numerical ranks of these blocks may increase mildly with the matrix size~\cite{Greengard_Gueyffier_Martinsson_Rokhlin_2009, RS}.
In contrast, under strong admissibility, the numerical ranks can be bounded independently of the matrix size for a prescribed accuracy, at the cost of a more complex \(\hierarchical\)-matrix structure~\cite{intro_Hmatrix_H2matrix, Hmatrix_partII}.

Compared with the \(\bigO(N^{2})\) storage required for a dense matrix, an \(\hierarchical\)-matrix of rank \(r\) has a significantly lower storage cost of \(\bigO(rN\log N)\)~\cite{construction_arithmetics_Hmatrix, Hmatrix_partI}.
However, as we will demonstrate in Section~\ref{subsec:tucker_decomposition_from_high_dimensional_interpolation}, for a prescribed accuracy, the rank bound obtained from tensor-product interpolation generally depends exponentially on the dimension \(d\), taking the form \(r=p^d\), as shown below in Theorem~\ref{theorem:interpolation_error_asymptotically_smooth_kernel};
see also Theorem 4.22 in~\cite{Hackbusch_2015_book}.
This exponential dependence can lead to substantial storage requirements even for moderately large \(d\).

\section{Hierarchical Matrices with Tucker Blocks} \label{sec:hierarchical_matrices_with_tucker_blocks}

When considering the interaction between two distinct tensor-product grids, the Tucker decomposition provides an effective approach to low-rank compression of the corresponding interaction tensor.
By combining this idea with the \(\hierarchical\)-matrix framework, we introduce the HMTB matrix in this section, which achieves linear storage complexity.
We also demonstrate how the HMTB matrix can be effectively applied to quasi-uniform grids.

\subsection{Tucker Decomposition from High-Dimensional Interpolation}
\label{subsec:tucker_decomposition_from_high_dimensional_interpolation}

Interpolation is widely used to construct low-rank approximations of kernel matrices~\cite{BF_algorithm, black_box_FMM}.
We first review high-dimensional interpolation on a regular domain and then apply it to construct Tucker decompositions of the interaction matrices between two tensor-product grids.

We begin with the \(1\)D case.
Let \(f\colon[a,b]\to\complex\) be interpolated at the \(p\) Chebyshev points
\begin{equation*}
    \xi_{t}
    =
    \frac{b-a}{2}
    \cos\biggl(\frac{(2t-1)\pi}{2p}\biggr)
    +
    \frac{b+a}{2},
    \quad
    1\leq t\leq p.
\end{equation*}
Its interpolant is
\(\inter_{[a,b]}[f](x)
\coloneq\sum_{t=1}^{p}f(\xi_{t})L_{[a,b];t}(x)\), where
\(L_{[a,b];t}(x)
\coloneq\prod_{\substack{j=1\\j\neq t}}^{p}
(x-\xi_{j})/(\xi_{t}-\xi_{j})\)
is the \(t\)th Lagrange basis polynomial.

In the \(d\)-dimensional case, for
\(B\coloneq\prod_{\ell=1}^{d}[a_{\ell},b_{\ell}]\) and
\(f\colon B\to\complex\), the Chebyshev interpolant is
\begin{equation}
    \label{eq:interpolation_{d}D}
    \inter_{B}[f](\Matrix{x})
    \coloneq
    \sum_{\Matrix{t}\in[p]^{d}}
    f(\Matrix{\xi}_{\Matrix{t}})
    L_{B;\Matrix{t}}(\Matrix{x}),
\end{equation}
where \(\Matrix{t}=(t_{1},\dotsc,t_{d})\),
\(\Matrix{\xi}_{\Matrix{t}}
\coloneq(\xi_{1;t_{1}},\dotsc,\xi_{d;t_{d}})\)
is the tensor-product Chebyshev point in \(B\) corresponding to
\(\Matrix{t}\), and
\(L_{B;\Matrix{t}}(\Matrix{x})
\coloneq\prod_{\ell=1}^{d}
L_{[a_{\ell},b_{\ell}];t_{\ell}}(x_{\ell})\)
is the associated Lagrange basis polynomial.
The following lemma bounds the interpolation error.

\begin{lemma}[Interpolation error in \(d\) dimensions, Lemma~B.7 in~\cite{Hackbusch_2015_book}]
    \label{lemma:interpolation_approximate_error_dD}
    Let \(B=\prod_{\ell=1}^{d}[a_{\ell},b_{\ell}]\) and suppose that
    \(f\in C^{p}(B;\complex)\).
    Then the Chebyshev interpolant~\eqref{eq:interpolation_{d}D} satisfies
    \begin{equation}
        \label{eq:interpolation_approximate_error_dD}
        \|\inter_{B}[f]-f\|_{\infty,B}
        \leq
        2\Lambda_{p}^{d-1}
        \sum_{\ell=1}^{d}
        \biggl(\frac{b_{\ell}-a_{\ell}}{4}\biggr)^{p}
        \frac{\|\partial_{\ell}^{p}f\|_{\infty,B}}{p!},
    \end{equation}
    where \(\Lambda_{p}\) is the Lebesgue constant and satisfies
    \(\Lambda_{p}\sim 2\log(p)/\pi\); see
    Theorem~15.2 in~\cite{Trefethen_2019_book}.
\end{lemma}

When this interpolation estimate is applied to kernel functions on two
distinct boxes, the convergence rate also depends on the geometric separation
between the boxes.
In particular, the strong admissibility parameter \(\eta\) controls the ratio
between the box diameters and their separation distance, and therefore enters
the interpolation error estimate derived in Theorem~\ref{theorem:interpolation_error_asymptotically_smooth_kernel}.
Let \(k(\Matrix{x},\Matrix{y})\) be a kernel function, and let
\(B_{\tau}=\prod_{\ell=1}^{d}[a_{\ell},b_{\ell}]\) and
\(B_{\sigma}=\prod_{\ell=1}^{d}[c_{\ell},d_{\ell}]\)
be two disjoint boxes in \(\real^{d}\).
Applying~\eqref{eq:interpolation_{d}D} successively with respect to
\(\Matrix{x}\) and \(\Matrix{y}\) gives
\begin{equation}
    \label{eq:kernel_function_interpolation}
    k(\Matrix{x},\Matrix{y})
    \approx
    \sum_{\Matrix{t}\in[p]^{d}}
    \sum_{\Matrix{s}\in[p]^{d}}
    k(\Matrix{\xi}_{\Matrix{t}},\Matrix{\eta}_{\Matrix{s}})
    L_{B_{\tau};\Matrix{t}}(\Matrix{x})
    L_{B_{\sigma};\Matrix{s}}(\Matrix{y}).
\end{equation}

Suppose that
\(\{\Matrix{x}_{\Matrix{i}}
=(x_{1;i_{1}},\dotsc,x_{d;i_{d}}):
\Matrix{i}\in[n]^{d}\}\subset B_{\tau}\) and
\(\{\Matrix{y}_{\Matrix{j}}
=(y_{1;j_{1}},\dotsc,y_{d;j_{d}}):
\Matrix{j}\in[n]^{d}\}\subset B_{\sigma}\)
are two tensor-product grids, each containing \(n^{d}\) points.
Their interaction matrix
\(\Matrix{K}\in\complex^{n^{d}\times n^{d}}\) is defined by
\(\Matrix{K}(\Matrix{i},\Matrix{j})
=k(\Matrix{x}_{\Matrix{i}},\Matrix{y}_{\Matrix{j}})\).
By~\eqref{eq:kernel_function_interpolation}, its entries satisfy
\begin{equation}
    \label{eq:kernel_matrix_interpolation}
    \Matrix{K}(\Matrix{i},\Matrix{j})
    \approx
    \sum_{\Matrix{t}\in[p]^{d}}
    \sum_{\Matrix{s}\in[p]^{d}}
    k(\Matrix{\xi}_{\Matrix{t}},\Matrix{\eta}_{\Matrix{s}})
    L_{B_{\tau};\Matrix{t}}(\Matrix{x}_{\Matrix{i}})
    L_{B_{\sigma};\Matrix{s}}(\Matrix{y}_{\Matrix{j}}).
\end{equation}
Define
\begin{equation}
    \label{eq:matrices_interpolation}
    \begin{gathered}
        \Matrix{U}_{\ell}(i_{\ell},t_{\ell})
        =
        L_{[a_{\ell},b_{\ell}];t_{\ell}}(x_{\ell;i_{\ell}}),
        \quad
        \Matrix{U}_{\ell}\in\complex^{n\times p},
        \quad
        1\leq\ell\leq d, \\
        \Matrix{V}_{\ell}(j_{\ell},s_{\ell})
        =
        L_{[c_{\ell},d_{\ell}];s_{\ell}}(y_{\ell;j_{\ell}}),
        \quad
        \Matrix{V}_{\ell}\in\complex^{n\times p},
        \quad
        1\leq\ell\leq d, \\
        \Matrix{G}(\Matrix{t},\Matrix{s})
        =
        k(\Matrix{\xi}_{\Matrix{t}},\Matrix{\eta}_{\Matrix{s}}),
        \quad
        \Matrix{G}\in\complex^{p^{d}\times p^{d}}.
    \end{gathered}
\end{equation}
Then~\eqref{eq:kernel_matrix_interpolation} can be written as
\begin{equation}
    \label{eq:kernel_matrix_low_rank_decomposition}
    \Matrix{K}
    \approx
    \Matrix{U}\Matrix{G}\Matrix{V}^{\transpose},
\end{equation}
where
\(\Matrix{U}
=\Matrix{U}_{d}\otimes\dotsm\otimes\Matrix{U}_{1}
\in\complex^{n^{d}\times p^{d}}\) and
\(\Matrix{V}
=\Matrix{V}_{d}\otimes\dotsm\otimes\Matrix{V}_{1}
\in\complex^{n^{d}\times p^{d}}\).
The matrices \(\Matrix{U}\) and \(\Matrix{V}\) are called the
\textit{basis matrices}, while \(\Matrix{G}\) is called the
\textit{core matrix}.

The decomposition~\eqref{eq:kernel_matrix_low_rank_decomposition} has rank
at most \(p^{d}\) and provides a low-rank representation of admissible leaf
blocks in \(\hierarchical\)-matrices~\cite{Hmatrix_partI,Hackbusch_2015_book}.
Explicitly storing \(\Matrix{U}\), \(\Matrix{V}\), and \(\Matrix{G}\)
requires \(2n^{d}p^{d}+p^{2d}\) entries.
This representation can be further compressed by orthogonalizing
\(\Matrix{U}\) and \(\Matrix{V}\) and truncating the modified core matrix
\(\Matrix{G}\) according to the prescribed accuracy.

However, owing to the inherent tensor-product structure of the grids, it is
more advantageous to represent~\eqref{eq:kernel_matrix_interpolation} in
Tucker format rather than explicitly constructing \(\Matrix{U}\) and
\(\Matrix{V}\).
Let \(\Tensor{K}\in\complex^{n\times\dotsb\times n}\) and
\(\Tensor{G}\in\complex^{p\times\dotsb\times p}\) be the order-\(2d\)
tensorizations of \(\Matrix{K}\) and \(\Matrix{G}\), respectively, defined by
\begin{equation*}
    \Tensor{K}(i_{1},\dotsc,i_d,j_{1},\dotsc,j_d)
    =
    \Matrix{K}(\Matrix{i},\Matrix{j}),
    \quad
    \Tensor{G}(t_1,\dotsc,t_d,s_1,\dotsc,s_d)
    =
    \Matrix{G}(\Matrix{t},\Matrix{s}).
\end{equation*}
Using~\eqref{eq:matrices_interpolation}, we can express
the approximation~\eqref{eq:kernel_matrix_interpolation} in Tucker format as
\begin{equation}
    \label{eq:kernel_matrix_tucker_decomposition}
    \Tensor{K}
    \approx
    \Tucker\bigl(
        \Tensor{G},
        \{\Matrix{U}_{\ell}\}_{\ell=1}^{d},
        \{\Matrix{V}_{\ell}\}_{\ell=1}^{d}
    \bigr).
\end{equation}
The matrix approximation~\eqref{eq:kernel_matrix_low_rank_decomposition}
is the matricization of the Tucker
approximation~\eqref{eq:kernel_matrix_tucker_decomposition}, obtained by
grouping the first \(d\) modes into the row index and the last \(d\)
modes into the column index.
The two forms therefore represent the same interpolant.
In the Tucker representation, \(\Matrix{U}\) and \(\Matrix{V}\)
are retained implicitly through their one-dimensional factors,
without explicitly forming the Kronecker products.
The core matrix and the core tensor both contain \(p^{2d}\) entries.
Explicitly storing the two basis matrices of size
\(n^d\times p^d\) requires \(2n^dp^d\) entries, whereas storing
their \(2d\) one-dimensional factors of size \(n\times p\)
requires only \(2dnp\) entries.
Therefore, for fixed \(n\) and \(p\), the storage required for the Tucker
factors grows linearly with \(d\), whereas that required for the conventional
basis matrices grows exponentially with \(d\).
We emphasize that this reduction does not eliminate the curse of
dimensionality.
The core tensor still contains \(p^{2d}\) entries, so an exponential
dimension dependence remains.
Rather, the Tucker representation mitigates the curse of dimensionality by
reducing the storage and computational costs associated with the basis factors
from exponential to linear dependence on \(d\).

After the interpolation, we apply Algorithm~\ref{alg:Tucker_compression}
to obtain orthonormal factor matrices and reduce the Tucker bond dimensions
according to the prescribed accuracy.
The resulting computational complexity is
\(\bigO\bigl(d p^{2} n + d p^{2d+1}\bigr)\).
For fixed \(d\) and \(p\) and sufficiently large \(n\), the leading term is
\(d p^{2} n\), which depends linearly on \(d\).
The complete construction procedure is presented in
Section~\ref{subsec:HMTB_construction}.

We conclude this section with an error estimate for the interpolation
approximation~\eqref{eq:kernel_function_interpolation}, which relates
the Tucker bond dimensions to the approximation accuracy for kernel
functions satisfying the so-called \textit{asymptotic smoothness} property.

\begin{definition}[Asymptotic smoothness,~\cite{Bebendorf_2008_book,Hackbusch_2015_book}]
    \label{definition:asymptotically_smoothness}
    A kernel function
    \(k\colon\real^{d}\times\real^{d}\to\complex\)
    is said to be asymptotically smooth if there exist constants
    \(C_{\as},\gamma>0\) and \(s\geq0\), depending only on \(k\), such that
    \begin{equation}
        \label{eq:asymptotically_smoothness}
        \bigl|
        \partial_{\Matrix{x}}^{\Matrix{\alpha}}
        \partial_{\Matrix{y}}^{\Matrix{\beta}}
        k\bigl(\Matrix{x},\Matrix{y}\bigr)
        \bigr|
        \leq
        C_{\as}
        \bigl(\Matrix{\alpha}+\Matrix{\beta}\bigr)!
        \gamma^{|\Matrix{\alpha}+\Matrix{\beta}|}
        \bigl\|\Matrix{x}-\Matrix{y}\bigr\|^{
        -s-|\Matrix{\alpha}|-|\Matrix{\beta}|}
    \end{equation}
    for all \(\Matrix{x}\neq\Matrix{y}\) and
    \(\Matrix{\alpha},\Matrix{\beta}\in\natural^{d}\) satisfying
    \(|\Matrix{\alpha}|+|\Matrix{\beta}|>0\).
    Here,
    \(\partial_{\Matrix{x}}^{\Matrix{\alpha}}
    \coloneq\partial_{x_1}^{\alpha_1}\dotsm\partial_{x_d}^{\alpha_d}\)
    and
    \(\partial_{\Matrix{y}}^{\Matrix{\beta}}
    \coloneq\partial_{y_1}^{\beta_1}\dotsm\partial_{y_d}^{\beta_d}\)
    denote the corresponding multidimensional partial differential operators.
    For a multi-index
    \(\Matrix{\nu}=(\nu_1,\ldots,\nu_d)\),
    we use the convention
    \(\Matrix{\nu}!\coloneq\prod_{\ell=1}^{d}\nu_\ell!\);
    in particular,
    \((\Matrix{\alpha}+\Matrix{\beta})!
    =\prod_{\ell=1}^{d}(\alpha_\ell+\beta_\ell)!\).
\end{definition}

A related one-sided interpolation estimate for asymptotically smooth kernels is given in~\cite[Theorem~3.18]{Bebendorf_2008_book}.
When such estimates are available in both variables, successive interpolation in \(\Matrix{x}\) and \(\Matrix{y}\), together with the stability bounds of the interpolation operators, gives a two-sided estimate.
The following theorem provides such an estimate under the asymptotic smoothness assumption and admissibility convention used here.

\begin{theorem}[Interpolation error for asymptotically smooth kernels]
    \label{theorem:interpolation_error_asymptotically_smooth_kernel}
    Suppose that \(k\) is an asymptotically smooth kernel on
    \(B_{\tau}\times B_{\sigma}\), where
    \(B_{\tau}=\prod_{\ell=1}^{d}[a_{\ell},b_{\ell}]\) and
    \(B_{\sigma}=\prod_{\ell=1}^{d}[c_{\ell},d_{\ell}]\)
    are two boxes in \(\real^{d}\).
    If
    \begin{equation*}
        \dist(B_{\tau},B_{\sigma})
        \geq
        \eta
        \max\bigl\{
            \diam(B_{\tau}),
            \diam(B_{\sigma})
        \bigr\},
    \end{equation*}
    then the error \(r(\Matrix{x},\Matrix{y})\) of the Chebyshev
    interpolation~\eqref{eq:kernel_function_interpolation}, defined by
    \begin{equation*}
        r(\Matrix{x},\Matrix{y})
        =
        \sum_{\Matrix{t},\Matrix{s}\in[p]^{d}}
        k\bigl(
            \Matrix{\xi}_{\Matrix{t}},
            \Matrix{\eta}_{\Matrix{s}}
        \bigr)
        L_{B_{\tau};\Matrix{t}}(\Matrix{x})
        L_{B_{\sigma};\Matrix{s}}(\Matrix{y})
        -
        k(\Matrix{x},\Matrix{y}),
    \end{equation*}
    satisfies
    \begin{equation}
        \label{eq:interpolation_error_asymptotically_smooth_kernel}
        \bigl\|r\bigr\|_{\infty,B_{\tau}\times B_{\sigma}}
        \leq
        \frac{
            4C_{\as}\gamma^{p}
            \Lambda_{p}^{2d-1}d
        }{
            (4\eta)^{p}
        }
        \dist(B_{\tau},B_{\sigma})^{-s},
    \end{equation}
    where \(s\), \(C_{\as}\), and \(\gamma\) are the constants
    in~\eqref{eq:asymptotically_smoothness}, and \(\Lambda_{p}\) is the
    Lebesgue constant in~\eqref{eq:interpolation_approximate_error_dD}.
\end{theorem}

\begin{proof}
    Let \(\mathcal{J}_{\tau}^{\Matrix{x}}\) and
    \(\mathcal{J}_{\sigma}^{\Matrix{y}}\) denote the tensor-product
    Chebyshev interpolation operators with respect to \(\Matrix{x}\) on
    \(B_{\tau}\) and \(\Matrix{y}\) on \(B_{\sigma}\), respectively.
    Then
    \begin{equation*}
        \begin{aligned}
            r
            &=
            \mathcal{J}_{\tau}^{\Matrix{x}}
            \mathcal{J}_{\sigma}^{\Matrix{y}}k-k
            \\
            &=
            \mathcal{J}_{\tau}^{\Matrix{x}}
            \bigl(
                \mathcal{J}_{\sigma}^{\Matrix{y}}k-k
            \bigr)
            +
            \bigl(
                \mathcal{J}_{\tau}^{\Matrix{x}}k-k
            \bigr).
        \end{aligned}
    \end{equation*}
    The operator norm of a \(d\)-dimensional tensor-product interpolation
    operator with respect to the uniform norm is bounded by
    \(\Lambda_{p}^{d}\).
    Therefore, applying
    Lemma~\ref{lemma:interpolation_approximate_error_dD} separately with
    respect to \(\Matrix{x}\) and \(\Matrix{y}\), and using
    \(\Lambda_{p}\geq1\), we obtain
    \begin{equation*}
        \begin{aligned}
            \|r\|_{\infty,B_{\tau}\times B_{\sigma}}
            &\leq
            2\Lambda_{p}^{2d-1}
            \Biggl[
                \sum_{\ell=1}^{d}
                \biggl(
                    \frac{b_{\ell}-a_{\ell}}{4}
                \biggr)^{p}
                \frac{
                    \|\partial_{x_{\ell}}^{p}k\|_{
                        \infty,B_{\tau}\times B_{\sigma}}
                }{
                    p!
                }
            \\
            &\qquad\qquad\qquad
                +
                \sum_{\ell=1}^{d}
                \biggl(
                    \frac{d_{\ell}-c_{\ell}}{4}
                \biggr)^{p}
                \frac{
                    \|\partial_{y_{\ell}}^{p}k\|_{
                        \infty,B_{\tau}\times B_{\sigma}}
                }{
                    p!
                }
            \Biggr].
        \end{aligned}
    \end{equation*}
    By the asymptotic smoothness property~\eqref{eq:asymptotically_smoothness}, for every
    \(1\leq\ell\leq d\),
    \begin{equation*}
        \begin{aligned}
            \|\partial_{x_{\ell}}^{p}k\|_{
                \infty,B_{\tau}\times B_{\sigma}}
            &\leq
            C_{\as}p!\gamma^{p}
            \dist(B_{\tau},B_{\sigma})^{-s-p},
            \\
            \|\partial_{y_{\ell}}^{p}k\|_{
                \infty,B_{\tau}\times B_{\sigma}}
            &\leq
            C_{\as}p!\gamma^{p}
            \dist(B_{\tau},B_{\sigma})^{-s-p}.
        \end{aligned}
    \end{equation*}
    Moreover,
    \begin{equation*}
        b_{\ell}-a_{\ell}
        \leq
        \diam(B_{\tau}),
        \qquad
        d_{\ell}-c_{\ell}
        \leq
        \diam(B_{\sigma}).
    \end{equation*}
    Substituting these estimates into the preceding inequality gives
    \begin{equation*}
        \begin{aligned}
            \|r\|_{\infty,B_{\tau}\times B_{\sigma}}
            &\leq
            2C_{\as}\gamma^{p}
            \Lambda_{p}^{2d-1}d
            \dist(B_{\tau},B_{\sigma})^{-s}
            \\
            &\qquad\times
            \Biggl[
                \biggl(
                    \frac{
                        \diam(B_{\tau})
                    }{
                        4\dist(B_{\tau},B_{\sigma})
                    }
                \biggr)^{p}
                +
                \biggl(
                    \frac{
                        \diam(B_{\sigma})
                    }{
                        4\dist(B_{\tau},B_{\sigma})
                    }
                \biggr)^{p}
            \Biggr].
        \end{aligned}
    \end{equation*}
    Finally, the strong admissibility condition implies
    \begin{equation*}
        \frac{\diam(B_{\tau})}
        {\dist(B_{\tau},B_{\sigma})}
        \leq
        \frac{1}{\eta},
        \qquad
        \frac{\diam(B_{\sigma})}
        {\dist(B_{\tau},B_{\sigma})}
        \leq
        \frac{1}{\eta}.
    \end{equation*}
    Hence,
    \begin{equation*}
        \|r\|_{\infty,B_{\tau}\times B_{\sigma}}
        \leq
        \frac{
            4C_{\as}\gamma^{p}
            \Lambda_{p}^{2d-1}d
        }{
            (4\eta)^{p}
        }
        \dist(B_{\tau},B_{\sigma})^{-s},
    \end{equation*}
    which proves
    \eqref{eq:interpolation_error_asymptotically_smooth_kernel}.
\end{proof}

For fixed \(d\), the Lebesgue constant satisfies
\(\Lambda_{p}=\bigO(\log p)\).
Hence, if \(4\eta>\gamma\), the polylogarithmic factor
\(\Lambda_{p}^{2d-1}\) is dominated by the exponential decay of
\(\bigl(\gamma/(4\eta)\bigr)^{p}\).
Since a larger \(\eta\) imposes stronger separation, it also improves the
convergence rate in~\eqref{eq:interpolation_error_asymptotically_smooth_kernel}.
Thus, for sufficiently well-separated boxes \(B_{\tau}\) and \(B_{\sigma}\),
the absolute interpolation error decreases nearly exponentially with \(p\).
Moreover, for fixed \(d\), \(\eta\), and a prescribed accuracy relative to
the scale \(\dist(B_{\tau},B_{\sigma})^{-s}\), the resulting rank bound
\(p^{d}\) is independent of the cardinalities of the corresponding index
clusters.

The constant \(\gamma\) in the general asymptotic smoothness estimate may
be conservative because it accounts for arbitrary mixed derivatives,
whereas the proof of
Theorem~\ref{theorem:interpolation_error_asymptotically_smooth_kernel}
requires only single-coordinate derivatives.
Let \(\Matrix{z}=\Matrix{x}-\Matrix{y}\) and \(r=\|\Matrix{z}\|\).
For \(m\geq1\), the two-dimensional logarithmic kernel and the
three-dimensional Newton kernel satisfy
\begin{equation*}
    \bigl|\partial_{z_{\ell}}^{m}\log r\bigr|
    \leq
    m!r^{-m},
    \qquad
    \biggl|\partial_{z_{\ell}}^{m}\frac{1}{r}\biggr|
    \leq
    m!r^{-m-1},
\end{equation*}
respectively, with the same bounds for derivatives with respect to
\(x_{\ell}\) and \(y_{\ell}\).
Thus, one may take \(\gamma=1\) for the derivatives required in the
interpolation analysis.
Since \(4\eta=4/\sqrt{d}>1\) for \(\eta=1/\sqrt{d}\) and
\(d\in\{2,3\}\), the adopted admissibility parameter guarantees the
nearly exponential convergence predicted by the interpolation estimate
for these two kernels.
For the Helmholtz kernel, the constants in the asymptotic-smoothness
estimate depend on the wavenumber.
Therefore, the above estimate should
not be interpreted as a frequency-uniform rank bound; in this work we
consider only a fixed mildly oscillatory regime.

\subsection{Hierarchical Matrices with Tucker Blocks}
\label{subsec:hierarchical_matrix_with_tucker_blocks}

Section~\ref{subsec:tucker_decomposition_from_high_dimensional_interpolation}
shows how tensor-product interpolation naturally yields Tucker
approximations of interaction matrices on block domains where
the kernel is sufficiently smooth.
The matrix representation~\eqref{eq:kernel_matrix_low_rank_decomposition}
and the Tucker representation~\eqref{eq:kernel_matrix_tucker_decomposition}
describe the same interpolant.
We now combine these local Tucker approximations of admissible leaf
blocks with a hierarchical block partition to define HMTB matrices.
To apply this construction blockwise, we require the cluster tree
and block cluster tree to inherit the tensor-product structure
of the grid.

\begin{definition}[Tensor node cluster tree]
    \label{definition:tensor_node_cluster_tree}
    Let
    \(I=\bigl\{\Matrix{i}\colon\Matrix{i}\in[n]^{d}\bigr\}\)
    be the index set associated with the tensor-product grid arising from
    the discretization~\eqref{eq:Nystrom_IE_uni}.
    A cluster tree \(\tree_{I}\) is called a
    \textit{tensor node cluster tree} if every node
    \(\tau\in\tree_{I}\) can be expressed as
    \(\tau
    =
    \tau_{1}\times\dotsb\times\tau_{d}\),
    where \(\tau_{\ell}\subseteq[n]\) is a contiguous index set associated with
the \(\ell\)-th dimension for each \(\ell\in[d]\).
\end{definition}

\begin{definition}[Tensor node block cluster tree]
    \label{definition:tensor_node_block_cluster_tree}
    Let \(\tree_{I}\) be a tensor node cluster tree with root index set \(I\).
    The block cluster tree \(\tree_{I\times I}\) constructed from
    \(\tree_{I}\) and an admissibility condition is called a
    \textit{tensor node block cluster tree}.
\end{definition}

A key property of the tensor node cluster tree is that the grid points
associated with each node retain a tensor-product structure, while their
computational domain forms a \(d\)-dimensional box.
On a tensor node block cluster tree, the following definition
retains the hierarchical block structure of
Definition~\ref{definition:Hmatrix}.
For each admissible leaf block, the matrix low-rank representation
is replaced by a Tucker representation of its order-\(2d\)
tensorization.
Inadmissible leaf blocks remain dense, and non-leaf blocks
retain the same recursive structure.

\begin{definition}[Hierarchical matrices with Tucker blocks, HMTB matrices]
    \label{definition:HMTB}
    Suppose \(\tree_{I\times I}\) is the tensor node block cluster tree
    with root index set \(I\).
    We define a matrix \(\Matrix{A}^{\HMTB}\) as an HMTB matrix on
    \(\tree_{I\times I}\) with Tucker bond dimensions bounded by \(p\)
    if the following conditions hold:
    \begin{itemize}
        \item If \((\tau,\sigma)\) is an admissible leaf,
        denoting the order-\(2d\) tensorization of
        \(\Matrix{A}^{\HMTB}_{\tau,\sigma}\) by
        \(\Tensor{A}^{\HMTB}_{\tau,\sigma}\), we have
        \begin{equation*}
            \Tensor{A}^{\HMTB}_{\tau,\sigma}
            =
            \Tucker\bigl(
                \Tensor{G}_{\tau,\sigma},
                \{\Matrix{U}_{\tau,\sigma;\ell}\}_{\ell=1}^{d},
                \{\Matrix{V}_{\tau,\sigma;\ell}\}_{\ell=1}^{d}
            \bigr).
        \end{equation*}
        for some
        \(\Matrix{U}_{\tau,\sigma; \ell}
        \in\complex^{|\tau_{\ell}|\times p}\),
        \(\Matrix{V}_{\tau,\sigma; \ell}
        \in\complex^{|\sigma_{\ell}|\times p}\), and
        \(\Tensor{G}_{\tau,\sigma}
        \in\complex^{p\times\dotsb\times p}\).

        \item If \((\tau,\sigma)\) is an inadmissible leaf, then
        \(\Matrix{A}^{\HMTB}_{\tau,\sigma}\) is a dense matrix.

        \item If \((\tau,\sigma)\) is a non-leaf node, then
        \(\Matrix{A}^{\HMTB}_{\tau,\sigma}\) is a block matrix with blocks
        \(\Matrix{A}^{\HMTB}_{\alpha,\beta}\), where
        \((\alpha,\beta)
        \in
        \children\bigl((\tau,\sigma)\bigr)\).
    \end{itemize}
\end{definition}

It follows from~\eqref{eq:kernel_matrix_low_rank_decomposition}
and~\eqref{eq:kernel_matrix_tucker_decomposition} that, when both
representations are derived from the same high-dimensional interpolation,
the matrix representation used in \(\hierarchical\)-matrices,
with rank at most \(p^d\), and the TLR representation with uniform
Tucker bond dimension \(p\) achieve the same approximation accuracy.
However, the tensor structure of the latter leads to a lower storage cost.
It should be emphasized that \(p^d\) is the rank bound obtained directly
from the tensor-product interpolation.
A subsequent algebraic recompression may reduce the numerical rank of the
corresponding \(\hierarchical\)-matrix substantially.
Therefore, the following complexity comparison concerns the
interpolation-generated representations and should not be interpreted as a
lower bound on the rank or complexity of a general algebraically compressed
\(\hierarchical\)-matrix.
In particular, the storage complexity of an HMTB matrix is linear in the
matrix dimension \(N\), as stated in the following proposition.

\begin{proposition}[Storage of HMTB matrices under weak admissibility]
    \label{prop:storage_HMTB_weak}
    For fixed \(d\geq2\), let \(N_{0}=n_{0}^{d}\) and
    \(I=\bigl\{\Matrix{i}\colon\Matrix{i}\in[n]^{d}\bigr\}\) be the
    index set associated with the tensor-product grid arising from the
    discretization~\eqref{eq:Nystrom_IE_uni}, where \(N=n^{d}\).
    Suppose that the tensor node cluster tree \(\tree_{I}\) is constructed
    by recursive \(2^{d}\)-way partitioning until every leaf node contains
    at most \(N_{0}\) points, and let \(\tree_{I\times I}\) be the
    corresponding tensor node block cluster tree under the weak
    admissibility condition.
    If \(\Matrix{A}^{\HMTB}\) is an HMTB matrix with Tucker bond
    dimensions bounded by \(p\) and \(p\leq n_{0}\leq2p\), then its
    storage complexity is \(\bigO(p^{d}N)\), where the prefactor depends
    only on \(d\).%
    \footnote{Hereafter, the prefactors in all complexity
    estimates are assumed to depend only on \(d\).}
\end{proposition}

\begin{proof}
    For simplicity, assume that
    \(N=n^{d}=2^{dL}n_{0}^{d}\), where \(L\) is the maximum level of
    \(\tree_{I}\).
    The general case follows from the same argument with modified constant
    factors.
    At level \(\ell\), the number of clusters is \(2^{d\ell}\), and each
    cluster contains \(N/2^{d\ell}\) indices.
    Under weak admissibility, each cluster is associated with at most one
    inadmissible block and \(2^{d}-1\) admissible blocks.

    Using the storage cost of TLR matrices discussed in
    Section~\ref{subsec:tucker_decomposition_from_high_dimensional_interpolation},
    the storage of \(\Matrix{A}^{\HMTB}\), denoted by \(S^{\HMTB}(N)\),
    satisfies
    \begin{equation}
        \label{eq:storage_HMTB_weak}
        \begin{aligned}
            S^{\HMTB}(N)
            &\leq
            \sum_{\ell=1}^{L}
            2^{d\ell}(2^{d}-1)
            \bigl(
                2d n_{0}2^{L-\ell}p+p^{2d}
            \bigr)
            +
            2^{dL}n_{0}^{2d}
            \\
            &\leq
            \left(
                \frac{2^{d}d(2^{d}-1)}
                {2^{d-1}-1}
                \frac{p}{n_{0}^{d-1}}
                +
                2^{d}\frac{p^{2d}}{n_{0}^{d}}
                +
                n_{0}^{d}
            \right)N
            \\
            &\leq
            \left(
                \frac{2^{d}d(2^{d}-1)}
                {2^{d-1}-1}p^{2-d}
                +
                2^{d}p^{d}
                +
                2^{d}p^{d}
            \right)N
            \\
            &=
            \bigO\bigl(p^{d}N\bigr).
        \end{aligned}
    \end{equation}
    This completes the proof.
\end{proof}

An argument analogous to that in the proof of
Proposition~\ref{prop:storage_HMTB_weak} gives the following
bound on the storage cost \(S^{\hierarchical}(N)\) of the
corresponding \(\hierarchical\)-matrix under weak admissibility:
\begin{equation}
    \label{eq:storage_HMAT_weak}
    \begin{aligned}
        S^{\hierarchical}(N)
        &\leq
        \sum_{\ell=1}^{L}
        2^{d\ell}(2^{d}-1)
        \biggl(
            2\frac{N}{2^{d\ell}}p^{d}
            +
            p^{2d}
        \biggr)
        +
        2^{dL}n_{0}^{2d}
        \\
        &\leq
        \left(
            2(2^{d}-1)Lp^{d}
            +
            2^{d}\frac{p^{2d}}{n_{0}^{d}}
            +
            n_{0}^{d}
        \right)N
        \\
        &\leq
        \left(
            \frac{2(2^{d}-1)}{d}
            p^{d}\log_{2}\bigl(N/N_{0}\bigr)
            +
            2^{d}p^{d}
            +
            2^{d}p^{d}
        \right)N
        \\
        &=
        \bigO\bigl(p^{d}N\log_{2}(N)\bigr).
    \end{aligned}
\end{equation}
For the interpolation-generated representations considered here,
explicitly storing the basis matrices of admissible
\(\hierarchical\)-matrix blocks requires
\(\bigO(p^dN\log_2N)\) entries.
For HMTB matrices, summing the factor-matrix contributions
in~\eqref{eq:storage_HMTB_weak} gives
\(\bigO(pN/n_0^{d-1})=\bigO(p^{2-d}N)\),
using \(d\geq2\) and \(p\leq n_0\leq2p\).
This dependence on \(p\) includes the choice of leaf size
\(N_0=n_0^d\) and the corresponding block partition.
The Tucker cores and dense leaf blocks together require
\(\bigO(p^dN)\) storage, so the total storage remains
\(\bigO(p^dN)\), which is linear in \(N\) for fixed \(d\) and \(p\).

Similar results can be obtained under strong admissibility, with only the
prefactor in the storage estimate being different.

\begin{proposition}[Storage of HMTB matrices under strong admissibility]
    \label{prop:storage_HMTB_strong}
       Under the same conditions as in
    Proposition~\ref{prop:storage_HMTB_weak}, except that
    \(\tree_{I\times I}\) is constructed under the strong admissibility
    condition with \(\eta=1/\sqrt{d}\), the storage complexity of an HMTB
    matrix \(\Matrix{A}^{\HMTB}\) with Tucker bond dimensions bounded
    by \(p\) is \(\bigO(p^{d}N)\).
\end{proposition}

\begin{proof}
    The proof is analogous to that of
    Proposition~\ref{prop:storage_HMTB_weak}.
    The only difference is that, under strong admissibility, the numbers of
    inadmissible and admissible blocks associated with each cluster are bounded
    by \(3^{d}\) and \(6^{d}-3^{d}\), respectively.
\end{proof}

Another limitation arises from the size of the leaf clusters.
Under the present tensor-grid construction, a leaf contains
\(N_0=n_0^d\) points, with \(n_0\) chosen to be comparable to the
interpolation order \(p\).
Hence, even for moderate \(p\), the leaf size grows exponentially with the
dimension.
This limits the polynomial orders that can be used in genuinely high-dimensional problems and makes the present construction most suitable
for low-to-moderate spatial dimensions.
Therefore, the complexity estimates in this work should be interpreted for fixed \(d\), rather than as dimension-uniform bounds.

\subsection{HMTB Matrices on a Quasi-Uniform Grid}
\label{subsec:applications_of_HMTB_matrices_on_the_quasi_uniform_grid}

The HMTB matrix proposed in
Section~\ref{subsec:hierarchical_matrix_with_tucker_blocks} requires a uniform
tensor grid.
However, nonuniform grids may arise from perturbed data or finite element
discretizations on nonuniform meshes.
Fully nonuniform discretizations can be challenging for HMTB matrices and may
also produce deep and unbalanced cluster trees for
\(\hierarchical\)-matrices, potentially destroying their quasi-linear
scaling.
We therefore consider quasi-uniform grids, which can be viewed as
perturbations of uniform grids.
In this section, we apply HMTB matrices to the IE~\eqref{eq:IE} discretized on
a quasi-uniform grid by interpolating between the tensor grid and the
quasi-uniform grid, as in the nonuniform fast Fourier
transform~\cite{Barnett_Magland_Klinteberg_2019,Greengard_Lee_2004}.

Suppose that the quasi-uniform grid points
\(\{\Matrix{x}_{i}\}_{i=1}^{N}\) are associated with computational cells
\(\{\Omega_{i}^{\quasi}\}_{i=1}^{N}\) having mutually disjoint interiors and
satisfying \(\Omega
=
\bigcup_{i=1}^{N}\Omega_{i}^{\quasi}\).
For example, the cells may be the triangular elements of a quasi-uniform
finite element mesh, with each \(\Matrix{x}_{i}\) chosen as the centroid of
\(\Omega_{i}^{\quasi}\), as illustrated in
Figure~\ref{fig:triangulation}.
We discretize the IE~\eqref{eq:IE} as
\begin{equation*}
    a(\Matrix{x}_{i})u_{i}
    +
    \sum_{j=1}^{N}
    K_{i,j}\bigl|\Omega_{j}^{\quasi}\bigr|u_{j}
    =
    f(\Matrix{x}_{i}),
    \qquad
    i=1,\dotsc,N,
\end{equation*}
where
\begin{equation*}
    K_{i,j}
    \coloneq
    \begin{cases}
        k(\Matrix{x}_{i},\Matrix{x}_{j}),
        & i\neq j, \\[2mm]
        \displaystyle
        \frac{1}{\bigl|\Omega_{i}^{\quasi}\bigr|}
        \int_{\Omega_{i}^{\quasi}}
        k(\Matrix{x}_{i},\Matrix{y})
        \diff\Matrix{y},
        & i=j.
    \end{cases}
\end{equation*}

\begin{figure}[htbp]
    \centering
    \includegraphics[width=.3\textwidth]{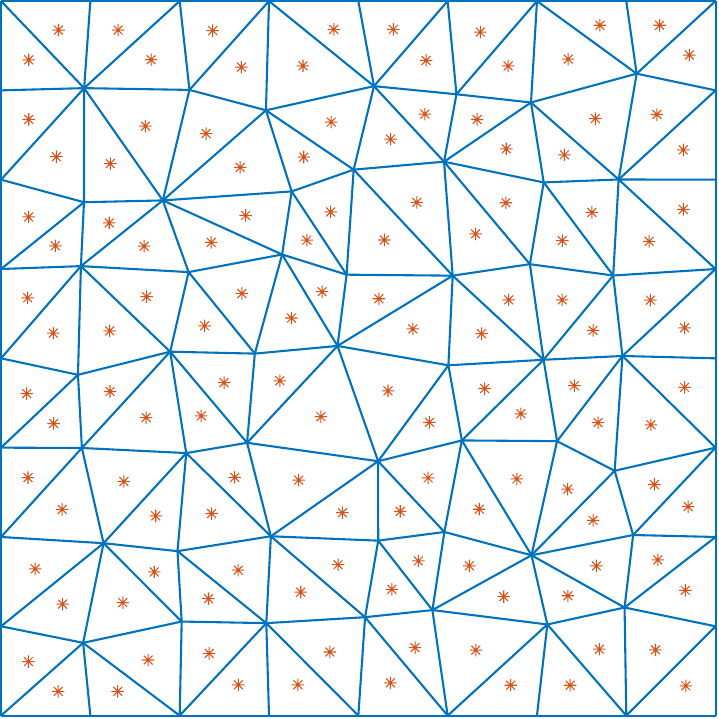}
    \caption{Triangulation of \([0,1]^2\).
    The triangle edges are shown in blue,
    and their centroids are marked by orange stars.}
    \label{fig:triangulation}
\end{figure}

Throughout this section, we assume that the input vector \(\Matrix{u}\) is
sampled from a smooth function \(u\), i.e.,
\(u_i=u(\Matrix{x}_i)\).
We introduce an auxiliary uniform tensor grid
\(\{\hat{\Matrix{x}}_t\}_{t=1}^{M}\), where \(M=m^d\) and \(m\) is the number
of grid points in each direction, together with its computational cells
\(\{\Omega_t^{\uni}\}_{t=1}^{M}\).
Let \(\Matrix{\hat{A}}\) be the matrix defined by
\eqref{eq:Nystrom_IE_uni} on this grid.
We further introduce grid-transfer matrices
\(\Matrix{S}\colon\complex^N\to\complex^M\) and
\(\Matrix{T}\colon\complex^M\to\complex^N\), which transfer data from the quasi-uniform grid to the uniform grid
and vice versa.
We construct these matrices using the local piecewise-constant
\(L^2\)-projections described below.
Their entries satisfy
\(\Matrix{S}_{t,i}=\Matrix{T}_{i,t}=0\) whenever
\(\lvert\Omega_t^{\uni}\cap\Omega_i^{\quasi}\rvert=0\).
We assume that each cell has positive-volume overlap with at most
a fixed number of cells of the other mesh under grid refinement.
Then
\begin{equation}
    \label{eq:quasi-uniform_uniform_relationship}
    \Matrix{A}\Matrix{u}
    \approx
    \Matrix{T}\Matrix{\hat{A}}\Matrix{S}\Matrix{u}.
\end{equation}
Specifically, \(\Matrix{S}\Matrix{u}\) transfers the input to the uniform
grid, \(\Matrix{\hat{A}}\) performs the matrix-vector multiplication there,
and \(\Matrix{T}\) transfers the result back to the quasi-uniform grid.
Since \(\Matrix{\hat{A}}\) admits an HMTB approximation and
\(\Matrix{S}\) and \(\Matrix{T}\) are sparse, this computation can be
performed efficiently.
Moreover, replacing \(\Matrix{\hat{A}}\) by its HMTB approximation
\(\Matrix{\hat{A}}^{\HMTB}\) yields a data-sparse approximation to the action
of \(\Matrix{A}\) on vectors sampled from smooth functions.

We next discuss the choice of \(M\) and the construction of
\(\Matrix{T}\) and \(\Matrix{S}\).
The choice \(M\approx\rho^{d}N\), where \(\rho\geq1\) is a small constant,
typically provides a suitable balance between accuracy and efficiency; see
Section~\ref{subsec:quasiuni_grid_2D} for numerical results.

We construct the grid-transfer matrices using an \(L^{2}\)-projection.
A vector \(\Matrix{u}\in\complex^{N}\) on the quasi-uniform grid is identified
with the piecewise-constant function
\(u^{\quasi}(\Matrix{x})
=\sum_{i=1}^{N}u_i\indexFunction_{\Omega_i^{\quasi}}(\Matrix{x})\).
Here, \(\indexFunction_D\) denotes the indicator function of a set \(D\),
equal to \(1\) on \(D\) and \(0\) outside \(D\).
Its approximation on the uniform grid is written as
\(u^{\uni}(\Matrix{x})
=\sum_{t=1}^{M}\hat{u}_t
\indexFunction_{\Omega_t^{\uni}}(\Matrix{x})\).
Projecting \(u^{\quasi}\) onto
\(\indexFunction_{\Omega_t^{\uni}}\) gives
\begin{equation*}
    \hat{u}_t
    =
    \sum_{i=1}^{N}
    \frac{
        \bigl|\Omega_t^{\uni}\cap\Omega_i^{\quasi}\bigr|
    }{
        \bigl|\Omega_t^{\uni}\bigr|
    }
    u_i.
\end{equation*}
Hence, the transfer matrices from the quasi-uniform grid to the uniform grid
and vice versa are
\(\Matrix{S}_{t,i}
=\lvert\Omega_t^{\uni}\cap\Omega_i^{\quasi}\rvert/
\lvert\Omega_t^{\uni}\rvert\) and
\(\Matrix{T}_{i,t}
=\lvert\Omega_i^{\quasi}\cap\Omega_t^{\uni}\rvert/
\lvert\Omega_i^{\quasi}\rvert\), respectively.
Under the stated bounded-overlap assumption, each row and column of
\(\Matrix{S}\) and \(\Matrix{T}\) contains only \(\bigO(1)\)
nonzero entries.
Thus, both matrices are sparse, with
\(\bigO(\max\{M,N\})=\bigO(M)\) nonzero entries.
Higher-order grid-transfer schemes may also be used; the present scheme is
chosen for simplicity.

\section{Construction and Forward Evaluation of HMTB Matrices}
\label{sec:HMTB_construction_matrix_vector_multiplication}

This section presents the construction and forward evaluation of HMTB matrices.
Both follow a framework similar to that for \(\hierarchical\)-matrices, except
that admissible leaf blocks are represented in Tucker format.
Throughout this section, we assume that a tensor node block cluster tree
\(\tree_{I\times I}\) associated with a uniform grid is given.

Although Algorithms~\ref{alg:HMTB_construction} and~\ref{alg:HMTB_matrix_vector_multiplication} are written recursively following the block cluster tree, the computation need not be implemented recursively.
After the block cluster tree is constructed, its leaf blocks can be stored in a linear list and processed sequentially or in parallel.
In particular, the blockwise computations can be scheduled independently, with the resulting contributions accumulated into the corresponding output subvectors, making this linearized representation convenient for parallel implementations.

\subsection{Construction of HMTB Matrices}
\label{subsec:HMTB_construction}

The construction of \(\Matrix{A}^{\HMTB}\) follows
Definition~\ref{definition:HMTB}.
For each node \((\tau,\sigma)\), we distinguish the following three cases:
\begin{itemize}
    \item If \((\tau,\sigma)\) is an admissible leaf node, then
    \(\Omega_{\tau}\) and \(\Omega_{\sigma}\) are disjoint, so the corresponding
    block contains no diagonal term.
    We first approximate the tensorization of the kernel matrix block by
    \begin{equation*}
        \Tensor{K}_{\tau,\sigma}
        \approx
        \Tucker\bigl(
            \Tensor{G},
            \{\Matrix{U}_{\ell}\}_{\ell=1}^{d},
            \{\Matrix{V}_{\ell}\}_{\ell=1}^{d}
        \bigr)
    \end{equation*}
    using the \(d\)-dimensional interpolation in
    \eqref{eq:matrices_interpolation}.
    Other suitable interpolation schemes may also be used to construct
    the Tucker approximation.
    We then absorb \(h^d\) into the core tensor and compress the resulting
    Tucker representation using
    Algorithm~\ref{alg:Tucker_compression}, yielding the TLR representation of
    \(\Matrix{A}^{\HMTB}_{\tau,\sigma}\) with the prescribed accuracy.

    \item If \((\tau,\sigma)\) is an inadmissible leaf node, we store the block
    densely as
    \begin{equation}
        \label{eq:A_t_s}
        \Matrix{A}^{\HMTB}_{\tau,\sigma}
        =
        \bigl(
            a(\Matrix{x}_{i})\delta_{i,j}
            +
            K_{i,j}h^{d}
        \bigr)_{i\in\tau,\,j\in\sigma},
    \end{equation}
    where \(\delta_{i,j}\) denotes the Kronecker delta.

    \item If \((\tau,\sigma)\) is a non-leaf node, the construction is applied
    recursively to every
    \((\alpha,\beta)\in\children(\tau,\sigma)\).
\end{itemize}
Algorithm~\ref{alg:HMTB_construction} summarizes the complete procedure,
starting from the root \((I,I)\) of \(\tree_{I\times I}\).

\begin{algorithm}
\caption{
    \(\HConstruct\bigl(
        a(\Matrix{x}),
        k(\Matrix{x},\Matrix{y}),
        (\tau,\sigma),
        p,
        \varepsilon
    \bigr)\)
}
\label{alg:HMTB_construction}
\begin{algorithmic}[1]
    \Require{
        Functions \(a(\Matrix{x})\) and \(k(\Matrix{x},\Matrix{y})\),
        current node \((\tau,\sigma)\), interpolation order \(p\),
        and compression tolerance \(\varepsilon\)
    }
    \Ensure{
        The HMTB representation of the block
        \(\Matrix{A}^{\HMTB}_{\tau,\sigma}\)
    }
    \If{\((\tau,\sigma)\) is an admissible leaf}
        \State{
            Compute \(\Tensor{G}\),
            \(\{\Matrix{U}_{\ell}\}_{\ell=1}^{d}\), and
            \(\{\Matrix{V}_{\ell}\}_{\ell=1}^{d}\) by Chebyshev interpolation
            as in~\eqref{eq:matrices_interpolation}
        }
        \State{
            Update \(\Tensor{G}\gets h^{d}\Tensor{G}\)
        }
        \State{
            Compress
            \(\Tucker\bigl(
                \Tensor{G},
                \{\Matrix{U}_{\ell}\}_{\ell=1}^{d},
                \{\Matrix{V}_{\ell}\}_{\ell=1}^{d}
            \bigr)\)
            to tolerance \(\varepsilon\) using
            Algorithm~\ref{alg:Tucker_compression}
        }
        \State{
            Set \(\Matrix{A}^{\HMTB}_{\tau,\sigma}\) to the resulting
            compressed TLR representation
        }
    \ElsIf{\((\tau,\sigma)\) is an inadmissible leaf}
        \State{
            Form \(\Matrix{A}^{\HMTB}_{\tau,\sigma}\)
            by~\eqref{eq:A_t_s}
        }
    \Else
        \For{\((\alpha,\beta)\in\children(\tau,\sigma)\)}
            \State{
                \(\Matrix{A}^{\HMTB}_{\alpha,\beta}
                \gets
                \HConstruct\bigl(
                    a(\Matrix{x}),
                    k(\Matrix{x},\Matrix{y}),
                    (\alpha,\beta),
                    p,
                    \varepsilon
                \bigr)\)
            }
        \EndFor
    \EndIf
\end{algorithmic}
\end{algorithm}

\begin{remark}
The interpolation factor matrices in~\eqref{eq:matrices_interpolation}
depend only on the normalized local grids and interpolation orders.
Factors corresponding to identical normalized grids and orders,
together with their QR factorizations, can therefore be precomputed
and reused.
For the regular uniform-grid hierarchy considered here, this applies
to clusters at the same level.
The same reuse is available for the auxiliary uniform-grid matrix
\(\Matrix{\hat{A}}\) in
Section~\ref{subsec:applications_of_HMTB_matrices_on_the_quasi_uniform_grid}.

More generally, factors can be grouped according to their normalized
local grids and interpolation orders and reused within each group.
This accommodates different interpolation orders without restricting
the applicability of the construction.
The subsequent compression of the transformed core remains specific
to each admissible block, so the final compressed factor matrices
need not coincide.

Algorithm~\ref{alg:HMTB_construction} and the implementation used in
the reported experiments compute the interpolation factors and their
QR factorizations independently for each admissible block.
The construction-cost analysis and timings therefore refer to this
implementation without reuse.
Reusing these quantities can eliminate repeated interpolation and QR
computations, but the resulting practical speedup has not been
quantified in the present experiments.
\end{remark}

The leading complexity in Algorithm~\ref{alg:HMTB_construction} arises from
the Tucker compression of the admissible leaf blocks.
Under the assumptions of Proposition~\ref{prop:storage_HMTB_weak}, in
particular, \(p\leq n_0\leq 2p\), the total compression cost is
\begin{equation*}
    \sum_{\ell=1}^{L}
    \bigO\bigl(
        2^{d\ell}
        (2d n_0 2^{L-\ell}p^2+2d p^{2d+1})
    \bigr)
    =
    \bigO\bigl(p^{3-d}N+p^{d+1}N\bigr)
    =
    \bigO\bigl(p^{d+1}N\bigr).
\end{equation*}
The dominant contribution comes from the compression of the
core tensors.

As in the storage comparison in Section~\ref{subsec:hierarchical_matrix_with_tucker_blocks}, the following comparison with
\(\hierarchical\)-matrices concerns the low-rank representations generated
by the same tensor-product interpolation before possible further algebraic
recompression.
For \(\hierarchical\)-matrices, the leading construction cost similarly
arises from the compression of the low-rank representations.
At level \(\ell\), the compression of each admissible block requires
\(\bigO(2^{d(L-\ell)}n_0^d p^{2d})\) operations for the QR factorizations
of the basis matrices and \(\bigO(p^{3d})\) operations for the subsequent
compression.
Consequently, the total compression cost for all admissible leaf blocks
of the \(\hierarchical\)-matrix is
\begin{equation*}
    \begin{aligned}
        &\sum_{\ell=1}^{L}
        \bigO\bigl(
            2^{d\ell}(2^d-1)
            (2^{d(L-\ell)}n_0^d p^{2d}+p^{3d})
        \bigr)
        \\
        &\qquad=
        \bigO\bigl(p^{2d}N\log_2N+p^{2d}N\bigr)
        =
        \bigO\bigl(p^{2d}N\log_2N\bigr).
    \end{aligned}
\end{equation*}
In this case, the dominant contribution comes from the QR factorizations of
the basis matrices.
Therefore, the construction complexities of HMTB matrices and
\(\hierarchical\)-matrices are
\(\bigO(p^{d+1}N)\) and
\(\bigO(p^{2d}N\log_2 N)\), respectively.
This difference arises because an admissible HMTB block has \(2d\) factor
matrices of size \(n\times p\), whereas the corresponding
\(\hierarchical\)-matrix block has two basis matrices of size
\(n^d\times p^d\).
We summarize this comparison in the following proposition.

\begin{proposition}[Complexity of HMTB construction]
    \label{prop:HMTB_construct_complexity}
    Under both weak and strong admissibility conditions, the complexity of
    constructing an HMTB matrix with interpolation order \(p\) and compression
    tolerance \(\varepsilon\) is \(\bigO(p^{d+1} N)\).
\end{proposition}

\subsection{Forward Evaluation of HMTB Matrices}
\label{subsec:HMTB_matrix_vector_multiplication}

Using the hierarchical structure, the forward evaluation of an HMTB matrix,
i.e., its multiplication by a vector, can also be performed efficiently.
Starting from the root \((I,I)\), let \((\tau,\sigma)\) be the current node and
consider the update
\begin{equation}
    \label{eq:HMTB_mult_vector}
    \Matrix{f}_{\tau}
    \gets
    \Matrix{f}_{\tau}
    +
    \Matrix{A}^{\HMTB}_{\tau,\sigma}
    \Matrix{u}_{\sigma}.
\end{equation}
We distinguish the following three cases:
\begin{itemize}
    \item If \((\tau,\sigma)\) is an admissible leaf, suppose that the TLR
    representation of \(\Matrix{A}^{\HMTB}_{\tau,\sigma}\) is
    \begin{equation*}
        \Matrix{A}^{\HMTB}_{\tau,\sigma}
        =
        \Tucker\bigl(
            \Tensor{G},
            \{\Matrix{U}_{\ell}\}_{\ell=1}^{d},
            \{\Matrix{V}_{\ell}\}_{\ell=1}^{d}
        \bigr).
    \end{equation*}
    The update is then computed as
    \begin{equation}
        \label{eq:TLR_mult_vector}
        \Matrix{f}_{\tau}
        \gets
        \Matrix{f}_{\tau}
        +
        (\Matrix{U}_{d}\otimes\dotsb\otimes\Matrix{U}_{1})
        \Matrix{G}
        (\Matrix{V}_{d}\otimes\dotsb\otimes\Matrix{V}_{1})^{\transpose}
        \Matrix{u}_{\sigma},
    \end{equation}
    where \(\Matrix{G}\) is obtained by reshaping \(\Tensor{G}\) into a matrix.

    The Kronecker products in~\eqref{eq:TLR_mult_vector} need not be formed
    explicitly. Let \(\Tensor{U}\) be the tensorization of
    \(\Matrix{u}_{\sigma}\). The contribution to \(\Matrix{f}_{\tau}\) can be
    computed through mode products and a tensor contraction:
    \begin{equation*}
        \begin{aligned}
            \Tensor{W}
            &=
            \Tensor{U}
            \times_{1}\Matrix{V}_{1}^{\transpose}
            \times_{2}\dotsb
            \times_{d}\Matrix{V}_{d}^{\transpose}, \\
            \Tensor{H}
            &=
            \Tensor{G}
            \times_{[d+1,\dotsc,2d],[1,\dotsc,d]}
            \Tensor{W}, \\
            \Tensor{F}
            &=
            \Tensor{H}
            \times_{1}\Matrix{U}_{1}
            \times_{2}\dotsb
            \times_{d}\Matrix{U}_{d}.
        \end{aligned}
    \end{equation*}

    If all factor matrices are of size \(n\times p\) and the matricized core
    \(\Matrix{G}\) is of size \(p^{d}\times p^{d}\), this computation requires
    \(\bigO(dpn^{d}+p^{2d})\) operations. By comparison, evaluating the
    corresponding tensor-product representation as a conventional
    rank-\(p^{d}\) representation requires
    \(\bigO(p^{d}n^{d}+p^{2d})\) operations.

    \item If \((\tau,\sigma)\) is an inadmissible leaf, the
    update~\eqref{eq:HMTB_mult_vector} is computed directly because
    \(\Matrix{A}^{\HMTB}_{\tau,\sigma}\) is stored as a small dense matrix.

    \item If \((\tau,\sigma)\) is a non-leaf node, the update is performed
    recursively for every child node:
    \begin{equation*}
        \Matrix{f}_{\alpha}
        \gets
        \Matrix{f}_{\alpha}
        +
        \Matrix{A}^{\HMTB}_{\alpha,\beta}
        \Matrix{u}_{\beta},
        \qquad
        (\alpha,\beta)\in\children(\tau,\sigma).
    \end{equation*}
\end{itemize}

Algorithm~\ref{alg:HMTB_matrix_vector_multiplication} summarizes the
forward evaluation procedure, and
Proposition~\ref{prop:HMTB_apply_complexity} gives its computational
complexity.

\begin{algorithm}[htbp]
  \caption{\(\HMultV\bigl(\Matrix{A}^{\HMTB},(\tau,\sigma),\Matrix{u},\Matrix{f}\bigr)\)}
  \label{alg:HMTB_matrix_vector_multiplication}
  \begin{algorithmic}[1]
    \Require{HMTB matrix \(\Matrix{A}^{\HMTB}\), current node \((\tau, \sigma)\), and vectors \(\Matrix{u}\) and \(\Matrix{f}\)}
    \Ensure{\(\Matrix{f}_{\tau} \gets \Matrix{f}_{\tau} + \Matrix{A}^{\HMTB}_{\tau, \sigma} \Matrix{u}_{\sigma}\)}
    \If{\((\tau, \sigma)\) is an admissible leaf}
      \State{Let \(\Matrix{A}^{\HMTB}_{\tau, \sigma} = \Tucker\bigl(\Tensor{G}, \{\Matrix{U}_{\ell}\}_{\ell = 1}^{d}, \{\Matrix{V}_{\ell}\}_{\ell = 1}^{d}\bigr)\)}
      \State{Reshape \(\Matrix{u}_{\sigma}\) into a \(d\)-dimensional tensor \(\Tensor{U}_{\sigma}\)}
      \State{\(\Tensor{W}_{\sigma} = \Tensor{U}_{\sigma} \times_{1} \Matrix{V}_{1}^{\transpose} \times_{2} \dotsb \times_{d} \Matrix{V}_{d}^{\transpose}\)}
      \State{\(\Tensor{H}_{\tau} = \Tensor{G} \times_{[d + 1, \dotsc, 2d], [1, \dotsc, d]} \Tensor{W}_{\sigma}\)}
      \State{\(\Tensor{F}_{\tau} = \Tensor{H}_{\tau} \times_{1} \Matrix{U}_{1} \times_{2} \dotsb \times_{d} \Matrix{U}_{d}\)}
      \State{Reshape \(\Tensor{F}_{\tau}\) into a vector and add it to \(\Matrix{f}_{\tau}\)}
    \ElsIf{\((\tau, \sigma)\) is an inadmissible leaf}
      \State{Compute \(\Matrix{f}_{\tau} \gets \Matrix{f}_{\tau} + \Matrix{A}^{\HMTB}_{\tau, \sigma} \Matrix{u}_{\sigma}\)}
    \Else
      \For{\((\tau^{\prime}, \sigma^{\prime}) \in \children(\tau, \sigma)\)}
        \State{\(\HMultV\bigl(\Matrix{A}^{\HMTB}, (\tau^{\prime}, \sigma^{\prime}), \Matrix{u}, \Matrix{f}\bigr)\)}
      \EndFor
    \EndIf
  \end{algorithmic}
\end{algorithm}

\begin{proposition}[Complexity of HMTB forward evaluation]
    \label{prop:HMTB_apply_complexity}
    Under weak or strong admissibility and the corresponding assumptions of
    Proposition~\ref{prop:storage_HMTB_weak} or
    Proposition~\ref{prop:storage_HMTB_strong}, respectively, the forward
    evaluation of an HMTB matrix with Tucker bond dimensions bounded by \(p\)
    requires \(\bigO(pN\log N+p^{d}N)\) operations.
\end{proposition}

As in the preceding storage and construction comparisons, the following
comparison with \(\hierarchical\)-matrices concerns the representations
generated by the same tensor-product interpolation before possible further
algebraic recompression.
For a rank-\(p^{d}\) \(\hierarchical\)-matrix of size \(N\), the complexity
of forward evaluation is
\(\bigO(p^{d}N\log N+p^{d}N)\).
For both HMTB matrices and \(\hierarchical\)-matrices, the computations
involving the dense matrices stored at inadmissible leaves require
\(\bigO(p^{d}N)\) operations.
Similarly, the computations involving the core matrices or tensors at
admissible leaves also require \(\bigO(p^{d}N)\) operations.
The difference lies in the computations involving the factor or basis
matrices: their complexity is \(\bigO(pN\log N)\) for HMTB matrices, compared
with \(\bigO(p^{d}N\log N)\) for \(\hierarchical\)-matrices.
Therefore, for these interpolation-generated representations, HMTB matrices
have a significantly smaller prefactor in the leading \(N\log N\) term.

\section{Numerical Results} \label{sec:numerical_results}

We compare HMTB matrices with \(\hierarchical\)-matrices through a series of
numerical experiments.
Three types of kernels are considered:
\begin{itemize}
    \item Gaussian kernel:
    \begin{equation*}
        k(\Matrix{x},\Matrix{y})
        =
        \exp\left(
            -\frac{\|\Matrix{x}-\Matrix{y}\|^{2}}{2\sigma^{2}}
        \right).
    \end{equation*}
    The Gaussian kernel is smooth and nonsingular.
    In our experiments, the bandwidth is set to
    \(\sigma=\sqrt{d}\), where \(d\) denotes the spatial dimension.

    \item Single-layer potential (SLP) kernel:
    \begin{equation*}
        k(\Matrix{x},\Matrix{y})
        =
        \begin{cases}
            -\dfrac{1}{2\pi}
            \log\bigl(\|\Matrix{x}-\Matrix{y}\|\bigr),
            & d=2, \\[2mm]
            \dfrac{1}{4\pi\|\Matrix{x}-\Matrix{y}\|},
            & d=3.
        \end{cases}
    \end{equation*}
    The SLP kernel is singular on the diagonal and smooth away from it.

    \item Helmholtz kernel:
    \begin{equation*}
        k(\Matrix{x},\Matrix{y})
        =
        \begin{cases}
            \dfrac{\imag}{4}
            H_{0}^{(1)}
            \bigl(\kappa\|\Matrix{x}-\Matrix{y}\|\bigr),
            & d=2, \\[2mm]
            \dfrac{
                \exp\bigl(\imag\kappa\|\Matrix{x}-\Matrix{y}\|\bigr)
            }{
                4\pi\|\Matrix{x}-\Matrix{y}\|
            },
            & d=3,
        \end{cases}
    \end{equation*}
    where \(H_{0}^{(1)}\) denotes the Hankel function of the first kind
    of order zero and \(\kappa\) is the wavenumber.
    The Helmholtz kernel is singular on the diagonal and oscillatory away from
    it.
    In our experiments, we set \(\kappa=10\).
\end{itemize}

All algorithms are implemented in MATLAB R2023b.
The experiments were performed on a server equipped with two Intel Xeon Gold
6226R processors running at 2.90 GHz and 1000.6 GB of RAM.
The code is available at \url{https://github.com/JingyuLiuMath/HMTB}.

For each experiment, \(t_{\construct}\) and \(t_{\apply}\) denote the
construction and forward-evaluation times, respectively, measured in seconds,
and \(m_{\h}\) denotes the memory usage in GB.
The relative error of the forward evaluation is estimated on a randomly
sampled index set \(I_r\subset[N]\) by
\begin{equation*}
    e_{\apply;\rand}
    =
    \frac{
        \|\Matrix{\widetilde{f}}(I_r,:)-\Matrix{f}(I_r,:)\|_{\mathrm{F}}
    }{
        \|\Matrix{f}(I_r,:)\|_{\mathrm{F}}
    },
\end{equation*}
where \(\Matrix{f}=\Matrix{A}\Matrix{u}\) is the exact result and
\(\Matrix{\widetilde{f}}\) is the corresponding approximate forward
evaluation.

In Sections~\ref{subsec:uniform_grid_2D}
and~\ref{subsec:uniform_grid_3D},
\(\Matrix{u}\in\complex^{N\times 10}\) contains ten randomly generated test
vectors.
In Section~\ref{subsec:quasiuni_grid_2D},
\(\Matrix{u}\in\complex^N\) is obtained by sampling a prescribed smooth
function, so that the Frobenius norm in the above definition reduces to the
Euclidean norm.
For all experiments, \(I_r\) is sampled uniformly at random from \([N]\) with
\(\lvert I_r\rvert=256\).
Since the number of test vectors is fixed, evaluating the exact product only
on \(I_r\) requires \(\bigO(\lvert I_r\rvert N)\) operations and avoids
explicitly forming the dense matrix \(\Matrix{A}\), which would require
\(\bigO(N^2)\) storage.

Unless otherwise specified, we set the leaf-size threshold and interpolation
order to
\(N_0=16^2\) and \(p=8\), respectively, for the 2D problems, and to
\(N_0=4^3\) and \(p=4\), respectively, for the 3D problems.
The compression tolerance is set to \(\varepsilon=10^{-3}\) throughout.
Based on the numerical investigation in~\ref{subsec:explore_tensor_low_rank_kernel_functions},
we use weak admissibility for the Gaussian kernel and strong admissibility
for the SLP and Helmholtz kernels.

For a consistent comparison, whenever a ratio is reported in the following
tables, it is defined as the corresponding \(\hierarchical\)-matrix quantity
divided by the HMTB quantity:
\begin{equation*}
    r_{\construct}
    =
    \frac{t_{\construct}^{\mathrm{H}}}
         {t_{\construct}^{\mathrm{HMTB}}},
    \qquad
    r_{\memory}
    =
    \frac{m_{\h}^{\mathrm{H}}}
         {m_{\h}^{\mathrm{HMTB}}},
    \qquad
    r_{\apply}
    =
    \frac{t_{\apply}^{\mathrm{H}}}
         {t_{\apply}^{\mathrm{HMTB}}}.
\end{equation*}
Thus, a ratio greater than one indicates an advantage of HMTB matrices for the
corresponding quantity, whereas a ratio below one indicates that the
\(\hierarchical\)-matrix is less expensive.

Our numerical experiments are designed to compare the two representations at
comparable levels of accuracy.
For both HMTB matrices and \(\hierarchical\)-matrices, admissible blocks are
initially approximated using the same interpolation order \(p\), and the same
compression tolerance is used in the subsequent compression step.
We focus on three aspects: construction time, memory usage, and
forward-evaluation time.
The dependence of the approximation accuracy on the interpolation order is
further investigated in~\ref{subsec:explore_tensor_low_rank_kernel_functions}.

Sections~\ref{subsec:uniform_grid_2D}
and~\ref{subsec:uniform_grid_3D}
present results on uniform grids in two and three dimensions, respectively.
Section~\ref{subsec:quasiuni_grid_2D} then illustrates the application of
HMTB matrices to quasi-uniform grids.

\begin{remark}
The construction-time comparison is specific to the interpolation-based
implementations used here.
The initial rank bound \(p^d\) for an admissible
\(\hierarchical\)-matrix block can substantially exceed its numerical
rank at the prescribed accuracy.
In our implementation, recompression requires QR factorizations of the
initial basis matrices and an SVD of the resulting
\(p^d\times p^d\) core matrix.
The latter costs \(\bigO(p^{3d})\), in addition to the basis
orthogonalization and update costs.
Although truncation reduces the final storage requirement, it does not
eliminate the work already incurred using the initial matrix rank bound
\(p^d\).

Alternative construction strategies include adaptive cross approximation
(ACA)~\cite{Bebendorf_2008_book}, randomized hierarchical compression
using applications of the operator and its
adjoint~\cite{Martinsson_Tropp_2020}, and proxy-surface methods for
suitable kernels~\cite{Martinsson_2019_book}.
These methods can construct approximations whose ranks reflect the
numerical rank more closely, without first forming the full
tensor-product interpolation representation.

Consequently, the observed construction-time differences reflect both
the representations and the chosen construction procedures.
They should not be interpreted as an intrinsic advantage of HMTB matrices
over all \(\hierarchical\)-matrix construction methods.
The memory savings reported here are measured against the recompressed
\(\hierarchical\)-matrix implementation used in these experiments;
comparisons with the alternative construction strategies above are
outside the scope of the present numerical study.
\end{remark}

\subsection{Uniform Grid in Two Dimensions}
\label{subsec:uniform_grid_2D}

This section presents numerical experiments on two-dimensional uniform grids.
Let \(\Omega=[0,1]^2\) be the unit square and
\(a(\Matrix{x})\equiv0\) in~\eqref{eq:IE}.

\begin{figure}[htbp]
    \centering
    \begin{subfigure}{0.3\textwidth}
        \centering
        \includegraphics[width=\linewidth]{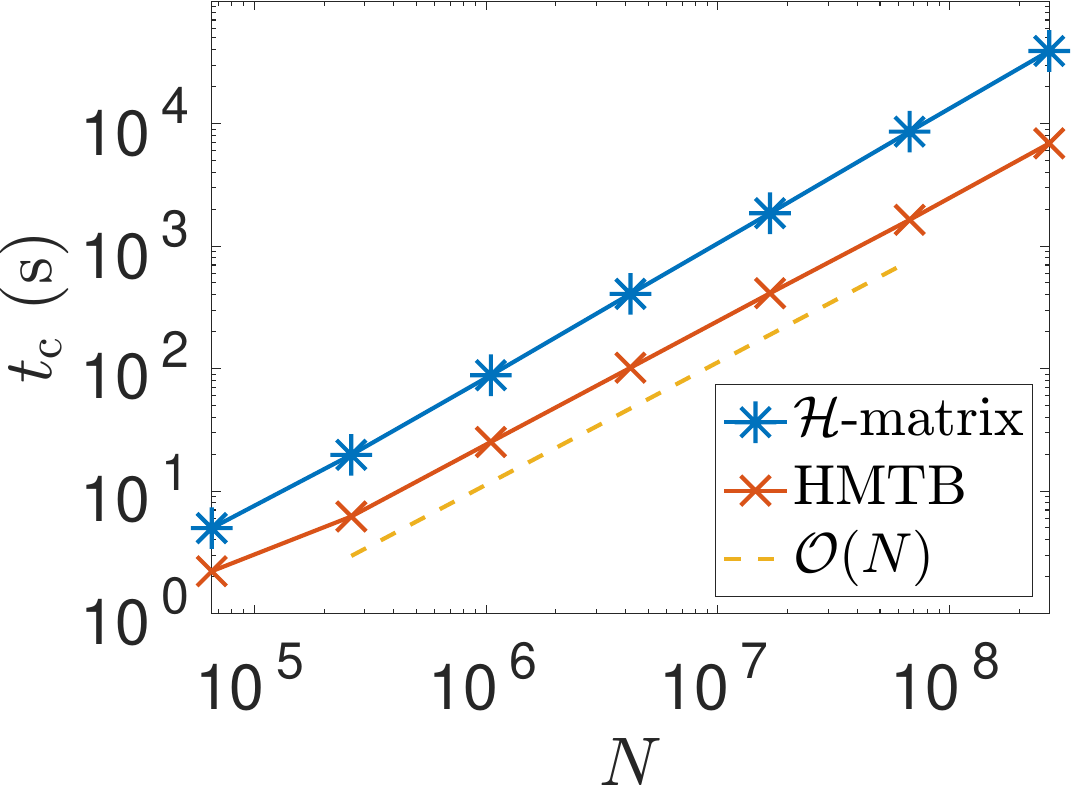}
        \caption{Construction time.}
        \label{fig:H2DW_construct_time}
    \end{subfigure}
    \hspace{0.01\textwidth}
    \begin{subfigure}{0.3\textwidth}
        \centering
        \includegraphics[width=\linewidth]{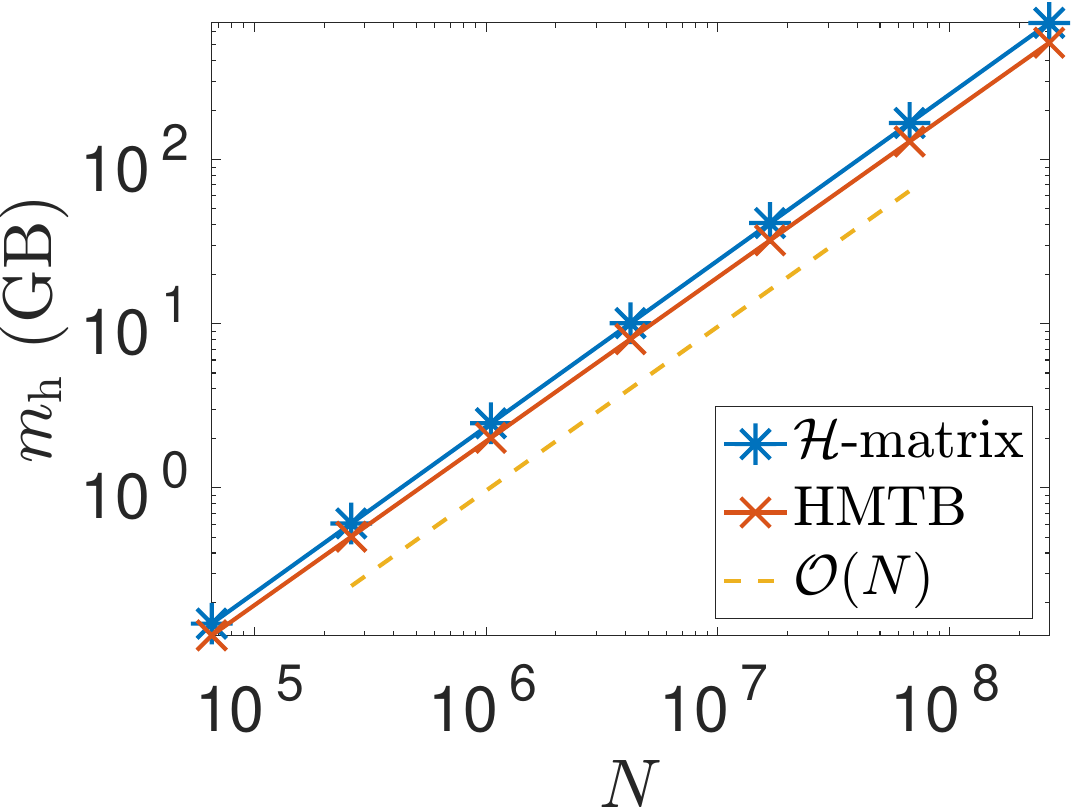}
        \caption{Memory.}
        \label{fig:H2DW_memory}
    \end{subfigure}
    \hspace{0.01\textwidth}
    \begin{subfigure}{0.3\textwidth}
        \centering
        \includegraphics[width=\linewidth]{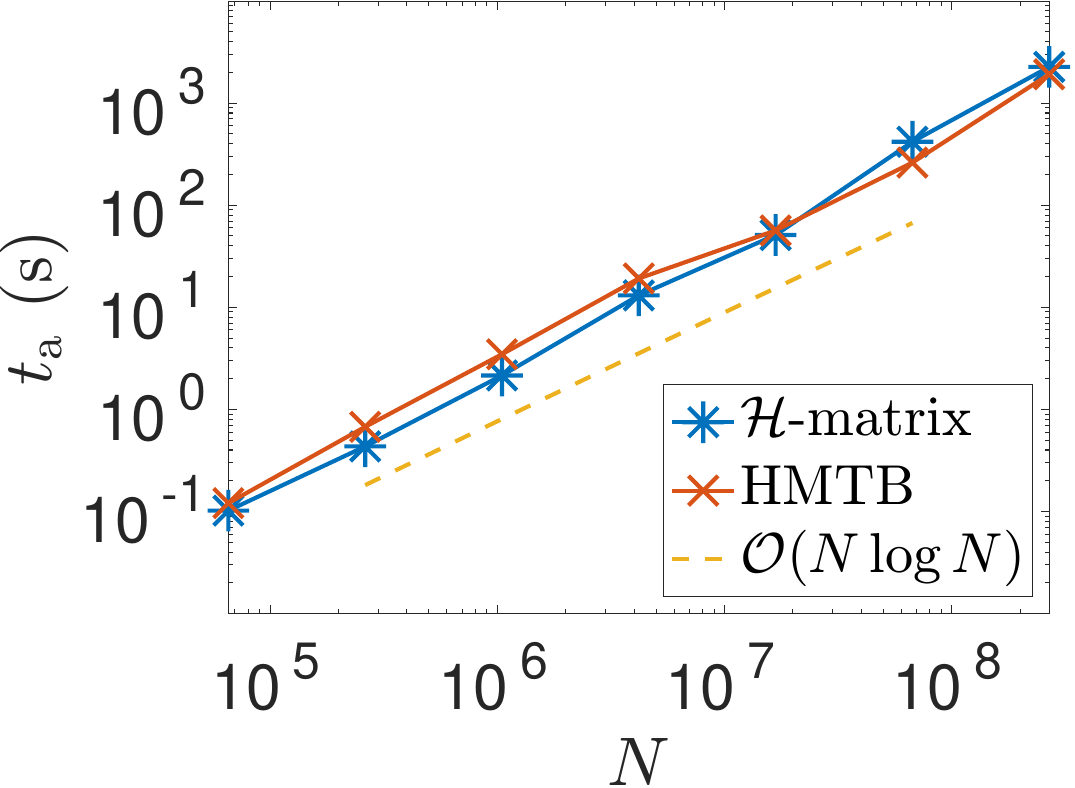}
        \caption{Forward time.}
        \label{fig:H2DW_hmultv_time}
    \end{subfigure}
    \hspace{0.01\textwidth}
    \caption{Time and memory costs of \(\hierarchical\)-matrices and HMTB matrices for the Gaussian kernel under weak
    admissibility in \(2\)D.}
    \label{fig:H2DW}
\end{figure}
\begin{table}[htbp]
    \centering
    \begin{tabular}{cccccc}
        \toprule
        \(N\) & & \(t_{\construct}\) (s) & \(m_{\h}\) (GB) & \(t_{\apply}\) (s) & \(e_{\apply; \rand}\) \\ 
        \midrule
        \multirow{3}{*}{\(256^2\)} & $\mathcal{H}$-matrix & 5.0e+00 & 1.5e-01 & 1.0e-01 & 3.8e-04 \\ 
        & HMTB & 2.2e+00 & 1.3e-01 & 1.2e-01 & 3.5e-04 \\ 
        & Ratio & \(2.2 \times\) & \(1.2 \times\) & \(0.8 \times\) & - \\ 
        \midrule
        \multirow{3}{*}{\(512^2\)} & $\mathcal{H}$-matrix & 2.0e+01 & 6.1e-01 & 4.3e-01 & 3.9e-04 \\ 
        & HMTB & 6.2e+00 & 5.0e-01 & 6.8e-01 & 3.6e-04 \\ 
        & Ratio & \(3.2 \times\) & \(1.2 \times\) & \(0.6 \times\) & - \\ 
        \midrule
        \multirow{3}{*}{\(1024^2\)} & $\mathcal{H}$-matrix & 8.8e+01 & 2.5e+00 & 2.1e+00 & 3.6e-04 \\ 
        & HMTB & 2.5e+01 & 2.0e+00 & 3.5e+00 & 3.2e-04 \\ 
        & Ratio & \(3.5 \times\) & \(1.2 \times\) & \(0.6 \times\) & - \\ 
        \midrule
        \multirow{3}{*}{\(2048^2\)} & $\mathcal{H}$-matrix & 4.1e+02 & 1.0e+01 & 1.3e+01 & 2.6e-04 \\ 
        & HMTB & 1.0e+02 & 8.0e+00 & 1.9e+01 & 2.3e-04 \\ 
        & Ratio & \(4.0 \times\) & \(1.3 \times\) & \(0.7 \times\) & - \\ 
        \midrule
        \multirow{3}{*}{\(4096^2\)} & $\mathcal{H}$-matrix & 1.9e+03 & 4.1e+01 & 5.1e+01 & 2.7e-04 \\ 
        & HMTB & 4.1e+02 & 3.2e+01 & 5.7e+01 & 2.6e-04 \\ 
        & Ratio & \(4.5 \times\) & \(1.3 \times\) & \(0.9 \times\) & - \\ 
        \midrule
        \multirow{3}{*}{\(8192^2\)} & $\mathcal{H}$-matrix & 8.6e+03 & 1.7e+02 & 4.2e+02 & 2.7e-04 \\ 
        & HMTB & 1.6e+03 & 1.3e+02 & 2.6e+02 & 2.4e-04 \\ 
        & Ratio & \(5.2 \times\) & \(1.3 \times\) & \(1.6 \times\) & - \\ 
        \midrule
        \multirow{3}{*}{\(16384^2\)} & $\mathcal{H}$-matrix & 3.9e+04 & 6.8e+02 & 2.3e+03 & 2.5e-04 \\ 
        & HMTB & 6.9e+03 & 5.2e+02 & 1.9e+03 & 2.3e-04 \\ 
        & Ratio & \(5.7 \times\) & \(1.3 \times\) & \(1.2 \times\) & - \\ 
        \bottomrule
    \end{tabular}
    \caption{Numerical results of $\mathcal{H}$-matrices and HMTB for the Gaussian kernel
    under weak admissibility in \(2\)D.} \label{tab:H2DW}
\end{table}

We first consider the Gaussian kernel under weak admissibility.
The problem is discretized using \(n\) points in each coordinate direction,
with \(n\) ranging from \(256\) to \(16384\).
Figure~\ref{fig:H2DW} shows the construction time, memory usage, and
forward-evaluation time of HMTB matrices and \(\hierarchical\)-matrices.
Both the construction time \(t_{\construct}\) and the memory usage \(m_{\h}\)
exhibit the predicted \(\bigO(N)\) scaling, while the forward-evaluation time
is consistent with the \(\bigO(N\log N)\) complexity for sufficiently large
problem sizes.
Detailed results are reported in Table~\ref{tab:H2DW}.

Both representations achieve comparable accuracy, with
\(e_{\apply;\rand}\) on the order of \(10^{-4}\).
Compared with \(\hierarchical\)-matrices, HMTB matrices reduce the construction
time by a factor of approximately \(2.2\) to \(5.7\) and the memory usage by
a factor of approximately \(1.2\) to \(1.3\).
For smaller problem sizes, the forward evaluation of HMTB matrices is somewhat
more expensive than that of \(\hierarchical\)-matrices.
As the problem size increases, however, the difference decreases, and HMTB
becomes faster for the two largest problem sizes considered.

\begin{figure}[htbp]
    \centering
    \begin{subfigure}{0.3\textwidth}
        \centering
        \includegraphics[width=\linewidth]{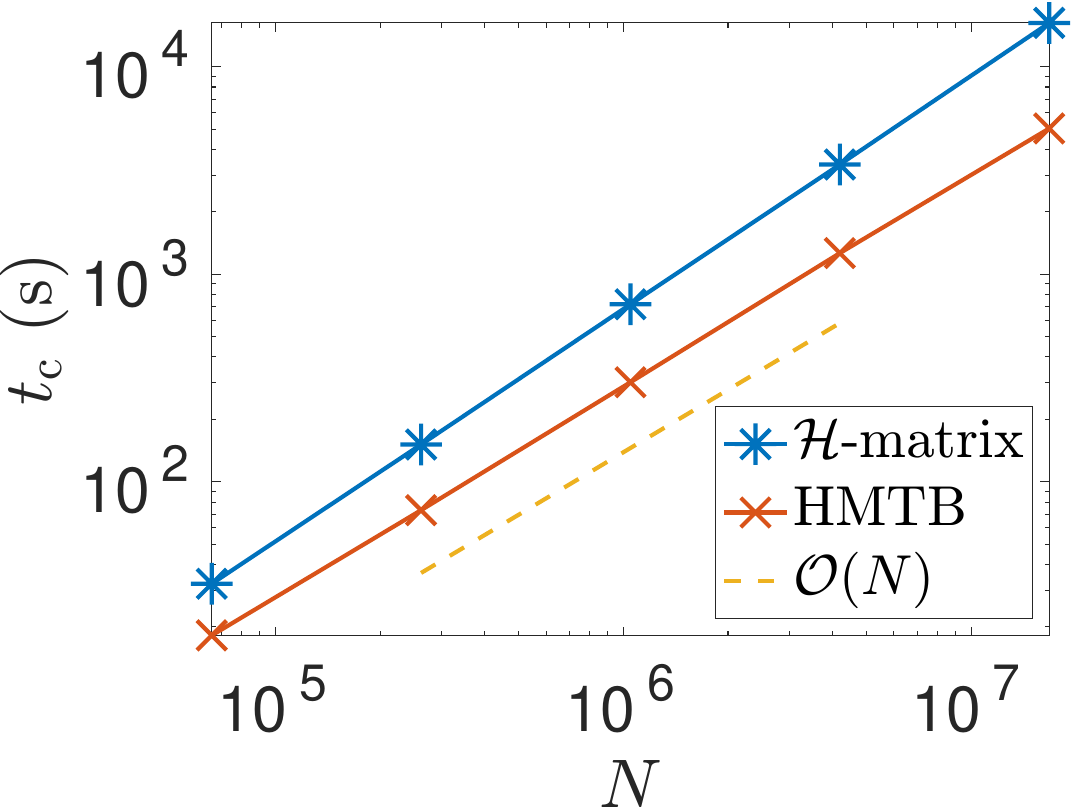}
        \caption{Construction time.}
        \label{fig:H2DS_construct_time}
    \end{subfigure}
    \hspace{0.01\textwidth}
    \begin{subfigure}{0.3\textwidth}
        \centering
        \includegraphics[width=\linewidth]{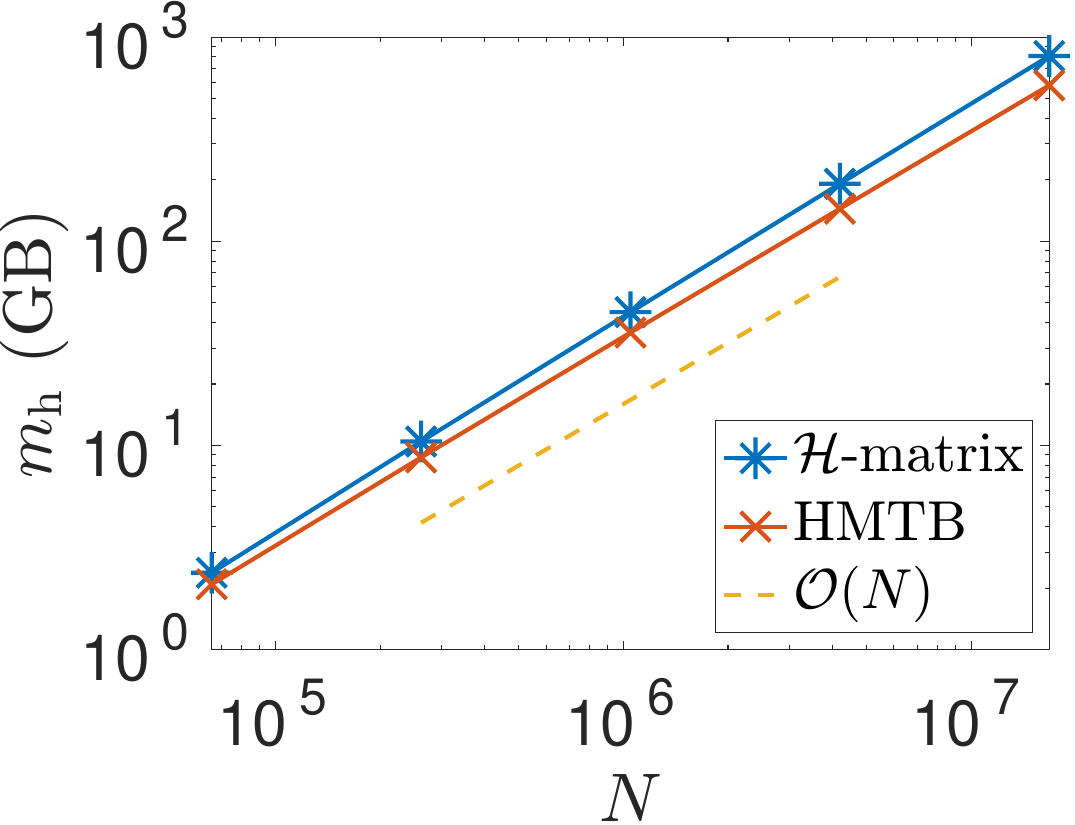}
        \caption{Memory.}
        \label{fig:H2DS_memory}
    \end{subfigure}
    \hspace{0.01\textwidth}
    \begin{subfigure}{0.3\textwidth}
        \centering
        \includegraphics[width=\linewidth]{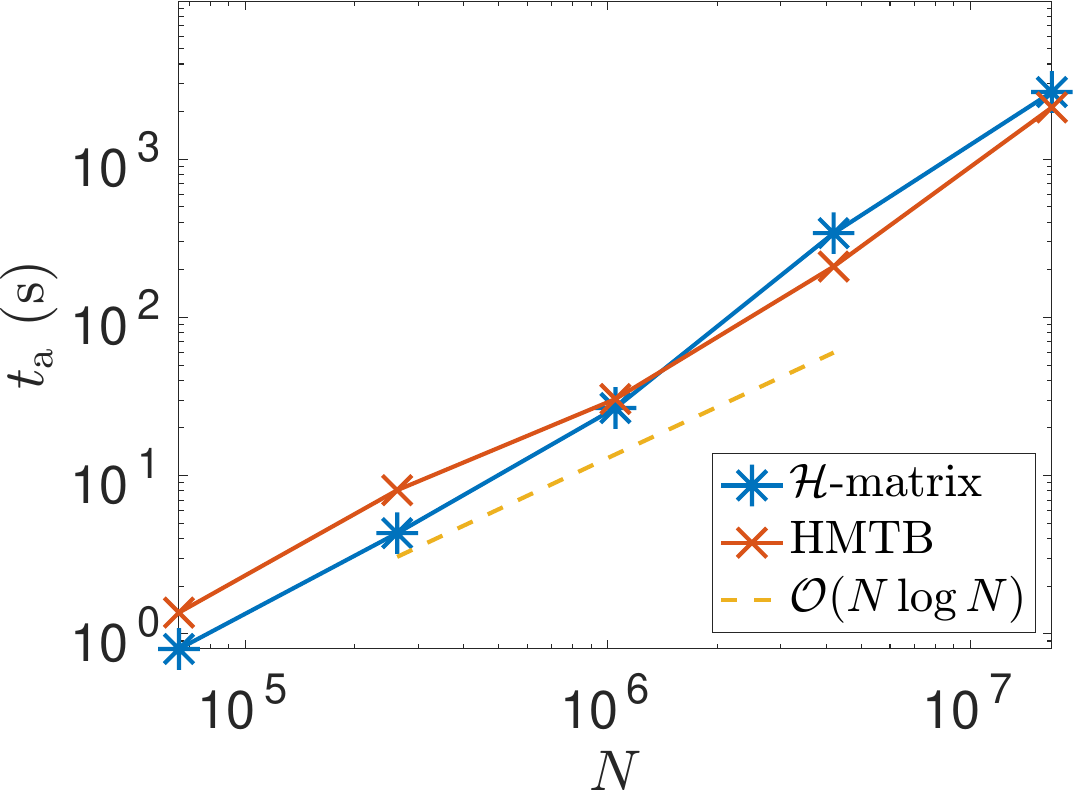}
        \caption{Forward time.}
        \label{fig:H2DS_hmultv_time}
    \end{subfigure}
    \hspace{0.01\textwidth}
    \caption{Time and memory costs of \(\hierarchical\)-matrices and HMTB matrices for the Helmholtz kernel under strong
    admissibility in \(2\)D.}

    \label{fig:H2DS}
\end{figure}
\begin{table}[htbp]
    \centering
    \begin{tabular}{cccccc}
        \toprule
        \(N\) & & \(t_{\construct}\) (s) & \(m_{\h}\) (GB) & \(t_{\apply}\) (s) & \(e_{\apply; \rand}\) \\ 
        \midrule
        \multirow{3}{*}{\(256^2\)} & $\mathcal{H}$-matrix & 3.2e+01 & 2.4e+00 & 8.0e-01 & 4.8e-04 \\ 
        & HMTB & 1.8e+01 & 2.1e+00 & 1.4e+00 & 6.7e-04 \\ 
        & Ratio & \(1.8 \times\) & \(1.1 \times\) & \(0.6 \times\) & - \\ 
        \midrule
        \multirow{3}{*}{\(512^2\)} & $\mathcal{H}$-matrix & 1.5e+02 & 1.0e+01 & 4.3e+00 & 5.0e-04 \\ 
        & HMTB & 7.3e+01 & 8.7e+00 & 8.1e+00 & 7.1e-04 \\ 
        & Ratio & \(2.1 \times\) & \(1.2 \times\) & \(0.5 \times\) & - \\ 
        \midrule
        \multirow{3}{*}{\(1024^2\)} & $\mathcal{H}$-matrix & 7.2e+02 & 4.5e+01 & 2.7e+01 & 5.1e-04 \\ 
        & HMTB & 3.0e+02 & 3.6e+01 & 3.0e+01 & 7.4e-04 \\ 
        & Ratio & \(2.4 \times\) & \(1.3 \times\) & \(0.9 \times\) & - \\ 
        \midrule
        \multirow{3}{*}{\(2048^2\)} & $\mathcal{H}$-matrix & 3.4e+03 & 1.9e+02 & 3.4e+02 & 4.9e-04 \\ 
        & HMTB & 1.3e+03 & 1.4e+02 & 2.1e+02 & 7.3e-04 \\ 
        & Ratio & \(2.7 \times\) & \(1.3 \times\) & \(1.6 \times\) & - \\ 
        \midrule
        \multirow{3}{*}{\(4096^2\)} & $\mathcal{H}$-matrix & 1.6e+04 & 8.1e+02 & 2.7e+03 & 5.8e-04 \\ 
        & HMTB & 5.1e+03 & 5.8e+02 & 2.1e+03 & 8.6e-04 \\ 
        & Ratio & \(3.2 \times\) & \(1.4 \times\) & \(1.3 \times\) & - \\ 
        \bottomrule
    \end{tabular}
    \caption{Numerical results of $\mathcal{H}$-matrices and HMTB for the Helmholtz kernel
    under strong admissibility in \(2\)D.} \label{tab:H2DS}
\end{table}

Next, we consider the Helmholtz kernel under strong admissibility.
In this experiment, \(n\) ranges from \(256\) to \(4096\).
Figure~\ref{fig:H2DS} compares the construction time, memory usage, and
forward-evaluation time of \(\hierarchical\)-matrices and HMTB matrices, while
detailed results are reported in Table~\ref{tab:H2DS}.
Both representations achieve comparable accuracy, with
\(e_{\apply;\rand}\) on the order of \(10^{-4}\).
The observed scaling of the construction time, memory usage, and
forward-evaluation time is consistent with the theoretical complexity estimates.

Compared with \(\hierarchical\)-matrices, HMTB matrices reduce the construction
time by a factor of approximately \(1.8\) to \(3.2\) and the memory usage by
a factor of approximately \(1.1\) to \(1.4\).
For smaller problem sizes, the forward evaluation of HMTB matrices is more
expensive than that of \(\hierarchical\)-matrices.
As \(N\) increases, this difference diminishes, and HMTB becomes faster for the
two largest problem sizes considered.

\subsection{Uniform Grid in Three Dimensions} \label{subsec:uniform_grid_3D}

We next consider the three-dimensional analogue of the experiments in
Section~\ref{subsec:uniform_grid_2D}.
The computational domain is \(\Omega=[0,1]^3\), with
\(a(\Matrix{x})\equiv0\) in~\eqref{eq:IE}.

\begin{figure}[htbp]
    \centering
    \begin{subfigure}{0.3\textwidth}
        \centering
        \includegraphics[width=\linewidth]{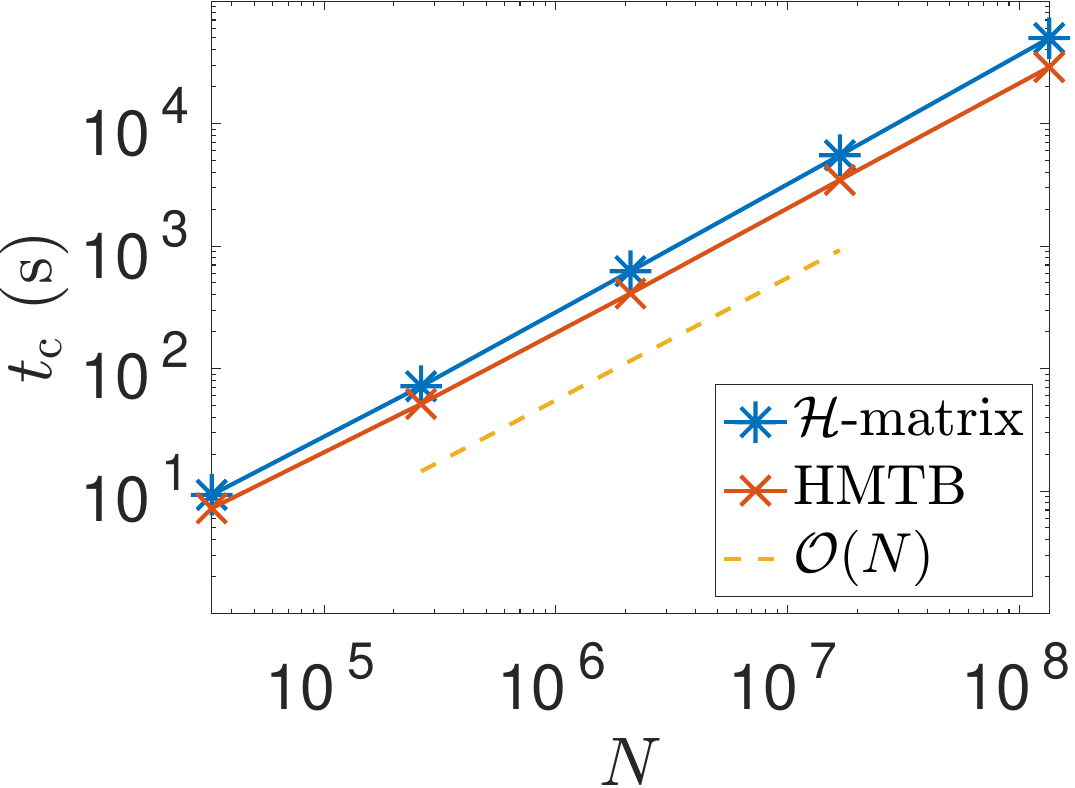}
        \caption{Construction time.}
        \label{fig:H3DW_construct_time}
    \end{subfigure}
    \hspace{0.01\textwidth}
    \begin{subfigure}{0.3\textwidth}
        \centering
        \includegraphics[width=\linewidth]{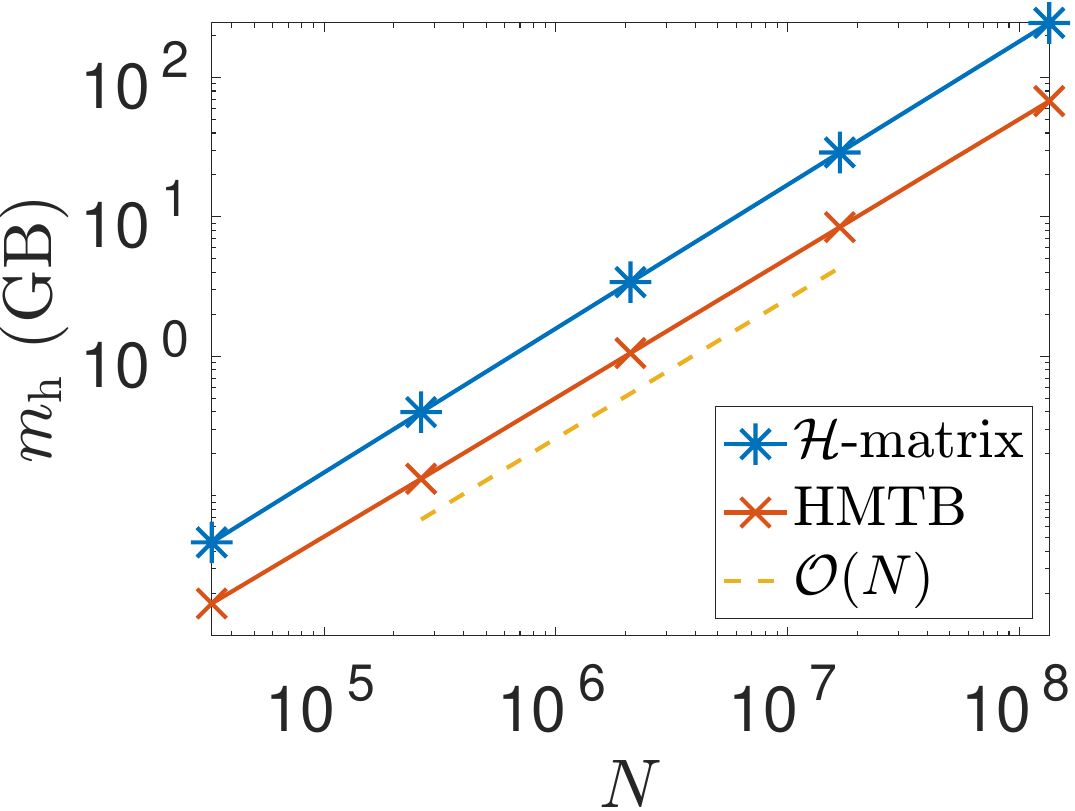}
        \caption{Memory.}
        \label{fig:H3DW_memory}
    \end{subfigure}
    \hspace{0.01\textwidth}
    \begin{subfigure}{0.3\textwidth}
        \centering
        \includegraphics[width=\linewidth]{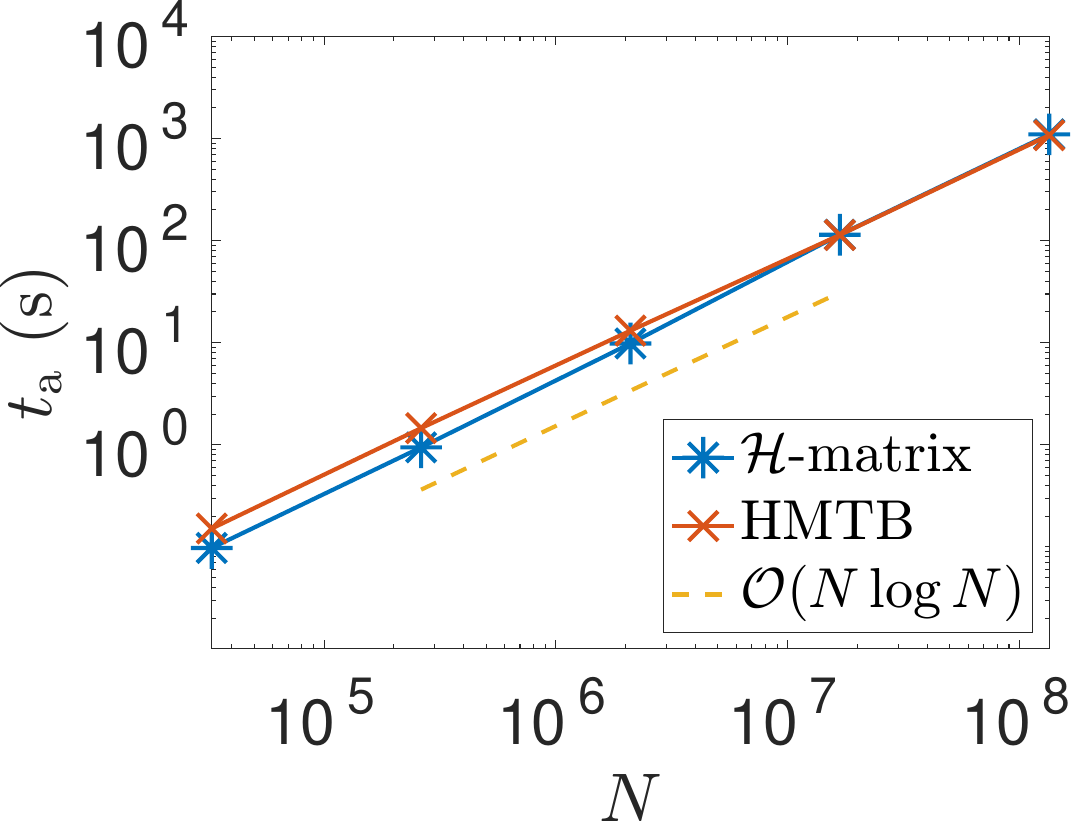}
        \caption{Forward time.}
        \label{fig:H3DW_hmultv_time}
    \end{subfigure}
    \hspace{0.01\textwidth}
    \caption{Time and memory costs of \(\hierarchical\)-matrices and HMTB for the Gaussian kernel under weak
    admissibility in \(3\)D.}
    \label{fig:H3DW}
\end{figure}
\begin{table}[htbp]
    \centering
    \begin{tabular}{cccccc}
        \toprule
        \(N\) & & \(t_{\construct}\) (s) & \(m_{\h}\) (GB) & \(t_{\apply}\) (s) & \(e_{\apply; \rand}\) \\ 
        \midrule
        \multirow{3}{*}{\(32^3\)} & $\mathcal{H}$-matrix & 9.3e+00 & 4.6e-02 & 9.8e-02 & 1.6e-04 \\ 
        & HMTB & 7.2e+00 & 1.7e-02 & 1.5e-01 & 1.3e-04 \\ 
        & Ratio & \(1.3 \times\) & \(2.7 \times\) & \(0.6 \times\) & - \\ 
        \midrule
        \multirow{3}{*}{\(64^3\)} & $\mathcal{H}$-matrix & 7.2e+01 & 4.0e-01 & 9.4e-01 & 1.7e-04 \\ 
        & HMTB & 5.1e+01 & 1.3e-01 & 1.5e+00 & 1.4e-04 \\ 
        & Ratio & \(1.4 \times\) & \(3.0 \times\) & \(0.6 \times\) & - \\ 
        \midrule
        \multirow{3}{*}{\(128^3\)} & $\mathcal{H}$-matrix & 6.2e+02 & 3.4e+00 & 9.8e+00 & 1.4e-04 \\ 
        & HMTB & 4.1e+02 & 1.1e+00 & 1.3e+01 & 1.1e-04 \\ 
        & Ratio & \(1.5 \times\) & \(3.2 \times\) & \(0.7 \times\) & - \\ 
        \midrule
        \multirow{3}{*}{\(256^3\)} & $\mathcal{H}$-matrix & 5.5e+03 & 2.9e+01 & 1.1e+02 & 2.1e-04 \\ 
        & HMTB & 3.5e+03 & 8.5e+00 & 1.1e+02 & 1.6e-04 \\ 
        & Ratio & \(1.6 \times\) & \(3.4 \times\) & \(1.0 \times\) & - \\ 
        \midrule
        \multirow{3}{*}{\(512^3\)} & $\mathcal{H}$-matrix & 5.0e+04 & 2.5e+02 & 1.1e+03 & 1.4e-04 \\ 
        & HMTB & 2.9e+04 & 6.8e+01 & 1.1e+03 & 1.1e-04 \\ 
        & Ratio & \(1.7 \times\) & \(3.6 \times\) & \(1.0 \times\) & - \\ 
        \bottomrule
    \end{tabular}
    \caption{Numerical results of $\mathcal{H}$-matrices and HMTB for the Gaussian kernel
    under weak admissibility in \(3\)D.} \label{tab:H3DW}
\end{table}

For the Gaussian kernel, \(n\) ranges from \(32\) to \(512\).
Figure~\ref{fig:H3DW} shows the construction time, memory usage, and
forward-evaluation time, with detailed results reported in
Table~\ref{tab:H3DW}.
The observed construction time and memory usage of both
\(\hierarchical\)-matrices and HMTB matrices scale linearly with \(N\),
while the forward-evaluation time is consistent with the
\(\bigO(N\log N)\) complexity.

Compared with \(\hierarchical\)-matrices, HMTB matrices reduce the construction
time by a factor of approximately \(1.3\) to \(1.7\) and the memory usage by
a factor of approximately \(2.7\) to \(3.6\).
In particular, for \(N=512^3\), the HMTB representation requires approximately
\(68\) GB of memory, compared with approximately \(250\) GB for the
corresponding \(\hierarchical\)-matrix.
For smaller problem sizes, the forward evaluation of HMTB matrices is somewhat
more expensive, but the difference decreases as \(N\) increases, and the two
methods have comparable forward-evaluation times for the largest problem sizes
considered.
Both representations achieve comparable accuracy, with
\(e_{\apply;\rand}\) on the order of \(10^{-4}\).

\begin{figure}[htbp]
    \centering
    \begin{subfigure}{0.3\textwidth}
        \centering
        \includegraphics[width=\linewidth]{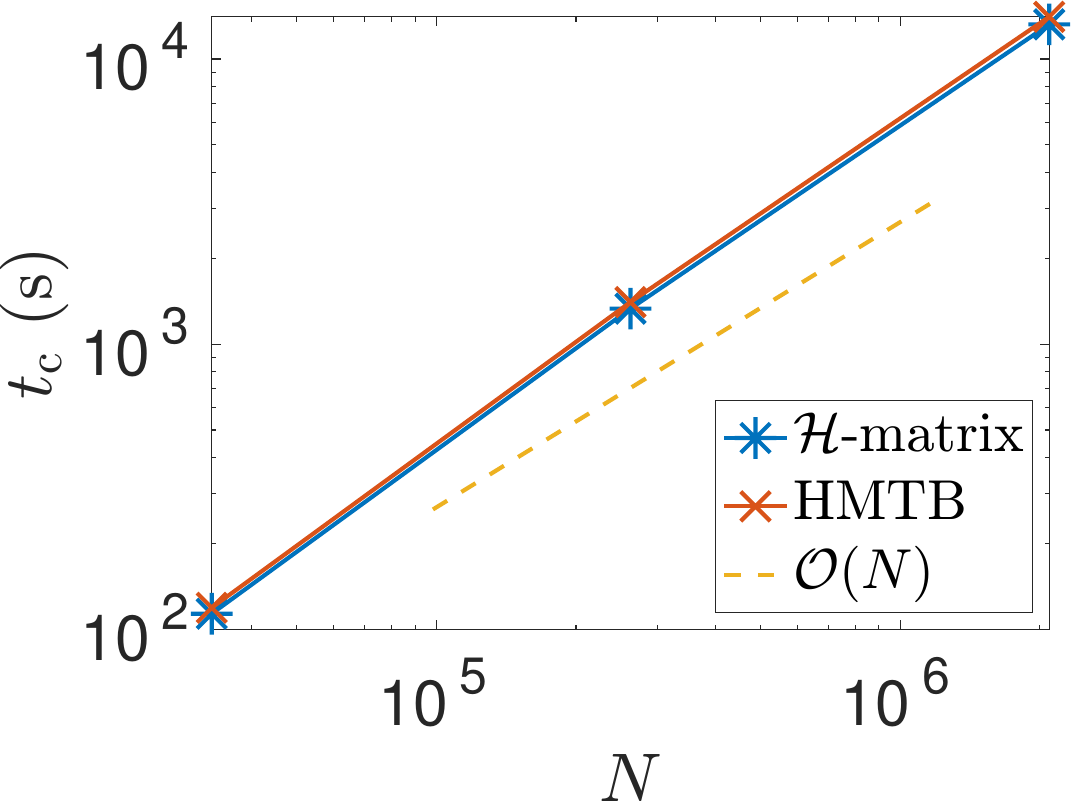}
        \caption{Construction time.}
        \label{fig:H3DS_construct_time}
    \end{subfigure}
    \hspace{0.01\textwidth}
    \begin{subfigure}{0.3\textwidth}
        \centering
        \includegraphics[width=\linewidth]{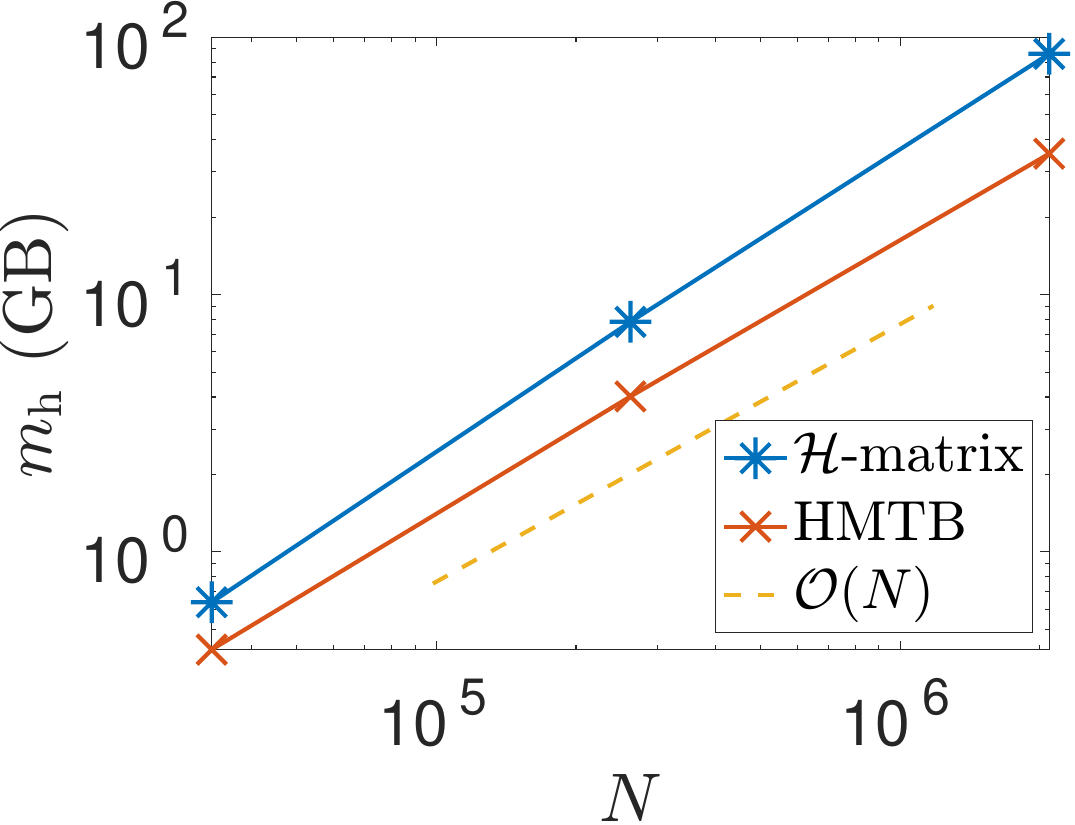}
        \caption{Memory.}
        \label{fig:H3DS_memory}
    \end{subfigure}
    \hspace{0.01\textwidth}
    \begin{subfigure}{0.3\textwidth}
        \centering
        \includegraphics[width=\linewidth]{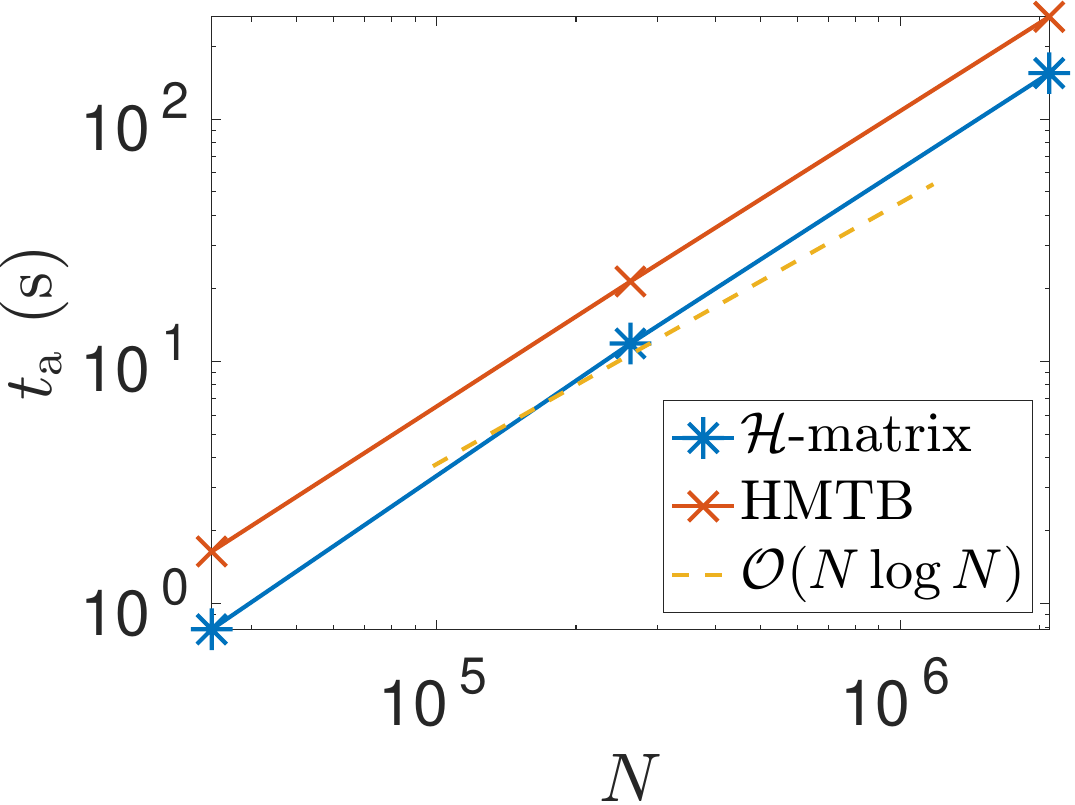}
        \caption{Forward time.}
        \label{fig:H3DS_hmultv_time}
    \end{subfigure}
    \hspace{0.01\textwidth}
    \caption{Time and memory costs of \(\hierarchical\)-matrices and HMTB matrices for the SLP kernel under strong
    admissibility in \(3\)D.}
    \label{fig:H3DS}
\end{figure}
\begin{table}[htbp]
    \centering
    \begin{tabular}{cccccc}
        \toprule
        \(N\) & & \(t_{\construct}\) (s) & \(m_{\h}\) (GB) & \(t_{\apply}\) (s) & \(e_{\apply; \rand}\) \\ 
        \midrule
        \multirow{3}{*}{\(32^3\)} & $\mathcal{H}$-matrix & 1.1e+02 & 6.4e-01 & 7.8e-01 & 8.6e-04 \\ 
        & HMTB & 1.2e+02 & 4.2e-01 & 1.6e+00 & 9.5e-04 \\ 
        & Ratio & \(1.0 \times\) & \(1.5 \times\) & \(0.5 \times\) & - \\ 
        \midrule
        \multirow{3}{*}{\(64^3\)} & $\mathcal{H}$-matrix & 1.3e+03 & 7.8e+00 & 1.2e+01 & 1.1e-03 \\ 
        & HMTB & 1.4e+03 & 4.0e+00 & 2.1e+01 & 1.1e-03 \\ 
        & Ratio & \(0.9 \times\) & \(1.9 \times\) & \(0.6 \times\) & - \\ 
        \midrule
        \multirow{3}{*}{\(128^3\)} & $\mathcal{H}$-matrix & 1.3e+04 & 8.6e+01 & 1.5e+02 & 9.8e-04 \\ 
        & HMTB & 1.4e+04 & 3.5e+01 & 2.6e+02 & 1.0e-03 \\ 
        & Ratio & \(0.9 \times\) & \(2.4 \times\) & \(0.6 \times\) & - \\ 
        \bottomrule
    \end{tabular}
    \caption{Numerical results of $\mathcal{H}$-matrices and HMTB for the single-layer potential kernel
    under strong admissibility in \(3\)D.} \label{tab:H3DS}
\end{table}

For the SLP kernel, we discretize the problem using
\(n=32\), \(64\), and \(128\) points in each coordinate direction.
The corresponding results are summarized in
Figure~\ref{fig:H3DS} and Table~\ref{tab:H3DS}.
The observed scaling of the construction time, memory usage, and
forward-evaluation time is consistent with the theoretical complexity
estimates.
Both \(\hierarchical\)-matrices and HMTB matrices achieve comparable accuracy,
with \(e_{\apply;\rand}\) on the order of \(10^{-3}\).

In this case, the construction costs of the two representations are comparable,
with HMTB matrices being slightly more expensive for the larger problem sizes.
The forward evaluation of HMTB matrices is also more expensive than that of
\(\hierarchical\)-matrices, by a factor of approximately \(1.7\) to \(2.1\).
This behavior is associated with the strong admissibility condition, under
which a larger fraction of the matrix is represented by dense inadmissible
blocks, reducing the computational advantage obtained from the Tucker
representation of admissible blocks.
Nevertheless, HMTB matrices reduce the memory usage by a factor of approximately
\(1.5\) to \(2.4\), with the memory advantage becoming more pronounced as the
problem size increases.

\subsection{Quasi-Uniform Grid in Two Dimensions} \label{subsec:quasiuni_grid_2D}

In this section, we present two examples to illustrate the application of HMTB matrices on quasi-uniform grids, considering the Gaussian and SLP kernels on the two-dimensional domain \(\Omega=[0,1]^2\).
The quasi-uniform grid is obtained from a triangulation of \(\Omega\), with each discretization point \(\Matrix{x}_i\) chosen as the centroid of the corresponding triangular cell; see Figure~\ref{fig:triangulation}.
We consider matrix sizes
\begin{equation*}
N = 8192,\ 32768,\ 131072,\ \text{and}\ 524288.
\end{equation*}
Let \(M\) denote the number of points on the auxiliary uniform grid, and let \(\Matrix{A}\) and \(\Matrix{\hat{A}}\) denote the matrices associated with the quasi-uniform and uniform grids, respectively, as described in Section~\ref{subsec:applications_of_HMTB_matrices_on_the_quasi_uniform_grid}.
The matrix \(\Matrix{\hat{A}}\) is represented by its HMTB approximation.
We define the oversampling ratio by
\begin{equation*}
  \rho = \sqrt{\frac{2M}{N}},  
\end{equation*}
where the factor \(2\) accounts for the fact that each square cell of the uniform grid corresponds to two triangular cells in the quasi-uniform triangulation.
The input vector \(\Matrix{u}\) is obtained by sampling the smooth function
\begin{equation*}
    u(\Matrix{x})
    = u(x_1,x_2)
    = 1
    + 0.5\e^{-(x_1-0.3)^2-(x_2-0.6)^2}
    + \sin(5x_1x_2)
\end{equation*}
at the quasi-uniform grid points.

\begin{figure}[htbp]
    \centering
    \begin{subfigure}{0.35\textwidth}
        \centering
        \includegraphics[width=\linewidth]{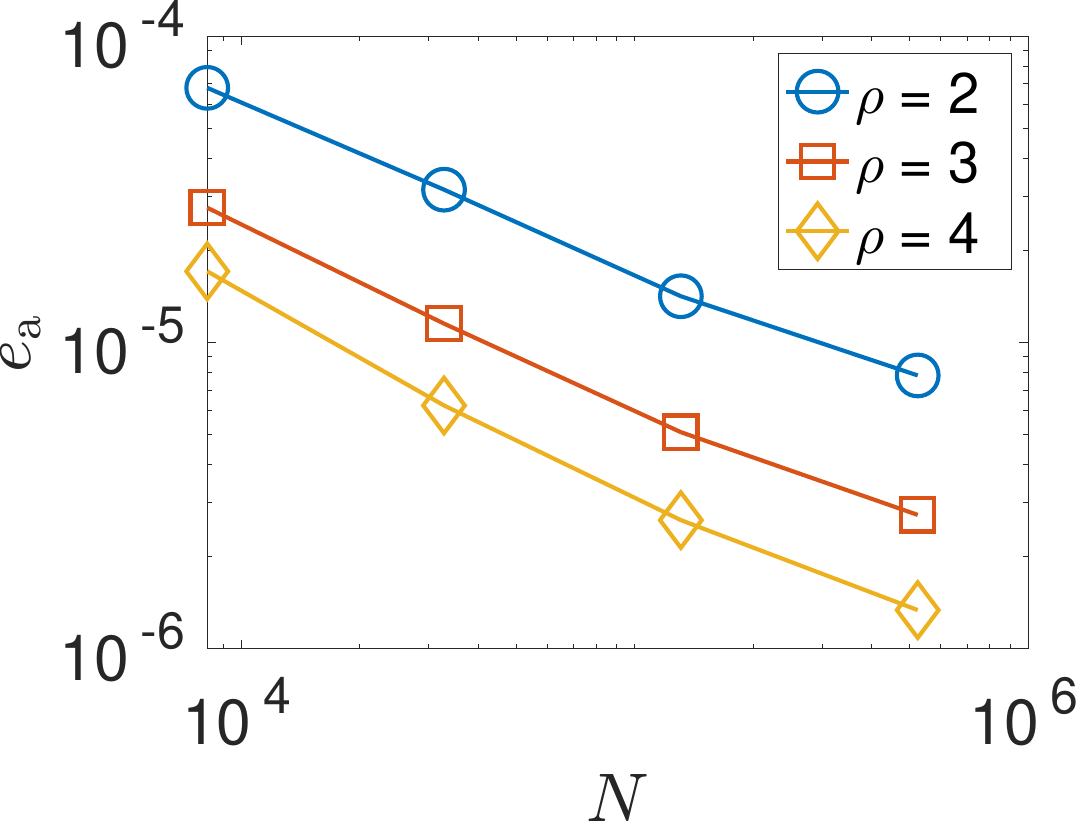}
        \caption{Gaussian kernel.}
        \label{fig:H2DWQ_err}
    \end{subfigure}
    \hspace{0.01\textwidth}
    \begin{subfigure}{0.35\textwidth}
        \centering
        \includegraphics[width=\linewidth]{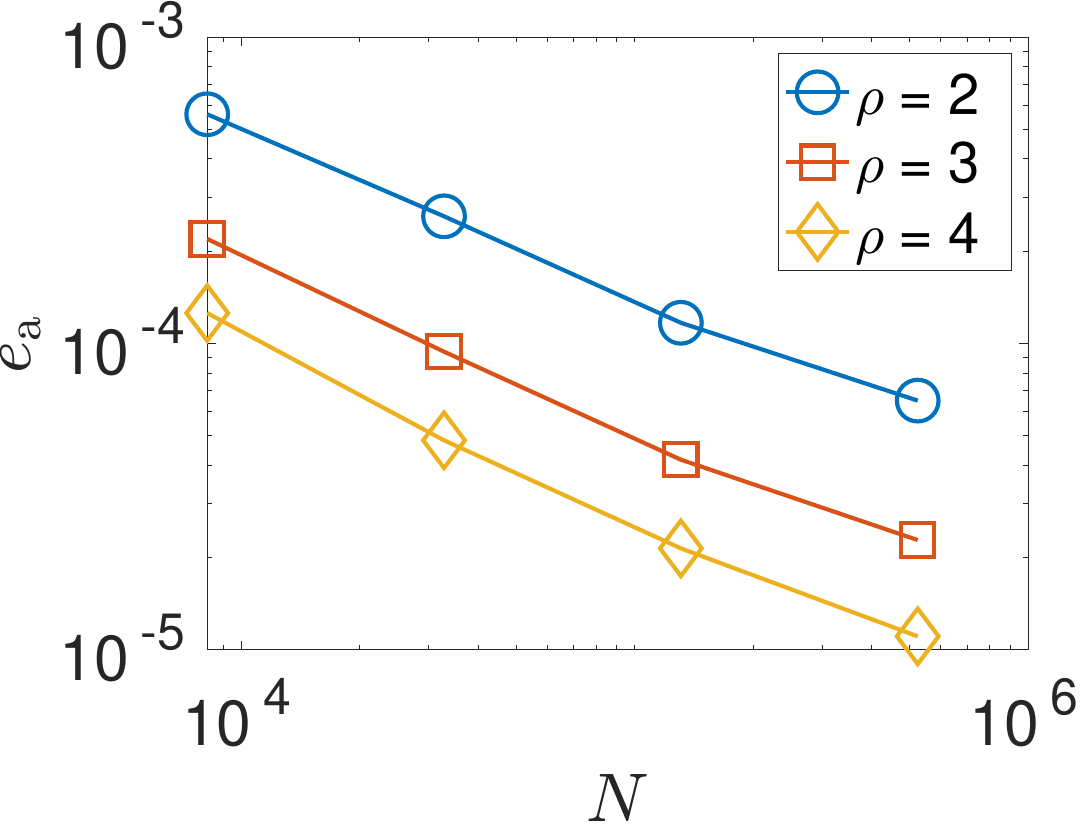}
        \caption{SLP kernel.}
        \label{fig:H2DSQ_err}
    \end{subfigure}
    \caption{
    Relative forward-evaluation errors on the quasi-uniform grid for the
    Gaussian and SLP kernels in 2D.
    Left: Gaussian kernel. Right: SLP kernel.
    }
    \label{fig:H2DQ}
\end{figure}
Figure~\ref{fig:H2DQ} shows the relative error of \(\Matrix{T}\Matrix{\hat{A}}^{\HMTB}\Matrix{S}\Matrix{u}\)
as an approximation to \(\Matrix{A}\Matrix{u}\) for different problem sizes and oversampling ratios.
Weak admissibility is used for the Gaussian kernel, while strong admissibility is used for the SLP kernel.
For each fixed \(N\), increasing \(\rho\) generally reduces the error, as the auxiliary uniform grid becomes finer.
For a fixed \(\rho\), the error also decreases as \(N\) increases, since both the quasi-uniform grid and the auxiliary uniform grid are refined simultaneously.
In particular, when \(N = 524288\) and \(\rho = 2\), the relative errors for the Gaussian and SLP kernels are approximately \(10^{-5}\) and \(10^{-4}\), respectively.
These results indicate that a modest oversampling ratio, such as \(\rho = 2\), already provides a reasonable balance between accuracy and computational cost in the present experiments.

\section{Conclusion and Future Work} \label{sec:conclusion_future_work}

In this paper, we introduce HMTB matrices and develop corresponding algorithms for their construction and forward evaluation.
When the underlying discretization has a tensor-product structure, HMTB matrices can exploit this structure to reduce the computational and storage costs of conventional \(\hierarchical\)-matrices.
For fixed \(d\) and \(p\), we show that HMTB matrices require
\(\bigO(N)\) storage and can be constructed in \(\bigO(N)\) complexity,
while their forward evaluation has complexity \(\bigO(N\log N)\).
Compared with \(\hierarchical\)-matrices, HMTB matrices reduce the dimension-dependent costs associated with the basis factors by exploiting
their tensor structure, replacing their exponential dependence on the dimension by a linear dependence.
The core-tensor contribution, however, retains an exponential dependence on the dimension.
We further extend the framework to quasi-uniform grids through sparse grid-transfer operators, thereby broadening its applicability.
Numerical experiments in two and three dimensions support the predicted
complexity scaling.
For the interpolation-based constructions tested, HMTB matrices offer
memory savings and, in several cases, shorter construction times,
while their forward-evaluation costs depend on the kernel and
admissibility condition.

The implications of these results are twofold.
First, from a methodological perspective, HMTB delivers a principled way to exploit separability and tensor structure in large-scale operator compression, contributing to the theory of hierarchical tensor approximations with provable complexity guarantees and favorable dimension dependence.
Second, from a practical standpoint, the reduced memory footprint and
competitive forward-evaluation cost make HMTB attractive for applications
where repeated operator evaluations or limited memory are critical.

There are several future directions to consider.
One potential avenue is to develop additional algebraic operations for HMTB matrices, including matrix addition, multiplication, and LU decomposition.
A second direction is to develop adaptive and algebraic construction strategies for settings without explicit tensor grids, for example using
randomized sampling.
Such approaches could enable HMTB compression of black-box operators or data-driven interaction matrices.
Furthermore, exploring the parallelization of these algorithms could enhance their efficiency and scalability, enabling them to accommodate larger problem sizes effectively.

\appendix

\section{Exploring Tensor Low-Rank Representations of Kernel Functions}\label{subsec:explore_tensor_low_rank_kernel_functions}

We numerically explore the Tucker low-rankness of interaction matrices, with a comparison to conventional low-rank matrices.
Two configurations of domains for the interaction matrices are examined: neighbor domains, which consist of two adjacent subdomains, and well-separated domains, which are characterized by two strongly admissible subdomains.
To construct a low-rank decomposition, we use the following approaches:
For both types of low-rank structures, the interpolation~(INTERP) method as discussed in Section~\ref{subsec:tucker_decomposition_from_high_dimensional_interpolation} is used to create a low-rank decomposition.
In addition, we apply SVD to compute a conventional low-rank decomposition and STHOSVD~\cite{STHOSVD} to compute a Tucker decomposition.
The accuracy of the compression is measured using the relative error with respect to the Frobenius norm.

In the \(2\)D case, we consider the following domains: \(\Omega_1 = [0, h] \times [0, h]\), \(\Omega_2 = [h, 2 h] \times [0, h]\) and \(\Omega_3 = [2 h, 3 h] \times [0, h]\) where \(h = 0.25\).
The domains \(\Omega_1\) and \(\Omega_2\) are neighbors, while \(\Omega_1\) and \(\Omega_3\) are well-separated.
Each domain is discretized with \(32\) points in each direction.
The uniform Tucker bond dimension \(p\) ranges from \(1\) to \(16\),
and the SVD rank is set to \(p^{2}\).
Similarly, the domains in \(3\)D are given by \(\Omega_1 = [0, h] \times [0, h] \times [0, h]\), \(\Omega_2 = [h, 2 h] \times [0, h] \times [0, h]\) and \(\Omega_3 = [2 h, 3 h] \times [0, h] \times [0, h]\) for \(h = 0.25\).
Each domain is discretized with \(16\) points in each direction, and \(p\) ranges from \(1\) to \(8\), with the SVD rank set to \(p^3\).

\clearpage

\begin{figure}[p]
    \centering

    \begin{subfigure}[t]{0.45\textwidth}
        \centering
        \includegraphics[width=.8\linewidth]{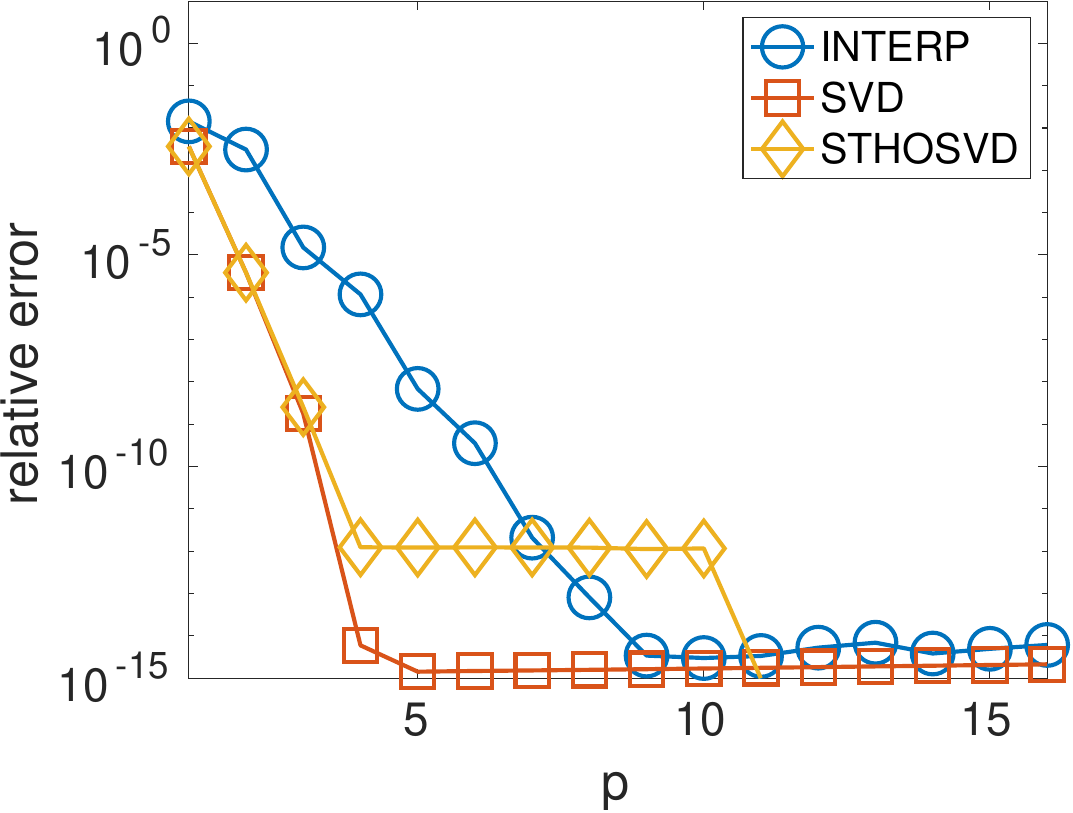}
        \includegraphics[width=.8\linewidth]{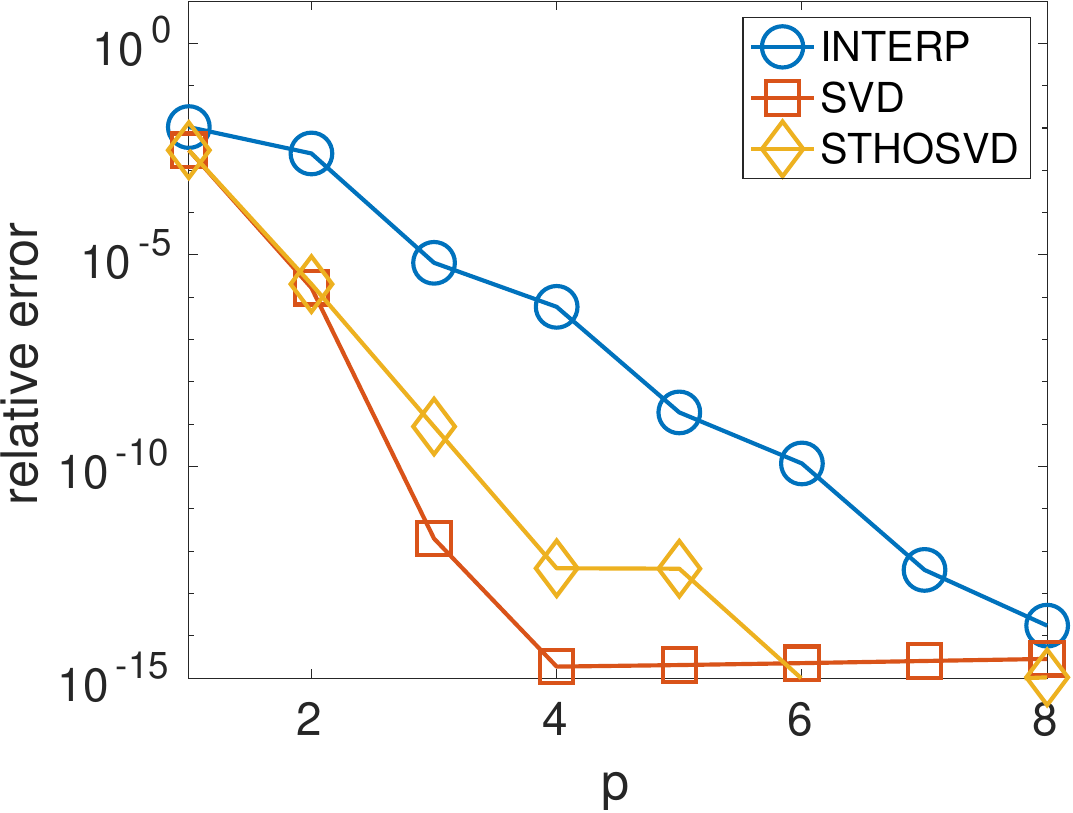}
        \caption{Gaussian kernel and neighbor domains.}
        \label{fig:NBR_Gaussian}
    \end{subfigure}
    \begin{subfigure}[t]{0.45\textwidth}
        \centering
        \includegraphics[width=.8\linewidth]{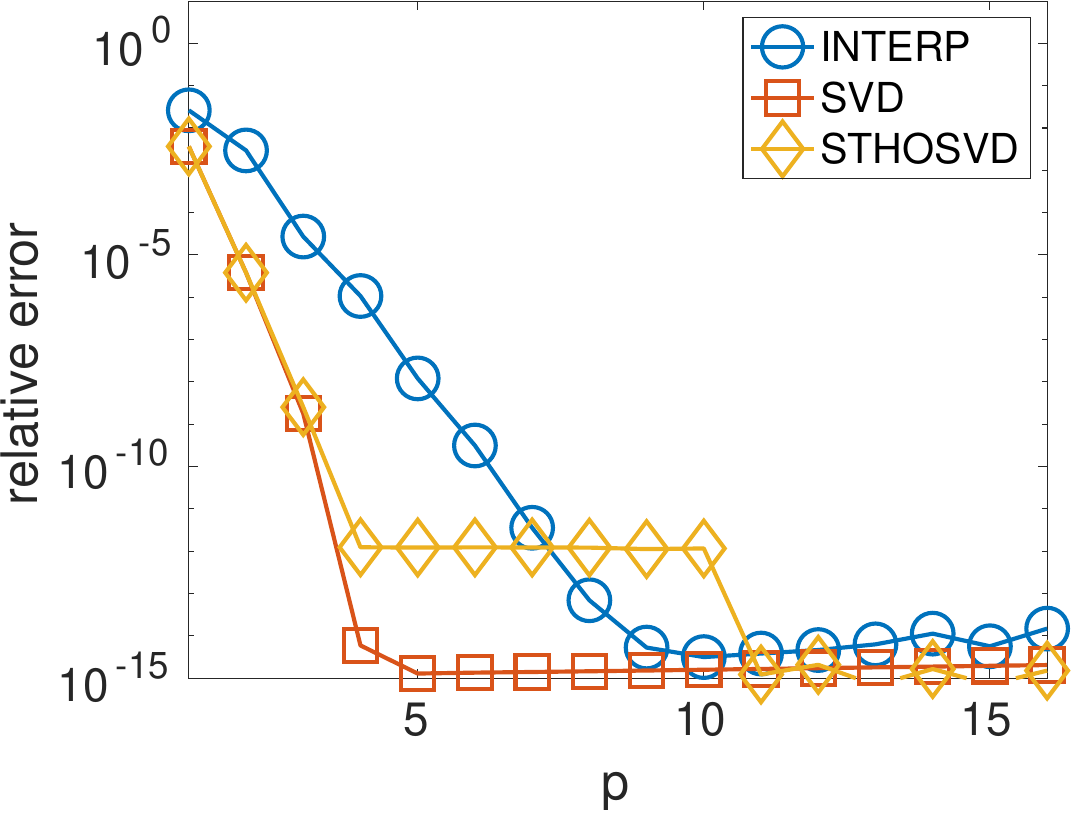}
        \includegraphics[width=.8\linewidth]{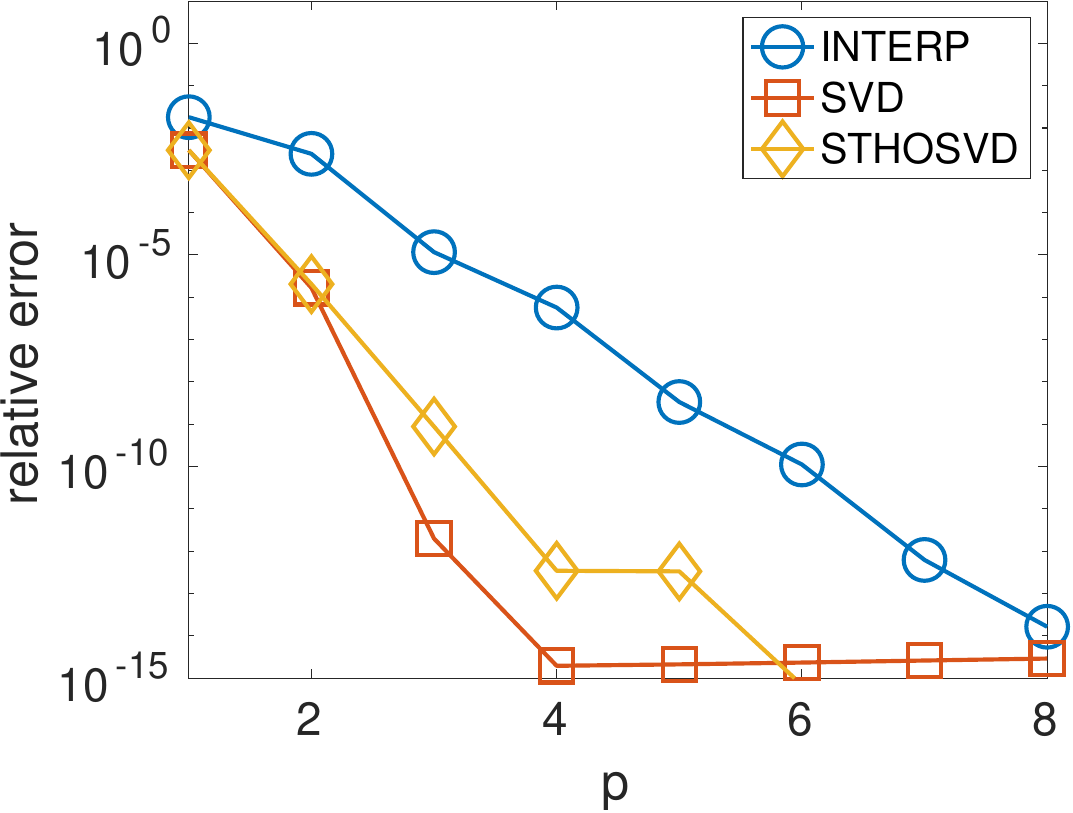}
        \caption{Gaussian kernel and well-separated domains.}
        \label{fig:WS_Gaussian}
    \end{subfigure}

    \begin{subfigure}[t]{0.45\textwidth}
        \centering
        \includegraphics[width=.8\linewidth]{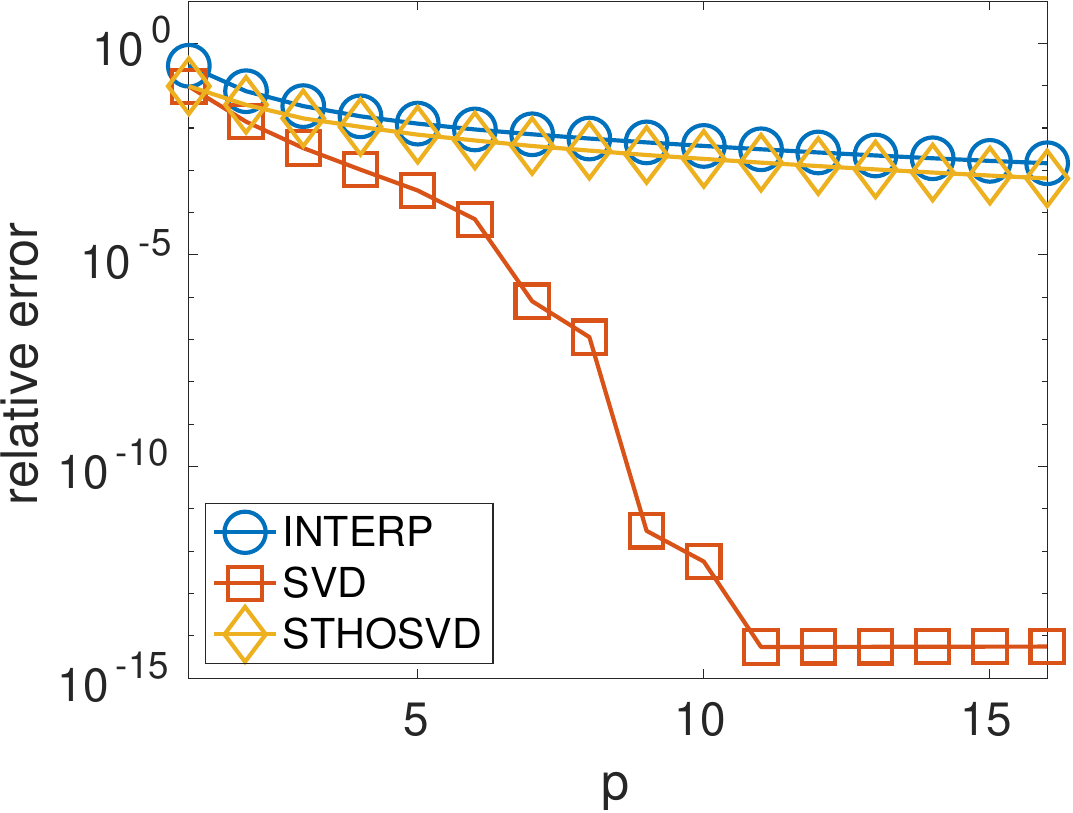}
        \includegraphics[width=.8\linewidth]{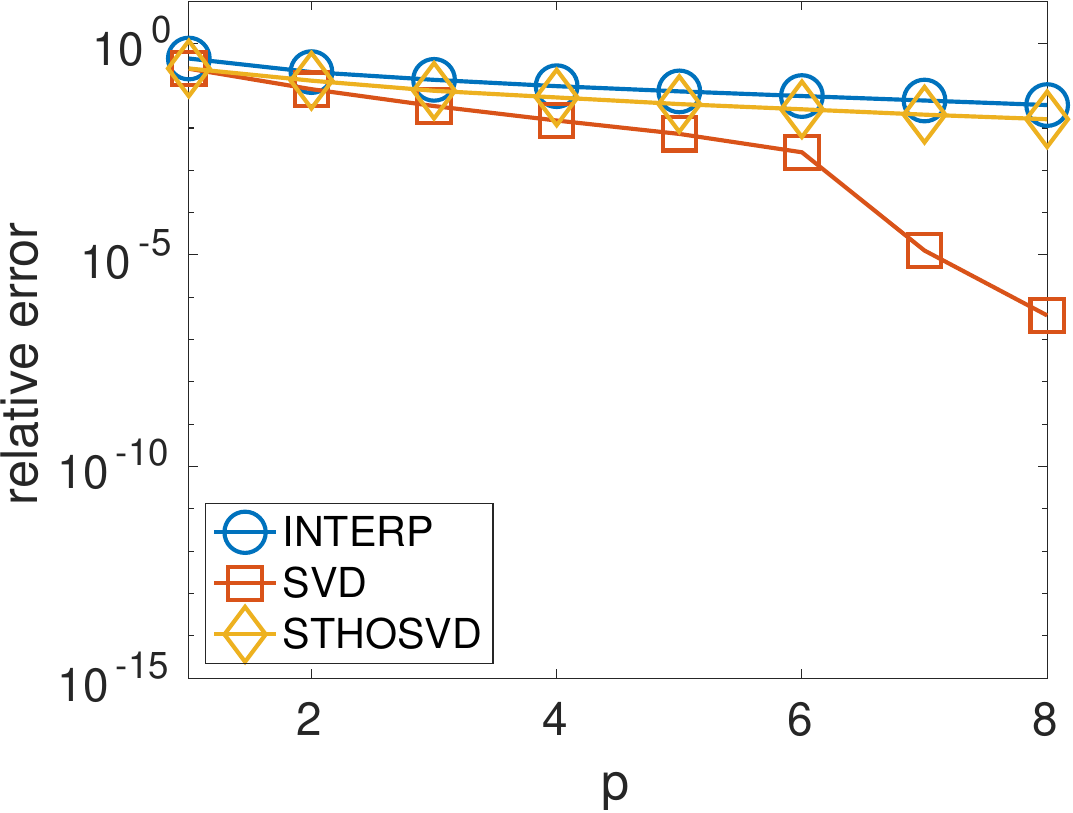}
        \caption{SLP kernel and neighbor domains.}
        \label{fig:NBR_SLP}
    \end{subfigure}
    \begin{subfigure}[t]{0.45\textwidth}
        \centering
        \includegraphics[width=.8\linewidth]{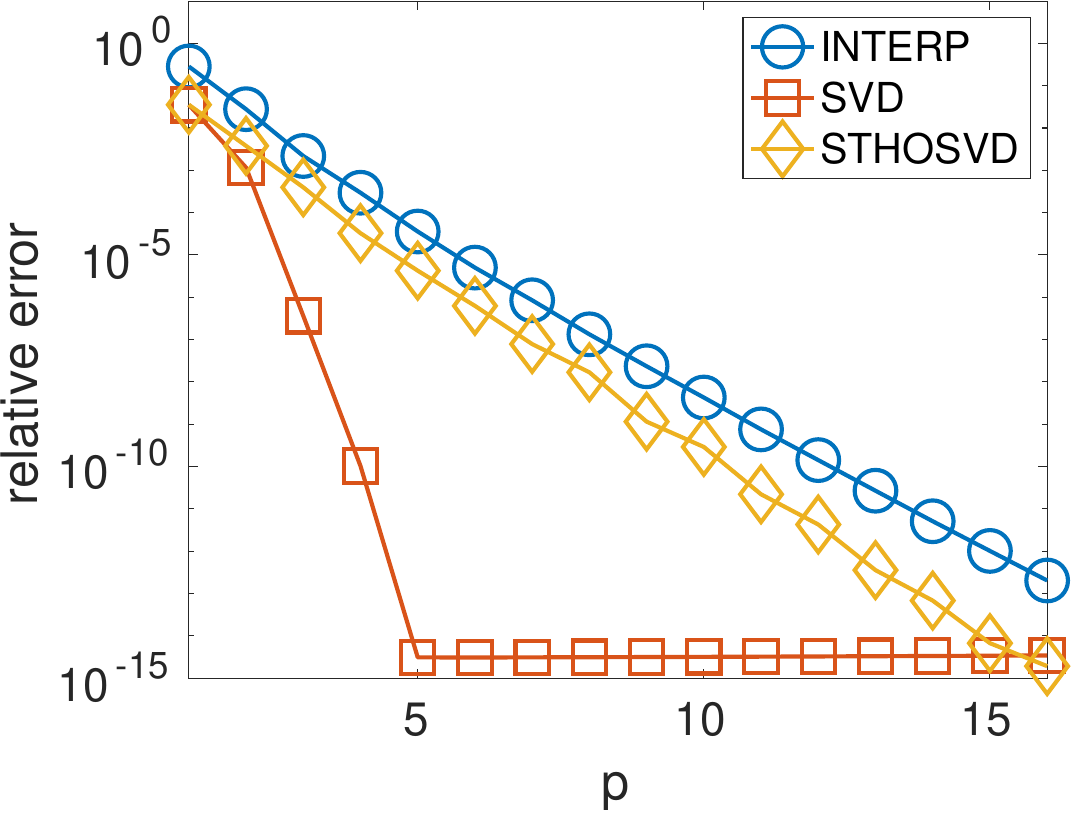}
        \includegraphics[width=.8\linewidth]{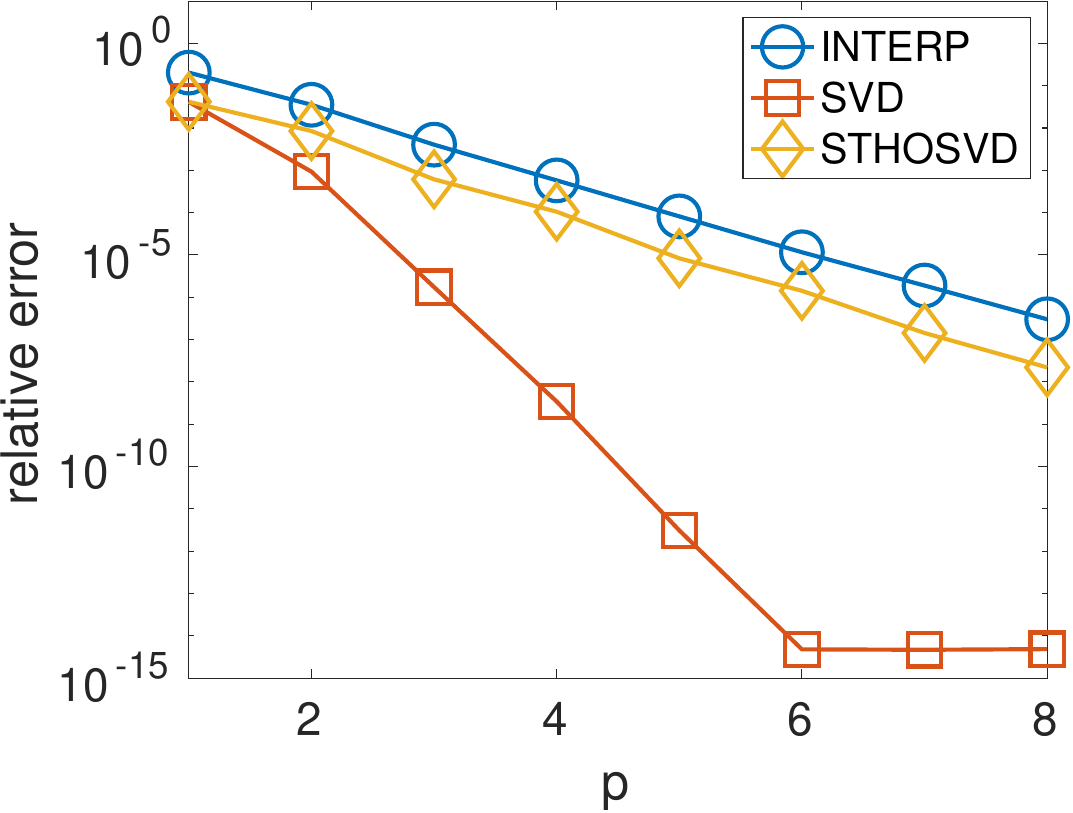}
        \caption{SLP kernel and well-separated domains.}
        \label{fig:WS_SLP}
    \end{subfigure}

    \caption{Relative approximation errors of the Gaussian, SLP, and Helmholtz kernels.
    The size of the interaction matrix is \(32^2 \times 32^2\) for \(2\)D
    problems~(top) and \(16^3 \times 16^3\) for \(3\)D problems~(bottom).}
    \label{fig:relative_error_constructions_low_rank}
\end{figure}
\begin{figure}[p]
    \ContinuedFloat
    \centering

    \begin{subfigure}[t]{0.45\textwidth}
        \centering
        \includegraphics[width=.8\linewidth]{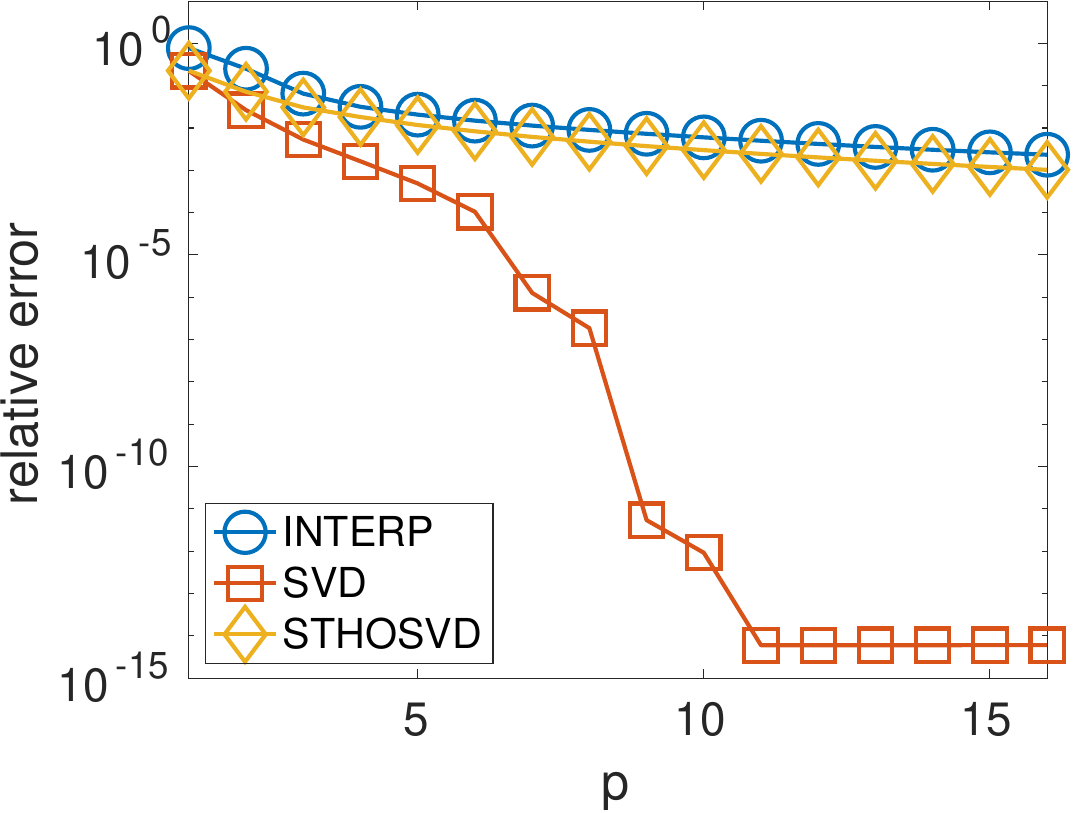}
        \includegraphics[width=.8\linewidth]{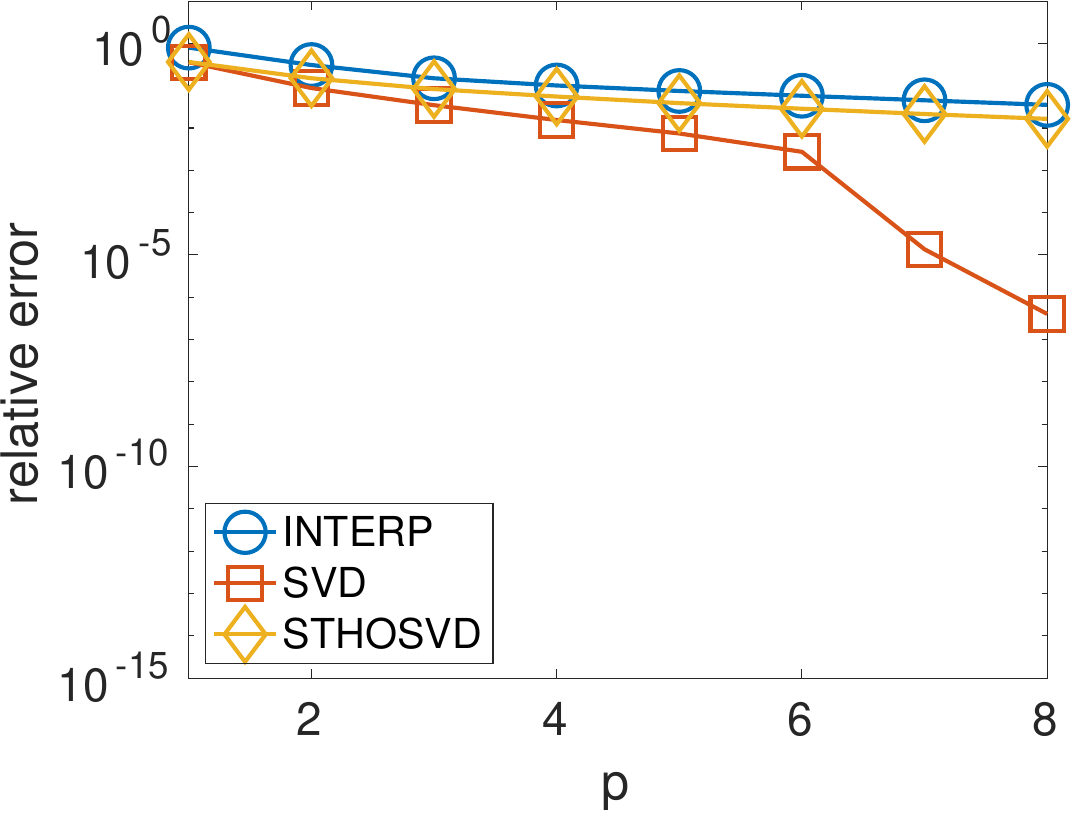}
        \caption{Helmholtz kernel and neighbor domains.}
        \label{fig:NBR_Helm}
    \end{subfigure}
    \begin{subfigure}[t]{0.45\textwidth}
        \centering
        \includegraphics[width=.8\linewidth]{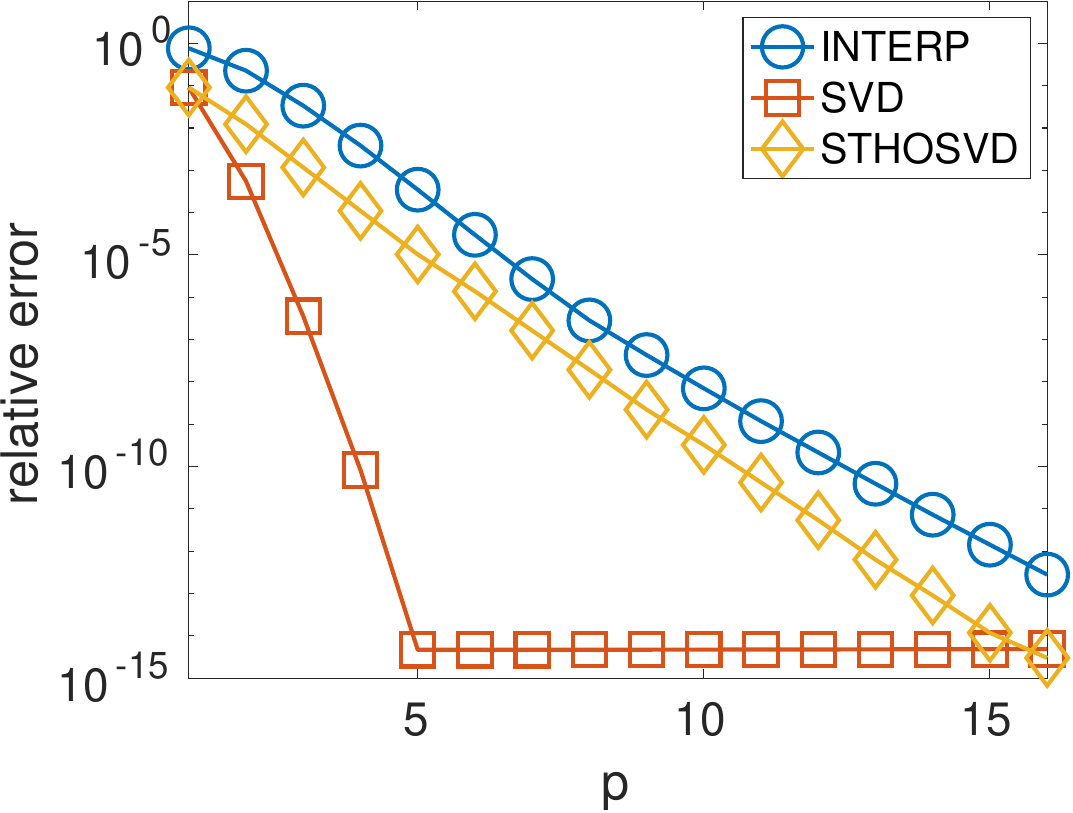}
        \includegraphics[width=.8\linewidth]{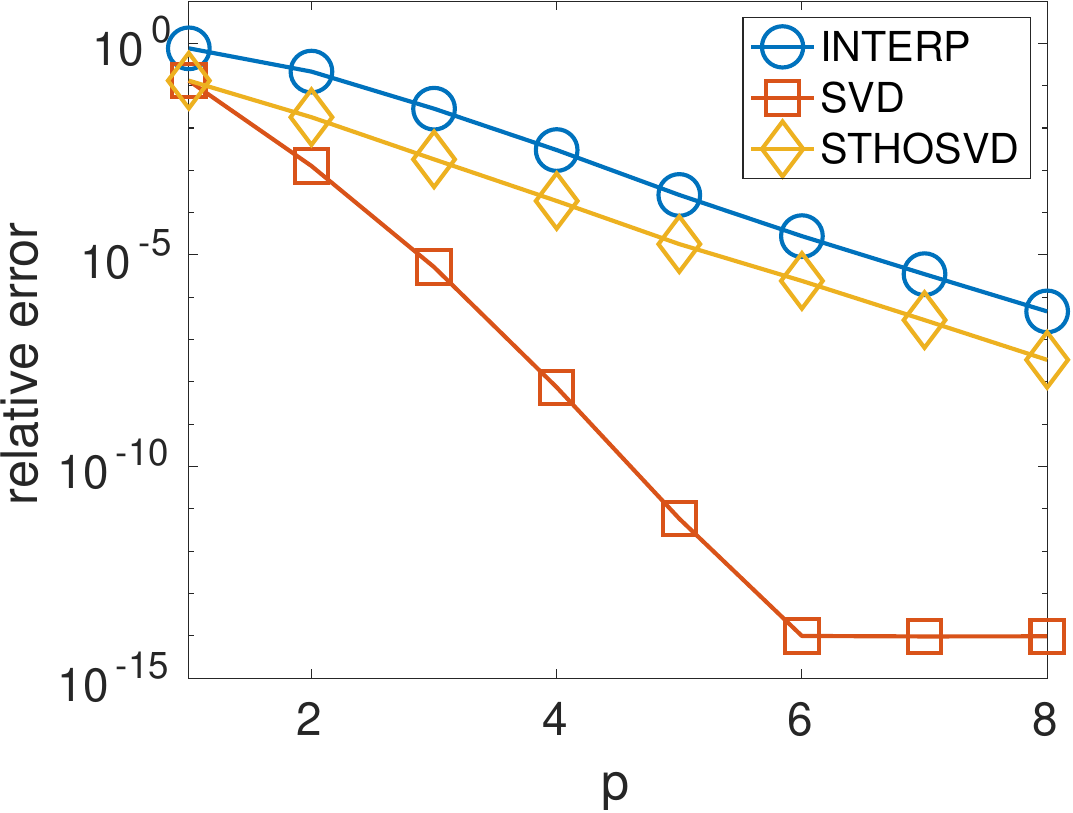}
        \caption{Helmholtz kernel and well-separated domains.}
        \label{fig:WS_Helm}
    \end{subfigure}

    \caption[]{Relative approximation errors of the Gaussian, SLP, and Helmholtz kernels
    (continued).}
\end{figure}

\clearpage

The results of different low-rank decompositions are plotted in Figure~\ref{fig:relative_error_constructions_low_rank}.
For well-separated domains (Figure~\ref{fig:WS_Gaussian},~\ref{fig:WS_SLP} and~\ref{fig:WS_Helm}), all three methods exhibit a rapid decay in error as \(p\) increases, regardless of the kernel.
In this scenario, the interpolation method is sufficient to achieve an accurate low-rank decomposition.
However, when two domains are neighbors~(as depicted in Figure~\ref{fig:NBR_Gaussian},~\ref{fig:NBR_SLP} and~\ref{fig:NBR_Helm}), the decay rates of the relative error exhibit dependence on the kernels.
For the Gaussian kernel, both low-rank structures still work well, with the interpolation-based construction yielding satisfactory results.
In contrast, for the SLP and Helmholtz kernels, the Tucker low-rank format proves unsuitable.
Even with the application of STHOSVD, the relative error decreases
very slowly as the uniform Tucker bond dimension \(p\) increases.
Accordingly, in Section~\ref{sec:numerical_results} we use strong admissibility for the SLP and
Helmholtz kernels and weak admissibility for the Gaussian kernel.

\section*{Acknowledgements}

This work was supported in part by the National Natural Science Foundation of China under Grant Nos. 12271109 and 12526211; by the Shanghai Pilot Program for Basic Research-Fudan University under Grant No. 21TQ1400100 (22TQ017); by the Scientific Research Innovation Capability Support Project for Young Faculty under Grant No. SRICSPYF-ZY2025159; and by the Xuemin Institute of Advanced Studies, Fudan University.
We thank the two anonymous reviewers for their valuable comments and suggestions.

\bibliographystyle{unsrt}
\bibliography{ref}

\end{document}